\documentclass{amsart}

\usepackage{hyperref}
\usepackage{amsmath}

\usepackage{amssymb}
\usepackage{mathrsfs}
\usepackage{amscd}
\usepackage{graphicx}
\usepackage{mathdots}
\usepackage{gensymb}
\usepackage{epstopdf}
 \usepackage{todonotes}
 \usepackage{youngtab}
 \usepackage{wrapfig}
 \usepackage{bbm}
 \usepackage{bm}
 \usepackage{verbatim}
 \usepackage[margin=1.1in]{geometry}

 \usepackage[nocompress]{cite}

\usepackage{epsfig}       
\usepackage{epic,eepic}        

\newtheorem{thm}{Theorem}[section]
\newtheorem{prop}[thm]{Proposition}
\newtheorem{proposition}[thm]{Proposition}
\newtheorem{lem}[thm]{Lemma}
\newtheorem{cor}[thm]{Corollary}

\newtheorem{claim}[thm]{Claim}

\theoremstyle{definition}
\newtheorem{definition}[thm]{Definition}
\newtheorem{example}[thm]{Example}

\theoremstyle{remark}
\newtheorem{remark}[thm]{Remark}

\numberwithin{equation}{section}

\newcommand{\Tr}{\operatorname{Tr}}

\newcommand{\Z}{\mathbb{Z}}

\newcommand{\C}{\mathbb{C}}
\newcommand{\R}{\mathbb{R}}
\newcommand{\diag}{{\rm diag}}
\newcommand{\mult}{{\rm mult}}
\newcommand{\fC}{\mathfrak{C}}

\newcommand{\gog}{\mathfrak{g}}

\newcommand{\tot}{\mathfrak{t}}

\newcommand{\ch}{\mathrm{ch}}

\newcommand{\bQ}{\mathbb{Q}}

\newcommand{\bmla}{{\bm \lambda}}
\newcommand{\bmeta}{{\bm \eta}}
\newcommand{\bmmu}{{\bm \mu}}
\newcommand{\bmrho}{{\bm \rho}}

\newcommand{\bmy}{{\bm y}}

\def\U{{\rm U}}

\newcommand{\oo}{{\rm o}}

\newcommand{\al}{\alpha}

\newcommand{\cP}{{\mathcal P}}
\newcommand{\cC}{{\mathcal C}}

\newcommand{\cF}{{\mathcal F}}

\newcommand{\cI}{{\mathcal I}}

\newcommand{\rd}{{\rm d}}
\newcommand{\bB}{{\mathbb B}}

\newcommand{\bP}{{\mathbb P}}
\newcommand{\del}{{\partial}}

\newcommand{\unif}{{\rm{unif}}}
\newcommand{\OO}{{\rm{O}}}

\newcommand{\cal}{\mathcal }
\newcommand{\fc}{{\mathfrak c}}
\newcommand{\fR}{{\mathfrak R}}
\newcommand{\fK}{{\mathfrak K}}

\newcommand{\cM}{{\mathcal M}}
\newcommand{\cA}{{\mathcal A}}

\newcommand{\bR}{{\mathbb R}}

\newcommand{\supp}{{\rm supp}}
\newcommand{\sfa}{{\mathsf a}}
\newcommand{\sfb}{{\mathsf b}}

\begin{document}

\title[Asymptotics of Generalized Bessel Functions via Large Deviations]{Asymptotics of Generalized Bessel Functions and Weight Multiplicities via Large Deviations of Radial Dunkl Processes}

\author{Jiaoyang Huang}
\address{Wharton Department of Statistics and Data Science, University of Pennsylvania}
\email{huangjy@wharton.upenn.edu}

 \author{Colin McSwiggen}
\address{Courant Institute of Mathematical Sciences, New York University}
\email{csm482@nyu.edu}

\maketitle

\begin{abstract}
This paper studies the asymptotic behavior of several central objects in Dunkl theory as the dimension of the underlying space grows large. Our starting point is the observation that a recent result from the random matrix theory literature implies a large deviations principle for the hydrodynamic limit of radial Dunkl processes. Using this fact, we prove a variational formula for the large-$N$ asymptotics of generalized Bessel functions, as well as a large deviations principle for the more general family of radial Heckman--Opdam processes.  As an application, we prove a theorem on the asymptotic behavior of weight multiplicities of irreducible representations of compact or complex simple Lie algebras in the limit of large rank.  The theorems in this paper generalize several known results describing analogous asymptotics for Dyson Brownian motion, spherical matrix integrals, and Kostka numbers.
\end{abstract}

 \tableofcontents

\section{Introduction}

One of the most pressing challenges in the modern analysis of special functions is the study of large-$N$ asymptotics.  Classical asymptotic analysis asks about the behavior of a given function at distant or singular points of its domain, such as its limit along a ray in space or at small or large times.  In contrast, large-$N$ analysis is concerned with the behavior of a {\it family} of functions in an increasing number of variables, and asks about their limits as the number of variables $N$ -- that is, the dimension of the domain -- goes to infinity.

Large-$N$ limits of this type admit a wide range of interpretations and appear in many different fields.  They have played a major role in theoretical high-energy physics at least since the 1970's, when 't Hooft \cite{Gt} introduced large-$N$ Yang--Mills theory as an approximate model of the strong nuclear force.  To this day, much of the mathematical work on Yang--Mills theories is concerned with large-$N$ phenomena \cite{Chatterjee:YM}.

Powerful techniques for rigorous large-$N$ analysis have come out of the random matrix theory literature, where multivariable special functions, in particular spherical integrals, appear in expressions for the joint spectral densities of many important matrix ensembles \cite{AGZ, AGsurvey}.  In this setting, the number of variables $N$ corresponds to the size of the matrices, and large-$N$ analysis is necessary to understand asymptotic properties of the spectra.  Two main approaches have emerged to the problem of large-$N$ limits of special functions in random matrix theory.  The \emph{stochastic process approach} exploits the fact that many important multivariable special functions are closely related to transition kernels for certain probabilistic interacting particle systems, where the number $N$ corresponds to the number of particles.  Information about the large-$N$ limits of the functions can then be extracted by studying the hydrodynamic behavior of the particle systems, which is usually of independent interest; see e.g. \cite{AM, GZ, GZ-addendum, guionnet2021large, belinschi2022large}.  There is also a \emph{combinatorial approach} based on the insight that the functions to be analyzed can often be interpreted as generating functions for some family of combinatorial objects \cite{GGN, Novak:int, Novak:osc}.

The main goal of this paper is to lay out a unifying framework for the stochastic process approach, based on Dunkl theory.  Dunkl theory is an extremely versatile special functions theory whose objects vastly generalize many classical special functions such as hypergeometric and Bessel functions.  It grew out of the theory of harmonic analysis on Lie algebras and symmetric spaces developed by Harish-Chandra, Helgason and others starting in the 1950's \cite{HC, HS, SHDS, SH}, with decisive contributions by Dunkl \cite{Dunkl1, Dunkl2, Dunkl3, Dunkl4, Dunkl5}, Heckman and Opdam \cite{RSHF1, RSHF2, RSHF3, RSHF4}, and Cherednik \cite{Ch1, Ch2, Ch-DAHA}.

Here we concern ourselves with two main variants of Dunkl theory: \emph{rational} and \emph{hyperbolic}. The hyperbolic theory is in a sense more general, as all of the objects studied in the rational theory can be recovered as degenerations of corresponding objects in the hyperbolic theory.  Rational Dunkl theory leads to the construction of the \emph{generalized Bessel functions}, which are closely related to the transition kernels of a family of stochastic processes called \emph{radial Dunkl processes}.  The analogous objects in the hyperbolic theory are the \emph{Heckman--Opdam hypergeometric functions} and \emph{radial Heckman--Opdam processes}.  We review the basic constructions of rational and hyperbolic Dunkl theory in Section \ref{sec:dunkl-prelims} below.  For a more complete introduction, we refer the the reader to the notes by Anker \cite{AnkerDunklNotes} and the references therein.

As we explain below, the spherical integrals studied in random matrix theory (such as Harish-Chandra and Itzykson--Zuber integrals, as well as their rectangular variants) are special cases of generalized Bessel functions, and the stochastic processes used to study them (such as Dyson Brownian motion and the Dyson Bessel process) are special cases of radial Dunkl processes.  By carrying out the stochastic process approach at the level of Dunkl theory, we can therefore unify many known results on large-$N$ limits of spherical integrals, while simultaneously extending these results to many other functions of interest in fields beyond random matrix theory, such as characters of arbitrary compact Lie groups and spherical functions on arbitrary Euclidean symmetric spaces.  We expect these results to have applications in several domains; as a first step, in the second half of this paper, we develop an application to asymptotic representation theory.  

Several existing papers have studied the asymptotics of radial Dunkl and Heckman--Opdam processes in the limits of large time or large parameters \cite{AKM1, AKM2, Andraus-Miyashita, Andraus-Voit-CLT, SchapiraHOProc, Voit-CLT}.  However, in these papers the dimension of the domain (equivalently, the number of particles) is fixed.  The hydrodynamic limits that we consider here are of a fundamentally different character.

The present work deals only with the leading-order behavior of special functions and related stochastic processes at large $N$, but of course one would like to know more.  Beyond clear next steps such as studying higher-order corrections, it is natural to wonder how much of the machinery of Dunkl theory can be recovered after the large-$N$ limit: whether, for example, one could talk meaningfully about infinite-rank analogues of the Dunkl or Cherednik operators, which could perhaps be related to known constructions in non-commutative probability, or understood in terms of all-order large-$N$ expansions as recently developed in \cite{Novak:int, Novak:osc}.  Such questions are far out of the scope of this paper, but they are fun topics for speculation, and we hope that enterprising readers will be tempted to think about them in the future. 

\subsection{Summary of contributions}

The starting point of our investigation is the observation that the Dyson Bessel process, which was studied by Guionnet and the first author in \cite{guionnet2021large}, is a reparametrization of the radial Dunkl process of type $B$ (or equivalently of type $C$).  Since this process includes all of the types of interactions that can occur in any of the radial Dunkl processes for the other classical root systems, the large deviations principle shown in \cite{guionnet2021large} in fact implies a large deviations principle for the hydrodynamic limit of \emph{any} radial Dunkl process, provided the parameter values are such that the process does not hit the boundary of its domain in finite time.  We recall the large deviations principle of \cite{guionnet2021large} in Theorem \ref{main2}.  The first main result of this paper, Theorem \ref{main1}, uses this large deviations principle to obtain a formula, in terms of a solution to a variational problem, for the leading-order contribution to the large-$N$ limit of generalized Bessel functions.  This theorem extends results on spherical integrals that were shown in \cite{GZ, GZ-addendum, guionnet2021large}, and can be stated informally as follows.

\begin{thm}[Informal version of Theorem \ref{main1}]
\label{main1-informal}
For each $N = 2, 3, \hdots,$ fix points
\begin{align*}
x &= x^{(N)} = (x^{(N)}_1, x^{(N)}_2, \cdots, x^{(N)}_N ), \\
y &= y^{(N)} = (y^{(N)}_1, y^{(N)}_2, \cdots, y^{(N)}_N ) \in \R^N,
\end{align*}
such that their scaled, symmetrized empirical measures, as defined below in (\ref{eqn:ssem}), converge weakly as $N \to \infty$ to two probability measures $\widehat \nu_A$ and $\widehat \nu_B$ respectively.  Under some mild growth assumptions on $x,$ $y$ and the multiplicity parameters $k = k^{(N)}$, and a regularity assumption on the measures $\widehat \nu_A$ and $\widehat \nu_B$, the following limit exists:
\[
\lim_{N \to \infty}\frac{1}{N^{2}}\log J_{k,x}(y)= I(\widehat \nu_{A},\widehat\nu_{B}), 
\]
where $J_{k,x}(y)$, defined below in Definition \ref{def:GBF}, is the generalized Bessel function associated with the $B_N$, $C_N$, or $D_N$ root system, and the functional $I$ is given by the solution to an explicit variational problem depending only on the measures $\widehat \nu_A$ and $\widehat \nu_B$ and on the asymptotic behavior of $k$. 
\end{thm}
We state the theorem in terms of generalized Bessel functions of types $B$, $C$ and $D$, because the corresponding statement for type $A$ takes a slightly different form and is not difficult to infer from existing results in the literature, as we explain in Remark \ref{rem:typeAlimit}.  We give the result for type $A$ in Theorem \ref{main1-A}.

Next, we turn to the hyperbolic theory.  In Proposition \ref{p:changem} and Corollary \ref{cor:HOP-asymp}, we establish a large deviations principle for radial Heckman--Opdam processes, which we state in terms of a modified version of the Dyson Bessel process.  As we explain, the hyperbolic theory presents some difficulties that are not present in the rational case, so that the large deviations principle does not immediately yield a formula for the large-$N$ limit of Heckman--Opdam hypergeometric functions.  However, we have chosen to include these results both because we hope that they will be a step towards such a formula and because the hydrodynamic behavior of the radial process in the hyperbolic setting is of independent interest.

The final sections of the paper develop applications of the above ideas to asymptotic representation theory.  In Lemmas \ref{lem:monom-asymp} and \ref{lem:logchi}, respectively, we prove formulae for the large-$N$ limits of the monomial Weyl group--invariant polynomials (which generalize the monomial symmetric polynomials in type $A$) and of the characters of compact or complex simple Lie algebras of type $B$, $C$ or $D$.  These results enable our main application, Theorem \ref{t:Kostka}, which gives large deviations asymptotics for weight multiplicities of irreducible representations of compact or complex simple Lie algebras.  This theorem extends an existing result on Kostka numbers, shown in \cite{belinschi2022large}, to weight multiplicities for Lie algebras of the remaining three classical families.  Theorem \ref{t:Kostka} can be stated informally as follows; we refer the reader to Section \ref{sec:mult-prelims} below for a review of definitions from representation theory.

\begin{thm}[Informal version of Theorem \ref{t:Kostka}] \label{t:Kostka-informal}
Let $(\gog_N)_{N=2}^\infty$ be a sequence of compact simple Lie algebras of root system type $B_N$, $C_N$ or $D_N$. For each $N$, let $\bmla_N$ be a deterministic highest weight of $\gog_N$, and suppose that the scaled, shifted empirical measures $m[\bmla_N]$, as defined in \eqref{e:defm}, converge weakly to a probability measure $m_{\bmla}$ as $N \to \infty$. Let $\mu$ be another probability measure on $[0,\infty)$. Under a mild growth assumption on the weights $\bmla_N$ and a suitability assumption on the measure $\mu$, the weight multiplicities $\mult_{\bmla_N}(\bmeta_N)$, defined in \eqref{eqn:weight-decomp}, satisfy
\begin{align*}
\nonumber \lim_{\delta\rightarrow 0} & \limsup_{N\rightarrow \infty} \frac{1}{N^2}\log \sup_{m[\bmeta_N]\in \bB_\delta(\mu)} \mult_{\bmla_N}(\bmeta_N)\\
&=\lim_{\delta\rightarrow 0}\liminf_{N\rightarrow \infty}\frac{1}{N^2}\log \sup_{m[\bmeta_N]\in \bB_\delta(\mu)} \mult_{\bmla_N}(\bmeta_N) =-\cI(\mu),
\end{align*}
 where $ \bB_{\delta}(\mu)$ is the open ball around $\mu$ with radius $\delta$ in Wasserstein distance in a space of suitable measures, and the functional $\cI$ is given by the solution to an explicit variational problem.  Moreover, $\cI$ is equal to $+\infty$ unless the measure $\mu$ satisfies an infinite-dimensional analogue of the Schur--Horn inequalities.
\end{thm}

\subsection{Organization of the paper}

In Section \ref{sec:dunkl-prelims}, we review basic definitions and constructions from rational and hyperbolic Dunkl theory.  We define the Dunkl and Cherednik operators, the radial Dunkl and Heckman--Opdam processes, the generalized Bessel functions, and the Heckman--Opdam hypergeometric functions.  We also give examples to illustrate how other functions of interest, such as many spherical matrix integrals, can be obtained from these as specializations.

In Section \ref{sec:LDPs} we relate radial Dunkl processes to the Dyson Bessel process and recall the large deviations principle shown in \cite{guionnet2021large}, which we use to prove the variational formula for the large-$N$ limit of generalized Bessel functions.  We then prove a large deviations principle for the modified Dyson Bessel process (equivalently, for radial Heckman--Opdam processes).

The remaining sections deal with applications to representation theory. Section \ref{sec:mult-prelims} reviews preliminaries on the representation theory of compact or complex Lie algebras, including basic character theory and the weight space decomposition.  We also prove lemmas on the large-$N$ limits of monomial Weyl group--invariant polynomials and characters of irreducible representations.

Section \ref{s:derSI} recalls some results on derivatives of spherical integrals for use in Section \ref{sec:weight-asymp}, where we prove a theorem on the asymptotic behavior of weight multiplicities of irreducible representations of complex or compact simple Lie algebras.

In Appendix \ref{sec:proof-appendix}, we give the proof of a technical result stated in Section \ref{sec:weight-asymp}.

\section{Preliminaries on Dunkl theory}
\label{sec:dunkl-prelims}

In this section, we recall some basic definitions from rational and hyperbolic Dunkl theory, largely following the conventions of Anker \cite{AnkerDunklNotes}, to which we refer the reader for further details.  We also define the radial Dunkl and Heckman--Opdam processes and give expressions for their respective transition kernels in terms of generalized Bessel functions and Heckman--Opdam hypergeometric functions, following R\"osler \cite{RoslerHermite} and Schapira \cite{SchapiraHGF, SchapiraHOProc}.

The rational theory is simpler than the hyperbolic theory, and in fact it can be recovered from the hyperbolic theory as a limiting case.  We therefore start by discussing the rational setting, and afterwards we introduce the corresponding objects from the hyperbolic theory in the same order.

\subsection{Root systems: definitions and examples}
\label{sec:root-systems}

Let $V$ be a Euclidean space with inner product $\langle \cdot, \cdot \rangle$. For $\alpha \in V$, let $s_\alpha$ be the reflection through the hyperplane $\{x \in V : \langle \alpha, x \rangle = 0 \}$, that is,
\[
s_\alpha x = x - 2 \frac{\langle \alpha, x \rangle}{\langle \alpha, \alpha \rangle} \alpha, \qquad x \in V.
\]

A finite set of non-zero vectors $\Phi \subset V$ is called a (crystallographic) {\it root system} if it satisfies the following properties:
\begin{enumerate}
    \item $s_\alpha(\Phi) = \Phi$ for all $\alpha \in \Phi$, and
    \item $2 \frac{\langle \alpha, \beta \rangle}{\langle \alpha, \alpha \rangle} \in \Z$ for all $\alpha, \beta \in \Phi$.
\end{enumerate}
The dimension of the span of $\Phi$ is called its {\it rank}. Here we consider only root systems with full rank, i.e., we assume that $\Phi$ spans $V$.

Since $s_\alpha \alpha = -\alpha$, the first property above implies that $-\alpha \in \Phi$ for all $\alpha \in \Phi$.  The root system $\Phi$ is said to be {\it reduced} if
$\Phi \cap \R \alpha = \{\pm \alpha\}$ for all $\alpha \in \Phi,$ and is said to be {\it irreducible} if it cannot be decomposed as a non-trivial union of root systems in orthogonal complementary subspaces of $V$.

The {\it Weyl group} of $\Phi$ is the group $W$ generated by the operators $s_\alpha$, $\alpha \in \Phi$.  A choice of {\it positive roots} is a subset $\Phi^+ \subset \Phi$ such that, for all $\alpha \in \Phi$, either $\alpha \in \Phi^+$ or $-\alpha \in \Phi^+$ but not both.  Given a choice of positive roots, the (closed) {\it positive Weyl chamber} is the cone
\[
\mathcal{C}^+ = \big \{ x \in V : \langle \alpha, x \rangle \ge 0 \quad \forall \ \alpha \in \Phi^+ \big \}.
\]
A {\it multiplicity parameter} is a function $k : \Phi \to \C$ that is constant on Weyl orbits, i.e., such that $k_{w(\alpha)} = k_\alpha$ for all $\alpha \in \Phi$ and $w \in W$, where we write the argument of $k$ in subscript. 

We will usually take $\Phi$ to be one of the four classical root systems $A_{N-1}$, $B_N$, $C_N$, or $D_N$, which are all reduced, or the non-reduced root system $BC_N$.  Since our results concern limits where $N$ tends to infinity, it is sufficient to consider these cases, as they are the only irreducible root systems for which the rank may be arbitrarily large. When considering sequences of root systems in spaces of increasing dimension, we will often add subscripts to various notation for clarity, writing $\Phi^+_N$, $\mathcal{C}^+_N$, etc.

For ease of reference, we briefly recall the standard embeddings of these five root systems in Euclidean space.  Let $e_1, \hdots, e_N$ be the standard orthonormal basis of $\R^N$.  When $\Phi = A_{N-1}$, we take $$V = \Big\{ x \in \R^N \, : \sum_i x_i = 0 \Big\}.$$  In all other cases, we take $V = \R^N$.

\begin{itemize}
    \item The $A_{N-1}$ root system consists of the single Weyl orbit $\{ \pm(e_j - e_k), \ 1 \le j < k \le N \}$. As positive roots we may take $\{ e_j - e_k, \ 1 \le j < k \le N \}$.

    \medskip

    \item The $B_N$ root system consists of the two Weyl orbits $$\{ \pm e_j, \ j = 1, \hdots, N \} \qquad \textrm{and}  \qquad \{ \pm e_j \pm e_k, \ 1 \le j < k \le N \}.$$ As positive roots we may take $\{ e_j, \ j = 1, \hdots, N \} \cup \{ e_j \pm e_k, \ 1 \le j < k \le N \}.$

    \medskip

    \item The $C_N$ root system consists of the two Weyl orbits $$\{ \pm 2 e_j, \ j = 1, \hdots, N \} \qquad \textrm{and}  \qquad \{ \pm e_j \pm e_k, \ 1 \le j < k \le N \}.$$ As positive roots we may take $\{ 2e_j, \ j = 1, \hdots, N \} \cup \{ e_j \pm e_k, \ 1 \le j < k \le N \}.$

    \medskip

    \item The $D_{N}$ root system consists of the single Weyl orbit $\{ \pm e_j \pm e_k, \ 1 \le j < k \le N \}$. As positive roots we may take $\{ e_j \pm e_k, \ 1 \le j < k \le N \}$.

    \medskip

    \item The $BC_N$ root system consists of the three Weyl orbits $$\{ \pm e_j, \ j = 1, \hdots, N \}, \quad \{ \pm 2 e_j, \ j = 1, \hdots, N \}, \quad \textrm{and} \quad \{ \pm e_j \pm e_k, \ 1 \le j < k \le N \}.$$ As positive roots we may take $\{ e_j, \ j = 1, \hdots, N \} \cup \{ 2 e_j, \ j = 1, \hdots, N \} \cup \{ e_j \pm e_k, \ 1 \le j < k \le N \}.$
\end{itemize}

In each of the cases above, all roots of the same length belong to the same Weyl orbit.  Accordingly, we will frequently label the value of the multiplicity parameter $k$ according to the length of the corresponding roots.  For example, when $\Phi = BC_N$, we will write the values of $k$ on the three Weyl orbits as $k_1$, $k_{\sqrt{2}}$ and $k_2$.

\subsection{The rational theory} \label{sec:rational-theory}

In this section we let $\Phi$ be a reduced crystallographic root system spanning $V$, and we fix a multiplicity parameter $k$ taking non-negative real values. Write $n = \dim V$, so that for the root systems $B_N$, $C_N$, $D_N$ and $BC_N$ we have $n = N$, while for $A_{N-1}$ we have $n=N-1$.

\subsubsection{Dunkl operators and generalized Bessel functions}

The generalized Bessel functions are a multivariable generalization of the classical Bessel functions on the line.  They include as special cases many orbital integrals studied in random matrix theory, representation theory, and harmonic analysis, such as Harish-Chandra and Itzykson--Zuber integrals over compact groups \cite{HC, IZ, McS-expository}, rectangular spherical integrals \cite{guionnet2021large}, and spherical functions on all Euclidean symmetric spaces.  They are defined via a system of eigenvalue problems for differential-difference operators.

\begin{definition} \label{def:dunkl-ops}
For $\xi \in V$, the {\it Dunkl operator} $D_{k,\xi}$ is the differential-difference operator defined by
\begin{equation} \label{eqn:dunkl-op-def}
D_{k,\xi} f(x) = \partial_\xi f(x) + \sum_{\alpha \in \Phi^+} k_\alpha \frac{\langle \alpha, \xi \rangle}{\langle \alpha, x \rangle} \big[ f(x) - f(s_\alpha x) \big]
\end{equation}
for $f \in C^1(V)$, where $\partial_\xi$ is the directional derivative, i.e.~$\partial_\xi f = \langle \xi, \nabla f \rangle$.
\end{definition}

A key property of the Dunkl operators is that they commute:
\begin{equation}
    D_{k,x} \circ D_{k,y} = D_{k,y} \circ D_{k,x} \qquad \forall \, x,y \in V.
\end{equation}
We refer the reader to \cite{Rosler-DunklOps} for further properties of the Dunkl operators and their proofs.  Write $V_\C = V \otimes_\R \C$ for the complexification of $V$. A complete system of joint eigenfunctions for the Dunkl operators is given by the Dunkl kernel.

\begin{definition}
For $\lambda \in V_\C$, the {\it Dunkl kernel} is the unique function $E_{k,\lambda} \in C^\infty(V)$ satisfying the system of differential-difference equations
\begin{equation} \label{eqn:Dunkl-eigproblem}
D_{k, \xi} E_{k, \lambda} = \langle \lambda, \xi \rangle E_{k, \lambda} \qquad \forall \, \xi \in V,
\end{equation}
and normalized such that $E_{k,\lambda}(0) = 1$.
\end{definition}

\begin{definition} \label{def:GBF}
For $\lambda \in V_\C$, the {\it generalized Bessel function} $J_{k, \lambda}$ is the symmetrization of $E_{k,\lambda}$ over $W$:
\begin{equation} \label{eqn:GBF-def}
J_{k, \lambda}(x) = \frac{1}{|W|} \sum_{w \in W} E_{k, \lambda}(w(x)).
\end{equation}
\end{definition}

Generalized Bessel functions arise in various guises in many different fields of mathematics and physics.  Aside from random matrix theory, they have a particular significance in quantum integrable systems, where they can be used to construct the eigenstates of quantum rational Calogero--Moser systems \cite{EtingofCM, OlPeQ}.  Many properties of generalized Bessel functions, including specializations to other special functions of interest, are collected in \cite{AnkerDunklNotes}.  We note in particular that $J_{k, \lambda}$ extends to a holomorphic function on $V_\C$, and we have the symmetry properties
\begin{align}
    \label{eqn:GBF-symm}
    J_{k, \lambda}(x) &= J_{k, x}(\lambda), \\
    \label{eqn:GBF-scale}
    J_{k, \lambda}(tx) &= J_{k, t \lambda}(x), & \qquad \forall \, \lambda, x \in V_\C, \ t \in \C.
\end{align}

\begin{example} \label{ex:GBF-HC-int}
Let $G$ be a compact semisimple Lie group with Lie algebra $\gog$, and let $\tot \subset \gog$ be a Cartan subalgebra equipped with an $\mathrm{Ad}$-invariant inner product $\langle \cdot, \cdot \rangle$, which we use to identify $\tot \cong \tot^*$. If we take $V = \tot$, $\Phi$ the root system of $\gog$ with respect to $\tot$, and $k = \vec 1$ the multiplicity parameter with $k_\alpha = 1$ for all $\alpha \in \Phi$, then the generalized Bessel function is given by the \emph{Harish-Chandra integral}:
\begin{equation} \label{eqn:gbf-hc}
    J_{\vec 1, \lambda}(x) = \int_G e^{\langle \mathrm{Ad}_g \lambda, x \rangle} \, \rd g, \qquad \lambda, x \in V_\C,
\end{equation}
where $\rd g$ is the normalized Haar measure.  There is a remarkable closed-form expression for this integral \cite{HC, IZ, McS-heateqn}, but it involves an alternating sum with $N!$ or more terms, and the difficulty of controlling the resulting cancelations makes the formula unhelpful for large-$N$ analysis. See the notes by the second author for a detailed exposition of Harish-Chandra integrals \cite{McS-expository}.
\end{example}

\begin{example} \label{ex:iz-int}
On the other hand, if we take $V = \{ x \in \R^N : \sum_i x_i = 0 \}$ and let $\Phi$ be the $A_{N-1}$ root system in $V$, then specific choices of $k$ give the \emph{Itzykson--Zuber integrals}:
\begin{equation} \label{eqn:iz-int}
    J_{k, \lambda}(x) = \begin{cases}
        \int_{\mathrm{SO}(N)} e^{\Tr(\Lambda OXO^T)} \, \rd O, & k_{\sqrt{2}} = 1/2, \\
        \int_{\mathrm{SU}(N)} e^{\Tr(\Lambda UXU^*)} \, \rd U, & k_{\sqrt{2}} = 1, \\
        \int_{\mathrm{USp}(N)} e^{\Tr(\Lambda SXS^*)} \, \rd S, & k_{\sqrt{2}} = 2,
    \end{cases}
\end{equation}
where $\Lambda = \diag(\lambda)$, $X = \diag(X)$, and the integrals are with respect to the normalized Haar measures on the special orthogonal, special unitary, and unitary symplectic groups respectively.  Note that the value of the multiplicity parameter $k$ on the unique Weyl orbit of the $A_{N-1}$ root system is equal to half the parameter $\beta$ usually used in random matrix theory. Here we have assumed that the coordinates in $V$ sum to zero --- equivalently, that $X$ and $\Lambda$ are traceless --- in order to be consistent with our previous stipulation that $\Phi$ span $V$, but these assumptions are not hard to remove.  For $k = \vec 1$, the Itzykson--Zuber integral coincides with the Harish-Chandra integral over $\mathrm{SU}(N)$ and is known as the \emph{HCIZ integral}.  The other two integrals in (\ref{eqn:iz-int}) are \emph{not} of the form (\ref{eqn:gbf-hc}), so that the closed-form expression for Harish-Chandra integrals does not apply. The integrals in (\ref{eqn:gbf-hc}) and (\ref{eqn:iz-int}) appear in formulae for the joint spectral densities of Wishart matrices, off-center Wigner matrices, and sums of uniform random matrices with deterministic eigenvalues, among other random matrix ensembles \cite{AGZ, CM-Projections, CMZ2, CMZ, AGsurvey, Z1}.
\end{example}

\begin{example}
    In fact, all of the integrals in the preceding examples are spherical functions on Euclidean symmetric spaces, which can always be expressed as generalized Bessel functions.  We refer the reader to \cite[Remark 3.9]{AnkerDunklNotes} for the details of this construction and to \cite{HS, SHDS, SH, MN-majorization} for introductions to the theory of symmetric spaces and spherical functions, which are beyond the scope of this paper.
\end{example}

\begin{example}
    Another interesting class of examples are rectangular spherical integrals.  Take $\beta = 1$ or $2$ and $m \ge n \ge 1$.  For $\beta = 1$, let $A$ and $B$ be $n \times m$ real matrices and let $G(n) = \mathrm{O}(n)$, $G(m) = \mathrm{O}(m)$ be the groups of $n \times n$ and $m \times m$ orthogonal matrices respectively.  For $\beta = 2$, let $A$ and $B$ be $n \times m$ complex matrices and let $G(n) = \mathrm{U}(n)$, $G(m) = \mathrm{U}(m)$ be the groups of $n \times n$ and $m \times m$ unitary matrices.  Then define
    \begin{equation} \label{eqn:rect-int}
    I_{m,n,\beta}(A,B) = \int_{G_m} \int_{G_n} e^{\beta n \mathrm{Re}[\Tr(A^* U B V^*)]} \, \rd U \, \rd V,
    \end{equation}
    where the integrals are with respect to the normalized Haar measure on each group. The integrals $I_{m,n,\beta}$ were studied via the stochastic process approach in \cite{guionnet2021large} and, for $\beta = 2$, via the combinatorial approach in \cite{Novak:rec}.  When one argument is an $n \times m$ elementary matrix, their asymptotics as $n \to \infty$ are related to a generalization of the $R$-transform in free probability \cite{FBG-rectangular}.

    Take $V = \R^n$ with the standard basis $\{e_i\}_{i=1}^n$, and let $\Phi$ be the $B_n$ root system.  Take
    \begin{equation} \label{eqn:rect-k-conditions}
    k_{\sqrt{2}} = \frac{\beta}{2}, \qquad \qquad
    k_1 = \frac{1}{2} \big[ \beta(m-n+1) - 1 \big].
    \end{equation}
    Let $a_1 \ge \hdots \ge a_n$, $b_1 \ge \hdots \ge b_n$ be the singular values of $A$ and $B$ respectively, and for $i = 1, \hdots, n$, set $\lambda_i = \sqrt{\beta n} \, a_i$, $x_i = \sqrt{\beta n} \, b_i$. Then
    \begin{equation} \label{eqn:GBF-rect-int}
    J_{k,\lambda}(x) = I_{m,n,\beta}(A,B).
    \end{equation}
\end{example}

\begin{remark}
    The examples above might give the impression that generalized Bessel functions can typically be expressed as an integral over a compact group, but in fact this is known to hold only for certain special choices of the multiplicity parameter $k$. In the general case, there are no known integral formulae for generalized Bessel functions or for the Heckman--Opdam hypergeometric functions defined below in Section \ref{sec:hyperbolic-dunkl}. Discovering integral expressions for new special cases is a problem of significant interest; see \cite{RoslerVoit, SunIntegrals} for examples.
\end{remark}

\subsubsection{The radial Dunkl process}

The {\it Dunkl Laplacian} is the differential-difference operator
\[
\mathscr{L}_k = \sum_{j=1}^{n} D_{k,e_j}^2.
\]
Explicitly,
\begin{equation}
    \mathscr{L}_k f(x) = \sum_{j=1}^{n} \partial_j^2 f(x) + \sum_{\alpha \in \Phi^+} \frac{2 k_\alpha}{\langle \alpha, x \rangle} \partial_\alpha f(x) - \sum_{\alpha \in \Phi^+} \frac{k_\alpha |\alpha|^2}{\langle \alpha, x \rangle^2} \big[ f(x) - f(s_\alpha x) \big]
\end{equation}
for $f \in C^2(V)$, where $|\alpha|^2 = \langle \alpha, \alpha \rangle$ and $\partial_j = \partial_{e_j}$.  The differential part of $\mathscr{L}_k$ is the operator
\begin{equation}
    \mathscr{L}_k^W = \sum_{j=1}^{n} \partial_j^2 + \sum_{\alpha \in \Phi^+} \frac{2 k_\alpha}{\langle \alpha, x \rangle} \partial_\alpha.
\end{equation}

\begin{definition}
The {\it radial Dunkl process} $X(t)$, $t \ge 0$ is the $\mathcal{C}^+$-valued continuous-paths Markov process with generator $\frac{1}{2} \mathscr{L}_k^W.$
\end{definition}

\begin{remark}
Elsewhere in the literature, it is common to write $X^W(t)$ for the radial Dunkl process, reserving $X(t)$ for the ordinary Dunkl process, which is the $V$-valued Markov process with generator $\frac{1}{2} \mathscr{L}_k$.  Since here we study only the radial process, we omit the superscript $W$.
\end{remark}

Dunkl processes and their radial variants were introduced by R\"osler and Voit in \cite{RoslerVoit-Markov} and studied intensively by Gallardo, Yor and Godefroy \cite{GG, GY1, GY2, GY3}. Radial Dunkl processes are natural generalizations of the radial components of Brownian motions on Euclidean symmetric spaces, which can be recovered by choosing particular values of the multiplicity parameter.  Stochastic processes that can be constructed as the radial component of Brownian motion on a Euclidean symmetric space include Bessel processes and Dyson Brownian motion.

Let $B(t)$, $t \ge 0$ be a standard Brownian motion on $V$, and define
\[
\phi(x) = - \sum_{\alpha \in \Phi^+} k_\alpha \log(\langle \alpha, x \rangle).
\]
We have the following characterization of the radial Dunkl process as the solution to a stochastic differential equation (SDE).

\begin{thm}[\!\! \cite{Chib-SP, Demni-DunklProcs}]
\label{thm:dunkl-SDE}
When $k_\alpha \ge 1/2$ for all $\alpha \in \Phi$, the radial Dunkl process $X(t)$ is the unique strong solution of the SDE
\begin{align}
\begin{split}\label{eqn:rational-sde}
    \rd X(t) &= \rd B(t) - \nabla \phi(X(t)) \, \rd t \\
    &= \rd B(t) + \sum_{\alpha \in \Phi^+} \frac{k_\alpha \alpha}{\langle \alpha, X(t) \rangle} \, \rd t, \qquad t \ge 0, \ X(0) \in \mathcal{C}^+.
\end{split}
\end{align}
Moreover, almost surely $X(t)$ does not hit the boundary of $\mathcal{C}^+$ in finite time.
\end{thm}

\begin{remark}
    According to the results of \cite{Chib-SP, Demni-DunklProcs, SchapiraHOProc}, the assumptions on $k$ in Theorem \ref{thm:dunkl-SDE} above and in Theorem \ref{thm:HO-SDE} below can in fact be relaxed to $k_\alpha > 0$ for all $\alpha \in \Phi$ without destroying existence and uniqueness of strong solutions.  However, under these relaxed assumptions, the process may hit the boundary of $\mathcal{C}^+$ in finite time. For simplicity, in this paper we assume conditions on $k$ such that the process almost surely avoids the boundary, as this is consistent with the assumptions that are typically made in the random matrix theory literature.  It is possible though that the assumptions on the multiplicity parameters in some of our results could be further relaxed.
\end{remark}

The transition kernel (semigroup density) of $X(t)$ is
\begin{equation} \label{eqn:rational-transition}
p_t(x,y) = \frac{c_k}{t^{\gamma + \dim V /2}} e^{-(|x|^2 + |y|^2)/(2t)} J_{k, \frac{x}{\sqrt{t}}} \Big( \frac{y}{\sqrt{t}} \Big) \prod_{\alpha \in \Phi^+} \langle \alpha, y \rangle^{2 k_\alpha}, \qquad t > 0, \ x, y \in \mathcal{C}^+,
\end{equation}
where $\gamma = \sum_{\alpha \in \Phi^+} k_\alpha$, and $c_k$ is a normalization constant that can be computed using the Macdonald--Mehta--Opdam integral \cite{OpdamShiftOps}.  Write
\begin{equation} \label{eqn:rho-def}
\rho = \frac{1}{2} \sum_{\alpha \in \Phi^+} k_\alpha \alpha,
\end{equation}
and for $\alpha \in \Phi$, set $\alpha^\vee = \frac{2}{|\alpha|^2} \alpha$. Then  
\begin{align} \begin{split} \label{eqn:rational-c}
c_k &= \left[ \int_{\mathcal{C}^+} e^{-|x|^2/2} \prod_{\alpha \in \Phi^+} \langle \alpha, x \rangle^{2 k_\alpha} \, \rd x \right]^{-1} \\
&= \frac{|W|}{(2\pi)^{\dim V/2}} \prod_{\alpha \in \Phi^+} \frac{\Gamma \big(1 +  \langle \rho, \alpha^\vee \rangle \big) }{ \Gamma \big( 1+ \langle \rho, \alpha^\vee \rangle + k_\alpha \big) } \left( \frac{2}{|\alpha|^{2}} \right)^{k_\alpha}.
\end{split} \end{align}
Simplified formulae for the normalization constants for the classical root systems can be found in equations (1.4), (3.4) and (5.3) of \cite{Voit-CLT}. For $A_{N-1}$, $B_N$ and $D_N$, they are:\footnote{Our value of $c_k^A$ differs by a factor of $\sqrt{2\pi}$ from the formula in \cite{Voit-CLT} due to the fact that for $A_{N-1}$ we take $V$ to be an $(N-1)$-dimensional subspace rather than all of $\R^N$.}
\begin{align} \label{eqn:rational-c-A}
c_k^A &= \frac{N!}{(2 \pi)^{(N-1)/2}} \prod_{j=1}^N \frac{\Gamma\big(1+k_{\sqrt{2}}\big)}{\Gamma \big( 1+ jk_{\sqrt{2}} \big)}, \\
\label{eqn:rational-c-B}
c_k^{B} &= \frac{N!}{2^{N(k_1 + (N-1) k_{\sqrt{2}} - 1/2)}} \prod_{j=1}^N \frac{\Gamma \big(1 + k_{\sqrt{2}}\big)}{\Gamma \big( 1 + j k_{\sqrt{2}} \big) \Gamma \big( \frac{1}{2} + k_1 + (j-1) k_{\sqrt{2}} \big)}, \\
\label{eqn:rational-c-D}
c_k^D &= \frac{N!}{2^{N(N-1) k_{\sqrt{2}} - N/2 + 1}} \prod_{j=1}^N \frac{\Gamma \big(1 + k_{\sqrt{2}}\big)}{\Gamma \big( 1 + j k_{\sqrt{2}} \big) \Gamma \big( \frac{1}{2} + (j-1) k_{\sqrt{2}} \big)}.
\end{align}

\begin{example}
When $\Phi$ is the $B_N$ root system, the SDE (\ref{eqn:rational-sde}) takes the explicit form
\begin{equation} \label{eqn:rational-sde-B}
    \rd X_i(t) = \rd B_i(t) + \left[ \frac{k_1}{X_i(t)} + k_{\sqrt{2}} \sum_{j:j \ne i} \bigg( \frac{1}{X_i(t) - X_j(t)} + \frac{1}{X_i(t) + X_j(t)} \bigg) \right] \rd t,
\end{equation}
where $X_i(t)$ and $B_i(t)$ are the $i$th coordinates of $X(t)$ and $B(t)$ with respect to the standard basis.  The SDE for $C_N$ is the same, with $k_2$ in the place of $k_1$. Taking $k_1 = 0$ in (\ref{eqn:rational-sde-B}) gives the SDE for $D_N$.
\end{example}

\begin{remark} \label{rem:BC-vs-D}
    It follows directly from Definition \ref{def:dunkl-ops} that the Dunkl operators of type $C$, along with the corresponding radial Dunkl process and generalized Bessel functions, coincide with those of type $B$. Therefore, in the context of the rational theory, it is unnecessary to discuss these two cases separately (though in the hyperbolic theory, types $B$ and $C$ are genuinely distinct).
    
    On the other hand, even though the SDE for the type $D$ radial Dunkl process can be obtained formally from the SDE (\ref{eqn:rational-sde-B}) for type $B$ by setting $k_1 = 0$, at the level of the stochastic processes themselves, the type $D$ radial Dunkl process is {\it not} the same as the type $B$ radial Dunkl process with $k_1 = 0$, because the two processes have different domains. For $B_N$ or $C_N$ we have
    \[
    \mathcal{C}^+ = \{ x \in \R^N : x_1 \ge \hdots \ge x_N \ge 0 \},
    \]
    whereas for $D_N$ we have 
    \[
    \mathcal{C}^+ = \{ x \in \R^N : x_1 \ge \hdots \ge |x_N| \}.
    \]
    For this reason, when $k_{\sqrt{2}} \ge 1/2$ and $k_1 = 0$, the type $B$ radial process hits the boundary of its domain in finite time and no longer has a unique strong solution \cite{Demni-DunklProcs}.  In contrast, when $k_{\sqrt{2}} \ge 1/2$, the type $D$ radial Dunkl process has a unique strong solution and stays within its domain for all time.

    Similarly, the type $D$ generalized Bessel function is not invariant under the action of the type $B$ Weyl group and cannot be obtained from the type $B$ generalized Bessel function by setting $k_1 = 0$.  However, the type $B$ function (with $k_1 = 0$) can be obtained by symmetrizing the type $D$ function, while the type $D$ function can be written in terms of two type $B$ functions with different multiplicity parameters \cite{Demni-typeD}. Nevertheless, in all of the asymptotic results shown in this paper, the limiting formulae for type $D$ can in fact be obtained from corresponding formulae for type $B$ by letting $k_1 \to 0$ as $N \to \infty$.
\end{remark}

\begin{example}
When $k_\alpha = 1$ for all $\alpha \in \Phi$, the radial Dunkl process is a Brownian motion conditioned in the sense of Doob not to leave $\mathcal{C}^+$, as studied by Grabiner \cite{DG} and by Biane, Bougerol and O'Connell \cite{BBO}.
\end{example}

\begin{example} \label{ex:DBM}
    When $\Phi$ is the $A_{N-1}$ root system and $k_\alpha = 1$ for all $\alpha \in \Phi$, if we let $B(t)$ be a standard Brownian motion on the full space $\R^N$ rather than a Brownian motion confined to the subspace $V = \{ x \in \R^N : \sum_i x_i = 0 \}$, then (\ref{eqn:rational-sde}) recovers the SDE for Dyson Brownian motion:
    \begin{equation}
        \rd X_i(t) = \rd B_i(t) + \sum_{j:j \ne i} \frac{\rd t}{X_i(t) - X_j(t)}.
    \end{equation}
\end{example}

\subsection{The hyperbolic theory}
\label{sec:hyperbolic-dunkl}

In this section we allow $\Phi$ to be an arbitrary crystallographic root system spanning $V$, not necessarily reduced. Again, $k$ is a multiplicity parameter taking non-negative real values, and $n = \dim V$.

\subsubsection{Cherednik operators and Heckman--Opdam hypergeometric functions}

The Heckman--Opdam hypergeometric functions are a multivariable generalization of the classical Gauss hypergeometric function.  They include as special cases the spherical functions on symmetric spaces of non-compact type.  In a sense, they are a further generalization of the generalized Bessel functions, which can be recovered from them via the {\it rational limit}, as explained at the end of this section.

Like the generalized Bessel functions, the Heckman--Opdam hypergeometric functions are defined via a system of eigenvalue problems for differential-difference operators.

\begin{definition}
For $\xi \in V$, the {\it Cherednik operator} $T_{k,\xi}$ is the differential-difference operator defined by
\begin{equation} \label{eqn:cherednik-op-def}
T_{k,\xi} f(x) = \partial_\xi f(x) + \sum_{\alpha \in \Phi^+} k_\alpha \frac{\langle \alpha, \xi \rangle}{1-e^{-\langle \alpha, x \rangle}} \big[ f(x) - f(s_\alpha x) \big] - \langle \rho, \xi \rangle f(x)
\end{equation}
for $f \in C^1(V)$, where $\rho$ is defined in (\ref{eqn:rho-def}).
\end{definition}

Like the Dunkl operators, the Cherednik operators commute:
\begin{equation}
    T_{k,x} \circ T_{k,y} = T_{k,y} \circ T_{k,x} \qquad \forall \, x,y \in V.
\end{equation}

Similarly to the rational setting, for $\lambda \in V_\C$, there is a unique function $G_{k,\lambda} \in C^\infty(V)$ satisfying the system of differential-difference equations
\begin{equation} \label{eqn:hyperbolic-eigproblem}
T_{k, \xi} G_{k, \lambda} = \langle \lambda, \xi \rangle G_{k, \lambda} \qquad \forall \, \xi \in V,
\end{equation}
and normalized such that $G_{k,\lambda}(0) = 1$.

\begin{definition}
For $\lambda \in V_\C$, the {\it Heckman--Opdam hypergeometric function} $F_{k, \lambda}$ is the symmetrization of $G_{k,\lambda}$ over $W$:
\begin{equation} \label{eqn:HGF-def}
F_{k, \lambda}(x) = \frac{1}{|W|} \sum_{w \in W} G_{k, \lambda}(w(x)).
\end{equation}
\end{definition}

The Heckman--Opdam hypergeometric functions are closely related to the eigenfunctions of the quantum hyperbolic Calogero--Moser Hamiltonian and were originally introduced in order to prove complete integrability of this and other quantum Calogero--Moser variants \cite{RSHF1, RSHF2, RSHF3, RSHF4}.

In general, analogues of the properties (\ref{eqn:GBF-symm}) and (\ref{eqn:GBF-scale}) do {\it not} hold for $F_{k,\lambda}$.

\begin{example}
    Take $V = \R$ and embed the $A_1$ root system as $\Phi = \{\pm 2\}$. The multiplicity parameter $k$ takes only a single value, and we have
    \begin{equation} \label{eqn:HO-to-Gauss}
F_{k, \lambda}(x) = {}_{2}F_1 \Big( k + \lambda, \, k - \lambda ; \, k + \frac{1}{2} ; \, - \sinh^2 \frac{x}{2} \Big),
    \end{equation}
where ${}_{2}F_1$ is the classical Gauss hypergeometric function.  Taking the rational limit as shown below in (\ref{eqn:rational-limit}), we can also obtain the classical Bessel functions on the line as special cases of generalized Bessel functions.
\end{example}

\begin{example}
    Characters of compact Lie groups can be expressed as Heckman--Opdam hypergeometric functions.  Let $G$ be a compact semisimple\footnote{The assumption of semisimplicity just ensures that $\Phi$ spans $V$, and is trivial to remove since characters of a non-semisimple compact Lie group are easily expressed in terms of characters of a semisimple group.} Lie group with Lie algebra $\gog$, and $\tot \subset \gog$ a Cartan subalgebra, which we identify with its dual via an $\mathrm{Ad}$-invariant inner product. Take $V = \tot$, let $\Phi$ be the root system of $\gog$ with respect to $\tot$, and let $k = \vec 1$ be the multiplicity parameter with $k_\alpha = 1$ for all $\alpha \in \Phi$.  Write $V_\lambda$ for the irreducible representation of $G$ with highest weight $\lambda \in \tot$, and $\chi_\lambda : G \to \C$ for the character of $V_\lambda$. In this case the Heckman--Opdam hypergeometric function extends holomorphically to $\tot_\C$, and we have
    \begin{equation}
        F_{\vec 1, \lambda + \rho}(ix) = \frac{\chi_\lambda(e^x)}{\dim V_\lambda}, \qquad x \in \tot,
    \end{equation}
    where $e^x$ denotes the Lie exponential. As discussed below in Section \ref{sec:mult-prelims}, the characters can also be expressed in terms of generalized Bessel functions.
\end{example}

\begin{example}
    Spherical functions on symmetric spaces of non-compact type can all be realized as special cases of Heckman--Opdam hypergeometric functions.  We refer the reader to \cite[Remark 4.6]{AnkerDunklNotes} or \cite[Example 4.4]{MN-majorization} for details of this construction.
\end{example}

\subsubsection{The radial Heckman--Opdam process}

The {\it Heckman--Opdam Laplacian} is the differential-difference operator
\[
\mathscr{D}_k = \sum_{j=1}^{n} T_{k,e_j}^2.
\]
Explicitly,
\begin{align}
    \mathscr{D}_k f(x) = \sum_{j=1}^{n} \partial_j^2 f(x) &+ \sum_{\alpha \in \Phi^+} k_\alpha \coth \frac{\langle \alpha, x \rangle}{2} \partial_\alpha f(x) + |\rho|^2 f(x) \\
    \nonumber
    & - \sum_{\alpha \in \Phi^+} k_\alpha \frac{|\alpha|^2}{4 \sinh^2 \frac{\langle \alpha, x \rangle}{2}} \big[ f(x) - f(s_\alpha x) \big]
\end{align}
for $f \in C^2(V)$.  The differential part of $\mathscr{D}_k$ is the operator
\begin{equation}
    \mathscr{D}_k^W = \sum_{j=1}^{n} \partial_j^2 + \sum_{\alpha \in \Phi^+} k_\alpha \coth \frac{\langle \alpha, x \rangle}{2} \partial_\alpha + |\rho|^2.
\end{equation}

\begin{definition}
The {\it radial Heckman--Opdam process} $Y(t)$, $t \ge 0$ is the $\mathcal{C}^+$-valued continuous-paths Markov process with generator $\frac{1}{2} (\mathscr{D}_k^W - |\rho|^2).$
\end{definition}

Just as radial Dunkl processes generalize the radial components of Brownian motions on Euclidean symmetric spaces, radial Heckman--Opdam processes generalize the radial components of Brownian motions on symmetric spaces of non-compact type.

Again let $B(t)$, $t \ge 0$ be a standard Brownian motion on $V$, and define
\[
\psi(x) = - \sum_{\alpha \in \Phi^+} k_\alpha \log \bigg( \sinh \frac{\langle \alpha, x \rangle}{2} \bigg).
\]
For $\alpha \in \Phi$, we set $k_{\alpha / 2} = 0$ if $\alpha/2 \not \in \Phi$ and $k_{2 \alpha} = 0$ if $2 \alpha \not \in \Phi$.  We have the following characterization of the radial Heckman--Opdam process as the solution to an SDE.

\begin{thm}[\!\! \cite{SchapiraHOProc}, Proposition 4.1 and remarks preceding Proposition 4.2] \label{thm:HO-SDE}
When $k_\alpha + k_{2 \alpha} \ge 1/2$ for all $\alpha \in \Phi$, the radial Heckman--Opdam process $Y(t)$ is the unique strong solution of the SDE
\begin{align} \begin{split} \label{eqn:hyperbolic-sde}
    \rd Y(t) &= \rd B(t) - \nabla \psi(Y(t)) \, \rd t \\
    &= \rd B(t) + \sum_{\alpha \in \Phi^+} \frac{k_\alpha \alpha}{2} \coth \frac{\langle \alpha, Y(t) \rangle}{2}\, \rd t, \qquad t \ge 0, \ Y(0) \in \mathcal{C}^+.
\end{split} \end{align}
Moreover, almost surely $Y(t)$ does not hit the boundary of $\mathcal{C}^+$ in finite time. 
\end{thm}

The expression for the transition kernel of $Y(t)$ in terms of $F_{k,\lambda}$ is less explicit than the expression for the transition kernel of $X(t)$ in terms of $J_{k, \lambda}$. For $\alpha \in \Phi$ set $\alpha^\vee = \frac{2}{|\alpha|^2} \alpha$, and for $\lambda \in V_\C$ define\footnote{The notation $\mathbf{\tilde c}$ follows \cite[\textsection3.4]{HS} and indicates the {\it unnormalized} version of the function $\mathbf{c}(\lambda) = \mathbf{\tilde c}(\lambda) / \mathbf{\tilde c}(\rho)$.}
\begin{equation} \label{eqn:def-cfunc}
    \mathbf{\tilde c}(\lambda) = \prod_{\alpha \in \Phi^+} \frac{\Gamma \big( \langle \lambda, \alpha^\vee \rangle + \frac{1}{2} k_{\frac{\alpha}{2}}\big)}{\Gamma \big( \langle \lambda, \alpha^\vee \rangle + k_\alpha + \frac{1}{2} k_{\frac{\alpha}{2}}\big)}.
\end{equation}
Then the transition kernel of $Y(t)$ is (see \cite[\textsection2]{SchapiraHOProc}):
\begin{equation} \label{eqn:hyperbolic-transition}
q_t(x,y) = c'_k\int_{i V} e^{-\frac{t}{2}(|\lambda|^2 + |\rho|^2)} F_{k,\lambda}(x) F_{k,\lambda}(-y) \, \mathbf{\tilde c}(\lambda)^{-1} \mathbf{\tilde c}(-\lambda)^{-1} \, \rd \lambda, \qquad t > 0, \ x, y \in \mathcal{C}^+,
\end{equation}
where $iV$ is the imaginary subspace in $V_\C$ and $c'_k$ is a constant.

\begin{example}
When $\Phi$ is the $BC_N$ root system, the SDE (\ref{eqn:hyperbolic-sde}) takes the explicit form
\begin{align} \begin{split}\label{eqn:hyperbolic-sde-BC}
    \rd Y_i(t) = \rd B_i(t) + \frac{1}{2} \bigg[ k_1 \coth &\frac{Y_i(t)}{2} + 2 k_2 \coth Y_i(t) \\
    &+ k_{\sqrt{2}} \sum_{j:j \ne i} \bigg( \coth \frac{Y_i(t) - Y_j(t)}{2} + \coth \frac{Y_i(t) + Y_j(t)}{2} \bigg) \bigg] \rd t.
\end{split} \end{align}
\end{example}

\begin{example}
When $\Phi$ is the $A_{N-1}$ root system, if we again take $B(t)$ to be a standard Brownian motion on the full space $\R^N$ as in Example \ref{ex:DBM} above, then (\ref{eqn:hyperbolic-sde}) gives
\begin{equation} \label{eqn:sinh-DBM}
    \rd Y_i(t) = \rd B_i(t) + \frac{k_{\sqrt{2}}}{2} \sum_{j:j \ne i} \coth \frac{Y_i(t) - Y_j(t)}{2} \rd t.
\end{equation}
The system (\ref{eqn:sinh-DBM}) describes a dynamical version of the $\sinh$-model studied in \cite{BGK}, in the case of a quadratic confining potential. The relation between (\ref{eqn:sinh-DBM}) and the $\sinh$-model is similar to the relation between Dyson Brownian motion (for general $\beta$) and $\beta$-models in random matrix theory.
\end{example}

\subsubsection{Relation to the rational theory}

Generalized Bessel functions are the {\it rational limit} of Heckman--Opdam hypergeometric functions for reduced root systems (see \cite[\textsection4.4]{AnkerDunklNotes}):
\begin{equation} \label{eqn:rational-limit}
J_{k, \lambda}(x) = \lim_{\varepsilon \to 0} F_{k, \varepsilon^{-1} \lambda}(\varepsilon x).
\end{equation}
The radial Dunkl process can also be realized as a limit of appropriately normalized radial Heckman--Opdam processes as described in \cite[\textsection6]{SchapiraHOProc}.

\section{Large deviations of radial Dunkl processes}
\label{sec:LDPs}
In this section we prove our first main result, a formula for the large-$N$ limits of generalized Bessel functions, which we derive from a large deviations principle for radial Dunkl processes.  We primarily consider types $B$, $C$ and $D$, leaving type $A$, which is slightly simpler and better understood in the existing literature, to the end of Section \ref{sec:bessel}.  In Section \ref{sec:hyp-LDP}, we study the large deviations of radial Heckman--Opdam processes.

We recall the following SDE from \eqref{eqn:rational-sde-B}, which describes the radial Dunkl processes of type $B$, $C$ (with $k_2$ in the place of $k_1$), or $D$ (when $k_1 = 0$):
\begin{align}
\label{e:rD}
    \rd X_i(t) = \rd B_i(t) + \left[ \frac{k_1}{X_i(t)} + k_{\sqrt{2}} \sum_{j:j \ne i} \bigg( \frac{1}{X_i(t) - X_j(t)} + \frac{1}{X_i(t) + X_j(t)} \bigg) \right] \rd t.
\end{align}
The law at $t=1$ is given by
\begin{equation} 
p_1(x,y) = C_k e^{-(|x|^2 + |y|^2)/2} J_{k, x} (y)   
\prod_i y_i^{2k_1}
\prod_{i<j}(y_i^2-y_j^2)^{2k_{\sqrt{2}}}.
\end{equation}

The solution to the following SDE is known as the {\it Dyson Bessel process} and has been studied in \cite{guionnet2021large}:
\begin{align}\begin{split}\label{e:DBP}
\rd s_i(t) 
&=\frac{\rd B_{i}}{\sqrt{\beta N}}+\left(\frac{1}{2N}\sum_{j: j \neq i}\frac{1}{s_i(t)-s_j(t)}+\frac{1}{2N}\sum_{j: j\neq i}\frac{1}{s_i(t)+s_j(t)}+\frac{\al_N}{2 s_i(t)}\right)\rd t,
\end{split}\end{align}
for $1\leq i\leq N$, where $B_1, B_2,\cdots, B_N$ are independent Brownian motions.

When $\alpha_N > 0$, the Dyson Bessel process is related to the type $B$ radial Dunkl process by a simple change of normalization and relabeling of parameters.  To identify \eqref{eqn:rational-sde-B} with \eqref{e:DBP}, we can simply take
\begin{equation} \label{eqn:rat-cov}
    X_i(t)= \sqrt{\beta N} s_i(t),\qquad k_{\sqrt 2}=\frac{\beta}{2},\qquad k_1=\frac{\beta N \al_N}{2}.
\end{equation}
When $\alpha_N = 0$, we instead obtain an analogous identification with the type $D$ radial Dunkl process.

Accordingly, we can write the law of the Dyson Bessel process at time $t=1$ in terms of generalized Bessel functions as follows. If $\alpha_N > 0$, let $\Phi$ be the $B_N$ root system and set $C_k = c_k^B$ as defined in (\ref{eqn:rational-c-B}). If $\alpha_N = 0$, instead let $\Phi$ be the $D_N$ root system and set $C_k = c_k^D$ as defined in (\ref{eqn:rational-c-D}). Then, by (\ref{eqn:rational-transition}), the law of the Dyson Bessel process at time $t=1$ is given by
\begin{align}\begin{split}\label{e:density} 
p_1(a,b) 
&= C_k e^{-\beta N(|a|^2 + |b|^2)/2} J_{k, \sqrt{\beta N}a} (\sqrt{\beta N}b)   
\prod_i (\sqrt{\beta N}b_i)^{\beta N \alpha_N}
\prod_{i<j}((\beta N)(b_i^2-b_j^2))^{\beta}\\
&=C_{k,N} e^{-\beta N(|a|^2 + |b|^2)/2} J_{k, \sqrt{\beta N}a} (\sqrt{\beta N}b)   
\prod_i b_i^{\beta N \alpha_N}
\prod_{i<j}(b_i^2-b_j^2)^{\beta}
\end{split}\end{align}
for $a, b \in \mathcal{C}^+$, where $C_{k,N}=C_k(\beta N)^{(\beta/2)(\alpha_N N^2+N(N-1))}$.

\subsection{Large-$N$ asymptotics for generalized Bessel functions}
\label{sec:bessel}

We denote the empirical particle density of the Dyson Bessel process \eqref{e:DBP} and its symmetrization (with respect to reflection through $0$) respectively by
\begin{align*}
\nu_t^N=\frac{1}{N}\sum_{i=1}^N \delta_{s_i(t)}, \quad \widehat\nu_t^N=\frac{1}{2N}\sum_{i=1}^N (\delta_{s_i(t)}+\delta_{-s_i(t)}).
\end{align*}
We recall the large deviations principle for $\{\widehat\nu_{t}^{N}\}_{0\leq t\leq 1}$ from \cite{guionnet2021large}. Write $\mathcal{P}(\R)$ for the space of probability measures on $\R$ and $\mathcal{P}(\R)_\pm$ for the space of probability measures on $\R$ that are symmetric with respect to reflection through $0$. For $\widehat \mu_0 \in \mathcal{P}(\R)_\pm$ and $\widehat \nu_t : (0,1) \to \mathcal{P}(\R)_\pm$, the rate function is given by
\begin{align}
\label{e:ratt}
S^\al_{\widehat\mu_{0}}(\{\widehat\nu_t\}_{0\leq t\leq 1})=\sup_{f\in \cC^{2,1}_b} S^\al(\{\widehat\nu_t,f_t\}_{0\leq t\leq 1}),
\end{align}
where $S^\al(\{\widehat\nu_t,f_t\}_{0\leq t\leq 1})$ is given by
\begin{align*}
\widehat\nu_1(f_1) &-\widehat\mu_0(f_0)-\int_0^1\int \del_s f_s(x) \, \rd\widehat\nu_s(x)\, \rd s-\frac{1}{2}\int_0^1\int \frac{f_s'(x)-f_s'(y)}{x-y} \, \rd \widehat\nu_s(x) \, \rd \widehat\nu_s(y) \, \rd s \\
& -\frac{\alpha}{2}\int_0^1\int \frac{f_s'(x)}{x} \, \rd \widehat\nu_s(x) \, \rd s  -\frac{1}{8\beta }\int^1_0 \int (f_s'(x)-f_s'(-x))^2 \, \rd\widehat\nu_s(x) \, \rd s
\end{align*}
when $\widehat\nu_0 = \widehat\mu_0$. If $\widehat\nu_0\neq \widehat\mu_0$, then $S^\al_{\widehat\mu_{0}}(\{\widehat\nu_t\}_{0\leq t\leq 1})= +\infty$.
As shown in \cite{guionnet2021large}, we have the following large deviations principle.

\begin{thm}\label{main2}
Fix a  symmetric probability measure  $\widehat\mu_0$ and a sequence of initial conditions for \eqref{e:DBP} with symmetrized empirical measures $\widehat\nu^N_0$ converging weakly to $\widehat\mu_{0}$, with uniformly bounded second moment.  Then,
if $\alpha_N$ converges towards $\alpha\in [0,\infty)$ when $N$ goes to infinity {{so that either $\alpha_N\ge 1/\beta N$ or $\alpha_N\equiv 0$}}, the distribution of   the empirical particle density $\{\widehat\nu_t^N\}_{0\leq t\leq 1}$ of the Dyson Bessel process \eqref{e:DBP} satisfies a large deviations principle in the scale $N^2$ and with good rate function $S_{\widehat\nu_0}^\al$. In particular, 
for any continuous symmetric measure-valued  process $\{\widehat\nu_t\}_{0\leq t\leq 1}$, we have: 
\begin{align}\begin{split}\label{e:ulbb}
&\phantom{{}={}}\lim_{\delta\rightarrow 0}\liminf_{N\rightarrow\infty}\frac{1}{N^2}\log \bP(\{\widehat\nu^N_t\}_{0\leq t\leq 1}\in \bB(\{\widehat\nu_t\}_{0\leq t\leq 1}, \delta))\\
&=\lim_{\delta\rightarrow 0}\limsup_{N\rightarrow\infty}\frac{1}{N^2}\log \bP(\{\widehat\nu^N_t\}_{0\leq t\leq 1}\in \bB(\{\widehat\nu_t\}_{0\leq t\leq 1}, \delta))
= -{{S^\al_{\widehat\mu_0}(\{\widehat\nu_t\}_{0\leq t\leq 1}).}}
\end{split}\end{align}
\end{thm}

As a consequence of Theorem \ref{main2} and the relation \eqref{e:density}, we have the following asymptotics for the generalized Bessel functions. We denote by $\Sigma$ the non-commutative entropy,
$$\Sigma(\nu)=\int \log |x-y| \, d\nu(x) \, d\nu(y), \qquad \nu \in \cP(\R).$$

\begin{thm}
\label{main1}
Let $\Phi = B_N,$ $C_N$ or $D_N$.  Take $k_{\sqrt 2} \ge 1/2$ and set
\begin{equation} \label{eqn:k-beta}
    \beta = 2 k_{\sqrt{2}}.
\end{equation}
If $\Phi = B_N$ or $C_N$, take 
$k_j = k_j(N) \ge 1/2$, where $j =1$ if $\Phi = B_N$ and $j=2$ if $\Phi = C_N$, such that
\begin{align}\label{e:k1k2}
\frac{2 k_j}{\beta N}=\alpha_N\rightarrow \alpha.
\end{align}
If $\Phi= D_N$, set $\alpha_N = \alpha = 0$. 

For each $N = 1, 2, \hdots,$ fix vectors
\begin{align*}
x &= x^{(N)} = (x^{(N)}_1, x^{(N)}_2, \cdots, x^{(N)}_N ), \\
y &= y^{(N)} = (y^{(N)}_1, y^{(N)}_2, \cdots, y^{(N)}_N ) \in \R^N,
\end{align*}
such that the following scaled, symmetrized empirical measures converge weakly as $N \to \infty$:
\begin{align} \begin{split} \label{eqn:ssem}
    \widehat\nu_{A}^{N} = \frac{1}{2N}\sum_{i}\delta_{\frac{x^{(N)}_i}{\sqrt{\beta N}}}+\delta_{\frac{-x^{(N)}_i}{\sqrt{\beta N}}}\rightarrow \widehat \nu_A,\\
    \widehat\nu_{B}^{N} = \frac{1}{2N}\sum_{i}\delta_{\frac{y^{(N)}_i}{\sqrt{\beta N}}}+\delta_{\frac{-y^{(N)}_i}{\sqrt{\beta N}}}\rightarrow \widehat \nu_B.
\end{split} \end{align}
 We moreover assume that for $C=A$ or $B$, we have  $\sup_{N}\widehat\nu_{C}^{N}(x^{2})<\infty$, $\Sigma(\widehat\nu_{C})>-\infty$ and, if $\alpha\neq 0$,
$\int \log |x| \rd \widehat \nu_{C}>-\infty$. 
Then the  following limit of the generalized Bessel function exists:
\begin{equation} \label{e:Jasymp}
\lim_{N \to \infty}\frac{1}{N^{2}}\log J_{k,x}(y)= I_{\alpha,\beta}(\widehat \nu_{A},\widehat\nu_{B}) = I(\widehat \nu_{A},\widehat\nu_{B}), 
\end{equation}
where we omit the $\alpha$ and $\beta$ in subscript when the values of these parameters are clear in context. The functional $I$ is given explicitly by
\begin{align}\begin{split}\label{e:arate}
I(\widehat\nu_{A},\widehat\nu_{B})= & -\frac{\beta}{2}\inf_{\{\widehat\rho_t\}_{0<t<1} \atop \text{satisfies \eqref{e:bbterm}}}\left\{\int_0^1 \int u_s^2  \widehat\rho_s(x)\rd x \rd s+\frac{\pi^2}{3}\int_0^1\int \widehat \rho^3_s(x) \rd x  \rd s
+\frac{\alpha^2}{4}\int \frac{\widehat\rho_s(x)}{x^2}\rd x\rd s
\right\}\\
& +\frac{\beta}{2}\big[\widehat\nu_{A}(x^{2}-\alpha \log |x|)+\widehat\nu_{B}(x^{2}-\alpha \log |x|)-\big(\Sigma(\widehat\nu_A)+\Sigma(\widehat\nu_B) \big)\big]-\fC(\alpha, \beta),
\end{split}\end{align}
where $\fC(\alpha, \beta)$ is a constant given by
\begin{align}
    \fC(\alpha, \beta)=\frac{\beta}{4}\left(3(\alpha+1) + \alpha^2\log \alpha - (\alpha+1)^2\log(\alpha+1) \right),
\end{align}
with the convention $0^2 \log 0 = 0$. The infimum is taken over continuous symmetric measure-valued processes  $\{ \widehat\rho_{t}(x)\rd x \}_{{0<t<1}}$ satisfying the weak limits
\begin{align}\label{e:bbterm}
\lim_{t\rightarrow 0^+}\widehat\rho_t(x)\rd x=\widehat\nu_A,\qquad
\lim_{t\rightarrow 1^-}\widehat\rho_t(x)\rd x=\widehat\nu_B.
\end{align}
Moreover, 
$u_s$ is  the weak solution of the following conservation of mass equation:
\begin{align} \label{e:cmass}
\del_s\widehat \rho_s+\del_x(\widehat \rho_s u_s)=0.
\end{align}
\end{thm}

\begin{remark}
    In Theorem \ref{main1}, we assumed that $k_{\sqrt{2}}\geq 1/2$, and either $k_1+k_2\geq 1/2$ or $k_1=k_2=0$. These assumptions are to ensure that the particles in the radial Dunkl process will not collide or collapse to $0$ in finite time (i.e., that the process will not hit the boundary of its domain), following the conditions of Theorem \ref{thm:dunkl-SDE}.
\end{remark}

\begin{remark}
When $\Phi = D_N$ the final coordinate of a point in the positive Weyl chamber may be negative.  However, the functional (\ref{e:arate}) depends only on the symmetrized measures $\widehat\nu_{A}$ and $\widehat\nu_{B}$, which are invariant under permutations and sign changes of the coordinates of the arguments $x$ and $y$.  Thus, for characterizing the asymptotic behavior of type $D$ generalized Bessel functions in the regime that we consider here, it is actually sufficient to consider only arguments in the Weyl chamber of type $B$.  Moreover, the limit is given by the same functional (\ref{e:arate}) as for type $B$, with $\alpha = 0$.  However, as explained in Remark \ref{rem:BC-vs-D}, the generalized Bessel function of type $D$ is {\it not} equal to the function of type $B$ with $k_1 = 0$.  See also Remark \ref{rem:BC-vs-D-rep} below.\end{remark}

\begin{remark}
    The functional to be minimized in (\ref{e:arate}) admits an elegant interpretation as the action of a one-dimensional fluid.  That is, if we let $\{\widehat\rho_t^*\}_{0 \le t \le 1}$ be the unique solution to the minimization problem
    \begin{equation} \label{eqn:euler-action}
    \inf_{\{\widehat\rho_t\}_{0<t<1} \atop \text{satisfies \eqref{e:bbterm}}}\left\{\int_0^1 \int u_s^2  \widehat\rho_s(x)\rd x \rd s+\frac{\pi^2}{3}\int_0^1\int \widehat \rho^3_s(x) \rd x  \rd s
+\frac{\alpha^2}{4}\int \frac{\widehat\rho_s(x)}{x^2}\rd x\rd s
\right\},
    \end{equation}
then $\{\widehat\rho_t^*\}_{0 \le t \le 1}$ satisfies the Euler--Lagrange equation
\begin{equation} \label{eqn:euler}
    \partial_t u_t + \partial_x \Big( u_t^2 - \pi^2(\widehat\rho_t^*)^2 - \frac{\alpha^2}{4x^2} \Big) = 0.
\end{equation}
Together with the conservation of mass equation (\ref{e:cmass}), the PDE (\ref{eqn:euler}) describes a one-dimensional Euler fluid with repulsion from the origin. That is, the minimum value (\ref{eqn:euler-action}) equals the action of this fluid along the trajectory that solves the shooting problem with density $\widehat \rho_0(x)$ at time 0 and density $\widehat \rho_1(x)$ at time 1. If we now define the complex-valued function
\[
f_t(x) = u_t(x) + i \pi \widehat\rho_t^*(x),
\]
then (\ref{e:cmass}) and (\ref{eqn:euler}) together can be rewritten as the single equation
\begin{equation} \label{eqn:CBx}
    \partial_t f + f \partial_x f = \frac{\alpha^2}{4x^3},
\end{equation}
which is the complex Burgers equation with an additional forcing term.  The complex Burgers equation appears ubiquitously in the study of limiting phenomena for integrable models; see e.g. \cite{KenOk, GM, ReshSri}.  In \cite[\textsection5.1]{guionnet2021large}, the equation (\ref{eqn:CBx}) is solved formally by the method of characteristics.  This yields a more explicit description of the minimizer, but caution is required: the non-rigorous solution method is complicated by the fact that the characteristics can flow into the complex plane, whereas $f_t(x)$ in (\ref{eqn:CBx}) is a function of two real variables.
\end{remark}

Before proving Theorem \ref{main1}, we first derive the following continuity estimate for the generalized Bessel function. We recall the Wasserstein $W_1$ distance. For two empirical measures
\[
\mu = \frac{1}{N}\sum_{i=1}^N \delta_{z_i}, \qquad \mu' = \frac{1}{N} \sum_{i=1}^N \delta_{z'_i},
\]
where $z_i, z'_i \in \R$, their $W_1$ distance is given by
$$W_1(\mu,\mu') = \min_{\sigma \in S_N} \frac{1}{N} \sum_{i=1}^N |z_{\sigma(i)}-z'_i|,$$
where $S_N$ is the permutation group on $N$ letters. 
\begin{prop}\label{p:cont}
We encode vectors $y=(y_1, y_2,\cdots, y_N)$, $y'= (y'_1, y'_2,\cdots, y_N') \in V \cong \R^N$ by their scaled empirical measures
\begin{align}
    \mu=\frac{1}{N}\sum_{i=1}^N\delta_{\frac{y_i}{\sqrt {\beta N}}},
    \quad 
    \mu'=\frac{1}{N}\sum_{i=1}^N\delta_{\frac{y'_i}{\sqrt{\beta N}}}.
\end{align}
For any $\delta>0$ and $x=(x_1,x_2,\cdots, x_N) \in V$, if $W_1(\mu, \mu')\leq \sqrt{\frac{N}{\beta}} \frac{\delta}{\max |x_i|}$, then 
\begin{align}\label{e:Jratio}
    \frac{1}{N^2}\log \left|\frac{J_{k,x}\left(y\right)}{J_{k,x}\left(y'\right)}\right|\leq \delta.
\end{align}
\end{prop}

In order to prove Proposition \ref{p:cont}, we show a lemma bounding the Lipschitz constant of $\log J_{k,\lambda}$.  We also record a similar statement for $F_{k,\lambda}$.  For $\lambda, x \in V$, we have the following positivity results and exponential bounds (see \cite[Proposition 3.10]{AnkerDunklNotes} and \cite[Proposition 4.2]{BrenneckenRosler}):
\begin{equation}
    0 < E_{k,\lambda}(x) \le e^{\max_{w \in W} \langle \lambda, w(x) \rangle},
\end{equation}
\begin{equation}
    0 < G_{k,\lambda}(x) \le e^{\max_{w \in W} \langle \lambda + \rho, w(x) \rangle}.
\end{equation}
Averaging over the Weyl group yields identical bounds on $J_{k,\lambda}(x)$ and $F_{k,\lambda}(x)$.

\begin{lem}\label{l:LipJ}
For $\lambda, x, y \in V$,
\begin{equation}
    J_{k,\lambda}(x+y) \le e^{\max_{w \in W} \langle \lambda, w(y) \rangle} J_{k, \lambda}(x).
\end{equation}
\end{lem}
\begin{proof}
From \cite[Proposition 4.2]{BrenneckenRosler} (see also \cite[Theorem 3.3]{RKV13}), we have
\begin{equation}
    G_{k, x+y}(\lambda) \le e^{\max_{w \in W} \langle \lambda, w(y) \rangle} G_{k, x}(\lambda).
\end{equation}
Averaging over the Weyl group gives an analogous inequality for $F_{k, x+y}(\lambda)$, and taking the rational limit (\ref{eqn:rational-limit}) on either side we obtain
\[
J_{k, x+y}(\lambda) \le e^{\max_{w \in W} \langle \lambda, w(y) \rangle} J_{k, x}(\lambda).
\]
The desired result then follows from the symmetry property (\ref{eqn:GBF-symm}).
\end{proof}

\begin{lem}
For $\lambda, x, y \in V$,
\begin{equation} \label{eqn:logF-lipschitz}
    F_{k,\lambda}(x+y) \le e^{(|\rho| + |\lambda|)|y|} F_{k, \lambda}(x).
\end{equation}
\end{lem}
\begin{proof}
From \cite[Lemma 3.4]{SchapiraHGF}, we have the estimate
\begin{equation}
    |\nabla F_{k, \lambda}(x)| \le (|\rho| + |\lambda|)F_{k, \lambda}(x),
\end{equation}
which implies (\ref{eqn:logF-lipschitz}) by Gr\"onwall's inequality.
\end{proof}

\begin{proof}[Proof of Proposition \ref{p:cont}]
Thanks to Lemma \ref{l:LipJ}, we have
\begin{align}\begin{split}
    \phantom{{}={}}\frac{1}{N^2}\left|\log J_{k,x}\left(y\right)-\log {J_{k,x}\left(y'\right)}\right|
    &\leq \frac{1}{N^2}\max_{w\in W}|\langle x,w(y-y')\rangle|\\
    &\leq \frac{1}{N^2 }\max_i |x_i| \sum_{i}|y_i-y_i'|\leq \max_i \left|\frac{x_i}{\sqrt{ N/\beta}}\right| W_1(\mu,\nu)\leq \delta,
\end{split}\end{align}
which gives \eqref{e:Jratio}. 
\end{proof}

We next determine the value of the constant term $\fC(\alpha, \beta)$ in Theorem \ref{main1} from the asymptotics of the normalization constant $C_{k,N}$ in \eqref{e:density}.
\begin{proposition}\label{p:coefficient}
    Under the assumptions of Theorem \ref{main1}, the constant $C_{k,N}$ in \eqref{e:density} satisfies
    \begin{align}
        \lim_{N\rightarrow \infty}\frac{1}{N^2}\log(C_{k,N})=\frac{\beta}{4}\left(3(\alpha+1) + \alpha^2\log \alpha -(\alpha+1)^2\log(\alpha+1) \right) =: \fC(\alpha,\beta),
    \end{align}
with the convention $0^2 \log 0 = 0$.
\end{proposition}
\begin{proof}[Proof of Proposition \ref{p:coefficient}]
    Recall that for the $B_N$ and $C_N$ root systems we have $C_k = c_k^{B}$ from \eqref{eqn:rational-c-B}, while for $D_N$ we have $C_k = c_k^{D}$ from \eqref{eqn:rational-c-D}. With $\beta$, $\alpha_N$, and $\alpha$ as defined in Theorem \ref{main1}, we can compute the asymptotics of $C_k$:
    \begin{align} \begin{split}\label{e:cBCk}
\frac{1}{N^2}\log(C_k) 
&=-\frac{\beta \alpha_N +\beta}{2}\log 2 -\frac{1}{N^2}\int_1^N \log \left(\left(\frac{\alpha_N \beta N+ x\beta}{2}\right)!\left(\frac{x\beta}{2}\right)!\right) \rd x+ \OO\left(\frac{\log N}{N}\right),
\end{split}\end{align}
where we understand $z! = \Gamma(1+z)$ when $z$ is not an integer.  Using Stirling's formula to simplify the integral term, we obtain 
\begin{align}\begin{split}\label{e:Nba}
&\phantom{{}={}}
-\frac{1}{N^2}\int_1^N \log \left(\left(\frac{\alpha_N \beta N+ x\beta}{2}\right)!\left(\frac{x\beta}{2}\right)!\right)\rd x\\
    &=-\frac{1}{N^2}\int_1^N \left(\frac{\alpha_N \beta N+ x\beta}{2}\right)\log \left(\frac{\alpha_N \beta N+ x\beta}{2e}\right)+ \left(\frac{x\beta}{2}\right)\log \left(\frac{x\beta}{2e}\right)\rd x + \OO(1/N) \\
    &=-\frac{2}{\beta}\left(\frac{\beta^2(\alpha_N+1)^2}{8}\log(\alpha_N+1)-\frac{(\alpha_N\beta )^2}{8}\log \alpha_N +\frac{\beta^2(\alpha_N+1)}{4}\left(\log \frac{N\beta}{2}-\frac{3}{2}\right)\right)+\OO(1/N)\\
    &=-\frac{\beta(\alpha_N+1)^2}{4}\log(\alpha_N+1)
    +\frac{\alpha_N^2\beta }{4}\log \alpha_N
    -\frac{\beta(\alpha_N+1)}{2}\left(\log \frac{N\beta}{2}-\frac{3}{2}\right)+\OO(1/N).
\end{split}\end{align}
Plugging \eqref{e:Nba} into \eqref{e:cBCk}, we find that $\log(C_k)/N^2$ equals
\begin{align}
-\frac{\beta(\alpha_N+1)^2}{4}\log(\alpha_N+1)
    +\frac{\alpha_N^2\beta }{4}\log \alpha_N
    -\frac{\beta(\alpha_N+1)}{2}\left(\log (N\beta)-\frac{3}{2}\right)+\OO\left(\frac{\log N}{N}\right).
\end{align}
Since $C_{k,N}=C_k(\beta N)^{(\beta/2)(\alpha_N N^2+N(N-1))}$, we conclude that
\begin{align*}
    \lim_{N\rightarrow \infty}\frac{1}{N^2}\log(C_{k,N})
    &=\lim_{N\rightarrow \infty}\frac{\beta}{4}\left(3(\alpha_N+1) + \alpha_N^2\log \alpha_N - (\alpha_N+1)^2\log(\alpha_N+1) \right)\\
   &=\frac{\beta}{4}\left(3(\alpha+1) + \alpha^2\log \alpha -(\alpha+1)^2\log(\alpha+1) \right).
\end{align*}
This finishes the proof of Proposition \ref{p:coefficient}.
\end{proof}

\begin{proof}[Proof of Theorem \ref{main1}]
For $x, y$ as in the statement of Theorem \ref{main1}, we set $a = x/\sqrt{\beta N}$, $b = y/\sqrt{\beta N}$. By the $W$-invariance of the generalized Bessel function in both arguments, it is sufficient to assume that $x, y \in \mathcal{C}^+_N$, so that $a$ and $b$ lie in the domain of the radial Dunkl process.  We recall that, from  the density formula \eqref{e:density}, the transition probability is given by 
\begin{align}\label{e:density2} 
p_1(a,b) =C_{k,N} e^{-\beta N(|a|^2 + |b|^2)/2} J_{k, \sqrt{\beta N}a} (\sqrt{\beta N}b)   
\prod_i b_i^{N\beta \alpha_N}
\prod_{i<j}(b_i^2-b_j^2)^{\beta}.
\end{align}

The large deviations principle for the Dyson Bessel process gives 
\begin{align*}
\lim_{N\rightarrow \infty}\frac{1}{N^2}\log \bP(\widehat\nu^N_{B} \in \bB(\widehat\nu_B,\delta))
=\inf_{\widehat\nu_{1}=\widehat\nu_{B}}S_{\widehat\nu_A}^{\al}(\{\widehat\nu_{t}\}_{0\leq t\leq 1})+\oo_\delta(1),
\end{align*}
where $\oo_\delta(1)$ goes to zero as $\delta$ goes to zero.
By integrating \eqref{e:density2} over the ball $\bB(\widehat\nu_B,\delta)$, we have
\begin{align*}
\phantom{{}={}}\int_{\widehat\nu^N_{B}\in \bB(\widehat\nu_B,\delta)} &C_{k,N} e^{-\beta N(|a|^2 + |b|^2)/2} J_{k, \sqrt{\beta N}a} (\sqrt{\beta N}b)   
\prod_i b_i^{N\beta \alpha_N}
\prod_{i<j}(b_i^2-b_j^2)^{\beta}\prod \rd b_i\\
&=C_{k,N} 
e^{\frac{\beta N^2}{2} (2\Sigma(\widehat\nu_B)-(\widehat\nu_A( x^2) +\widehat\nu_B(x^2-2\alpha \log |x|))+\oo_\delta(1))}
\int_{\widehat\nu^N_{B}\in \bB(\widehat\nu_B,\delta)} \int J_{k, \sqrt{\beta N}a} (\sqrt{\beta N}b)\prod \rd b_i.
\end{align*}
Even though the logarithm is singular, $ \int \log |x|\rd \widehat\nu^N_B$ is close to $\int \log |x|\rd \widehat\nu_B$ on the ball (and similarly for the non-commutative entropy term), which can be rigorously justified using the same techniques as in \cite{belinschi2022large}. We omit the details. 

Thanks to Proposition \ref{p:cont}, we can replace the integral over the ball $\bB(\widehat\nu_B,\delta)$ by $J_{k, \sqrt{\beta N} a} (\sqrt{\beta N} b)$ with an error $e^{O(\delta N^2)}$.
By rearranging, we obtain the following asymptotics of the generalized Bessel function:
\begin{align}\begin{split}\label{e:sphi}
\phantom{{}={}}\lim_{N\rightarrow \infty}\frac{1}{N^{2}}\log J_{k, \sqrt{\beta N}a} (\sqrt{\beta N}b)
= &-\inf_{\nu_{1}=\widehat\nu_{B}}S_{\mu_A}^{\al}\big(\{\nu_{t}\}_{0\leq t\leq 1} \big)\\
&-\frac{\beta}{2}\left[ 2\Sigma(\widehat\nu_B)-\big(\widehat\nu_{A}(x^{2})+\widehat\nu_{B}(x^{2})-2\alpha \log |x|)\big) \right]-\fC(\alpha,\beta),
\end{split}\end{align}
where the constant $\fC(\alpha, \beta)$ is determined in Proposition \ref{p:coefficient}.
If $S_{\widehat \nu_A}^{\al}(\{\widehat\nu_{t}\}_{0\leq t\leq 1})<\infty$, then $\widehat\nu_t$ has a density, i.e., $\{\widehat\nu_t\}_{0\leq t\leq 1}=\{\widehat\rho_t(x)\rd x\}_{0\leq t\leq 1}$ is a symmetric measure-valued process satisfying the weak limits (\ref{e:bbterm}):
\begin{align*}
\lim_{t\rightarrow 0}\widehat\rho_t(x)\rd x=\widehat\nu_A,\qquad
\lim_{t\rightarrow 1}\widehat\rho_t(x)\rd x=\widehat\nu_B.
\end{align*}
Let $u_s$ be the weak solution of the conservation of mass equation (\ref{e:cmass}):
\begin{align*}
\del_s\widehat \rho_s+\del_x(\widehat \rho_s u_s)=0.
\end{align*}
We recall the following formula for the dynamical entropy $S_{\widehat\nu_A}^{\al}(\{\widehat\nu_{t}\}_{0\leq t\leq 1})$ from \cite[Proposition 4.1]{guionnet2021large}: 
\begin{align}\begin{split}\label{e:largeupb3}
S_{\mu_0}^{\al}(\{\widehat\nu_t\}_{0\leq t\leq 1})&=
\frac{\beta}{2}\Bigg(\int_0^1 \int u_s^2  \widehat\rho_s(x)\rd x \rd s+\frac{\pi^2}{3}\int_0^1\int \widehat \rho^3_s(x) \rd x  \rd s
+\frac{\alpha^2}{4}\int \frac{\widehat\rho_s(x)}{x^2}\rd x\rd s\\
&\quad -\left.\left(\Sigma(\widehat\nu_t)+\alpha\int \log|x| \widehat\rho_t(x)\rd x\right)\right|_{t=0}^1 \Bigg).
\end{split}\end{align}

Plugging \eqref{e:largeupb3} into \eqref{e:sphi}, we obtain the desired result on the asymptotics of the generalized Bessel function,
\begin{align}\begin{split}\label{e:aratecopy}
I(\widehat \nu_{A},\widehat\nu_{B}) &= \lim_{N \to \infty} \frac{1}{N^2} \log J_{k,\sqrt{\beta N}a}(\sqrt{\beta N}b) \\
&= -\frac{\beta}{2}\inf_{\{\widehat\rho_t\}_{0<t<1} \atop \text{satisfies \eqref{e:bbterm}}}\left\{\int_0^1 \int u_s^2  \widehat\rho_s(x)\rd x \rd s+\frac{\pi^2}{3}\int_0^1\int \widehat \rho^3_s(x) \rd x  \rd s
+\frac{\alpha^2}{4}\int \frac{\widehat\rho_s(x)}{x^2}\rd x\rd s
\right\}\\
& \quad +\frac{\beta}{2}\big[\widehat\nu_{A}(x^{2}-\alpha \log |x|)+\widehat\nu_{B}(x^{2}-\alpha \log |x|)- \big(\Sigma(\widehat\nu_A)+\Sigma(\widehat\nu_B) \big) \big] -\fC(\alpha, \beta).
\end{split}\end{align}
This finishes the proof of Theorem \ref{main1}.
\end{proof}

\begin{remark} \label{rem:typeAlimit}
    We have stated Theorem \ref{main1} for the root systems of types $B$, $C$ and $D$, because the statement for type $A$ takes a slightly different form and is straightforward to infer from the work of Guionnet and Zeitouni \cite{GZ, GZ-addendum}, who studied the asymptotics of spherical integrals of the form (\ref{eqn:iz-int}) via the large deviations of Dyson Brownian motion. For completeness, we give the statement for type $A$ in Theorem \ref{main1-A} below. The formula for the limit is similar to the formula for type $D$, but with some important differences: the empirical measures of the arguments to the Bessel function are scaled by $(\beta N / 2)^{-1/2}$ rather than $(\beta N)^{-1/2}$ and are not symmetrized, and the limiting functional differs by a factor of 2.  We omit the proof, as it uses a similar technique to the proof of Theorem \ref{main1} and can be obtained almost immediately from the proofs in \cite{GZ, GZ-addendum} by relating the spherical integrals in those papers to generalized Bessel functions as in Example \ref{ex:iz-int} and letting the parameter $\beta = 2 k_{\sqrt{2}}$ take values in $[1, \infty)$.
    \end{remark}

\begin{thm} \label{main1-A}
    Let $\Phi = A_{N-1}$, let $k_{\sqrt{2}} \ge 1/2$, and write $\beta = 2 k_{\sqrt{2}}$. For each $N = 1, 2, \hdots,$ fix vectors
\begin{align*}
x &= x^{(N)} = (x^{(N)}_1, x^{(N)}_2, \cdots, x^{(N)}_N ), \\
y &= y^{(N)} = (y^{(N)}_1, y^{(N)}_2, \cdots, y^{(N)}_N ) \in \R^N
\end{align*}
satisfying $\sum_i x_i = \sum_i y_i = 0$, such that the scaled, \emph{unsymmetrized} empirical measures converge weakly as $N \to \infty$:
\begin{align} \begin{split} \label{eqn:ssem-A}
    \nu_{A}^{N} = \frac{1}{N}\sum_{i}\delta_{\sqrt{\frac{2}{\beta N}}x^{(N)}_i} \rightarrow \nu_A,\\
    \nu_{B}^{N} = \frac{1}{N}\sum_{i}\delta_{\sqrt{\frac{2}{\beta N}}y^{(N)}_i} \rightarrow \nu_B.
\end{split} \end{align}
Assume that for $C=A$ or $B$, we have  $\sup_{N} \nu_{C}^{N}(x^{2})<\infty$, $\Sigma(\nu_{C})>-\infty$. Then the  following limit of the generalized Bessel function exists:
\begin{equation} \label{e:Jasymp-A}
\lim_{N \to \infty}\frac{1}{N^{2}}\log J_{k,x}(y)= \frac{1}{2} I_{0,\beta}(\nu_{A},\nu_{B}),
\end{equation}
where $I_{0,\beta}$ is as defined in (\ref{e:arate}), with the infimum taken over continuous measure-valued processes $\{ \rho_t(x) \rd x \}_{0<t<1}$, which are \emph{not} assumed to be symmetric with respect to reflection through $x=0$, satisfying the weak limits
\[
\lim_{t \to 0^+} \rho_t(x) \rd x = \nu_A, \qquad \lim_{t \to 1^-} \rho_t(x) \rd x = \nu_B.
\]
\end{thm}

\subsection{Large deviations of radial Heckman--Opdam processes}
\label{sec:hyp-LDP}

In this section we turn to the hyperbolic theory and study the large-$N$ hydrodynamics of the radial Heckman--Opdam process. More specifically, we prove a large deviations principle for a modified version of the Dyson Bessel process obtained as a reparametrization of the radial Heckman--Opdam process with $\Phi=BC_N$.  Corresponding large deviations principles for radial Heckman--Opdam processes associated with the classical root systems are easily deduced from the large deviations of the modified Dyson Bessel process by setting one or both of the parameters $\gamma,\delta$ in (\ref{eqn:gamma-delta-lims}) to 0 and, in the case of type $A$, by considering unsymmetrized measures as in Theorem \ref{main1-A}.

We make the following change of variables in the SDE (\ref{eqn:hyperbolic-sde-BC}) for the radial Heckman--Opdam process of type $BC$, for $1\leq i\leq N$:
\begin{equation} \label{eqn:hyp-cov}
Y_i(t) =  s_i(\beta N t), \qquad
k_{\sqrt{2}}= \frac{\beta}{4},
\end{equation}
which gives
\begin{align} \label{e:hyper}
    \rd s_i(t) = \frac{\rd B_{i}}{\sqrt{\beta N}} 
    &+\left(\frac{k_{\sqrt 2}}{\beta N}\sum_{j: j \neq i}\coth\bigg(\frac{s_i(t)-s_j(t)}{2}\bigg)+\frac{k_{\sqrt 2}}{\beta N}\sum_{j: j\neq i}\coth\bigg(\frac{s_i(t)+s_j(t)}{2}\bigg)\right)\rd t\\
\nonumber &+\left(\frac{k_{2}}{\beta N}\coth \big( s_i(t) \big)+\frac{k_{1}}{2\beta N}\coth\left( \frac{s_i(t)}{2} \right)\right)\rd t
\end{align}
where $B_1, B_2,\cdots, B_N$ are independent Brownian motions. Note that (\ref{eqn:hyp-cov}) is not the same as the reparametrization used for the radial Dunkl process in (\ref{eqn:rat-cov}), though we will ultimately obtain the large deviations of the radial Heckman--Opdam process via a slightly more complicated relation with the same Dyson Bessel process as in the rational case. Concretely, we can rewrite \eqref{e:hyper} as
\begin{align}\begin{split}\label{e:DBPcopy}
\rd s_i(t) 
=\frac{\rd B_{i}}{\sqrt{\beta N}}&+\left(\frac{1}{2N}\sum_{j: j \neq i}\frac{1}{s_i(t)-s_j(t)}+\frac{1}{2N}\sum_{j: j\neq i}\frac{1}{s_i(t)+s_j(t)}+\frac{\al_N}{2 s_i(t)}\right)\rd t\\
&+\frac{1}{2N} \sum_{j:j\neq i} \Big( V'(s_i(t)-s_j(t))+V'(s_i(t)+s_j(t))
+W_N'(s_i(t)) \Big) \rd t,
\end{split}\end{align}
where $\alpha_N=2(k_1+k_2)/(\beta N)$, and $V'$ and $W_N'$ are smooth functions. Explicitly, they are given by 
\begin{align}
    &V'(x)=\frac{1}{2}\coth\left(\frac{x}{2}\right)-\frac{1}{x},\\
    &W_N'(x)=\frac{k_2}{\beta N}\coth(x)+\frac{k_1}{2\beta N}\coth\left(\frac{x}{2}\right)-\frac{\alpha_N}{2x}.
\end{align} 
We call the process \eqref{e:DBPcopy} the {\it modified Dyson Bessel process}.  We prove that it can be obtained from the Dyson Bessel process  by a change of measure using an exponential martingale constructed from the following function:
\begin{align}\label{e:theta}
\theta_N( s_1,  s_2,\cdots,  s_N)=\frac{\beta}{2}\left(\sum_{i<j} \big[ V(s_i-s_j)+V(s_i+s_j) \big] +2N\sum_i W_N(s_i) \right).
\end{align} 
As a corollary, we obtain the large deviations of \eqref{e:DBPcopy}.

In the following we fix $k_{\sqrt{2}} \ge 1/2$ for all $N$, and we assume that $k_1 = k_1(N)$, $k_2 = k_2(N)$ are chosen so that the limits
\begin{align} \label{eqn:gamma-delta-lims}
    \gamma=\lim_{N\rightarrow \infty} \frac{k_1}{\beta N},\quad 
    \delta=\lim_{N\rightarrow \infty} \frac{k_2}{\beta N}
\end{align}
exist. Then $\alpha_N$ converges to $\alpha=2(\delta+\gamma)$, and $W_N(x)$ converges to
\begin{align}
W(x)=\gamma \left(\coth\left(x\right)-\frac{1}{x}\right)
+\delta\left(\frac{1}{2}\coth\left(\frac{x}{2}\right)-\frac{1}{x}\right).
\end{align}

\begin{proposition}\label{p:changem}
Let $\cF_{t}$ be the $\sigma$-algebra generated by the Brownian motions $\{B_i(t)\}$. 
Let $\bQ$ be the law of the Dyson Bessel process \eqref{e:DBP},
and $\bP$ the law of the modified Dyson Bessel process \eqref{e:DBPcopy}. Suppose that the limits $\gamma, \delta$ in (\ref{eqn:gamma-delta-lims}) exist. 
Then the two laws $\bP$ and $\bQ$ are related by a change of measure
\begin{align*}
\bP =e^{{L_{1  }-\frac{1}{2}\langle L,L\rangle_{1  }}} \bQ,
\end{align*}
where the exponent is given by
\begin{align*}\begin{split}
L_{1  }-\frac{1}{2}\langle L, L\rangle_{1  }
&=-N^2 (F_{V,W}(\{\widehat\nu_t\}_{0\leq t\leq 1}+\oo(1)),
\end{split}\end{align*}
and the functional $F_{V,W}$ takes the explicit form 
\begin{align}\begin{split}\label{e:FVW}
F_{V,W}(\widehat \nu_t)
=
&\frac{\beta }{2}\Bigg[\left.\left(2\int V(x-y) \, \rd \widehat\nu_t(x) \, \rd\widehat \nu_t(y)+4 \int W(x) \, \rd \widehat \nu_t(x)\right)\right|_{t=0}^{t=1}
\\
&-\int_0^1\rd t\int \frac{V'(x-z)-V'(x-z) + W'(x)-W'(y)}{x-y} \, \rd \widehat \nu_t(x) \, \rd \widehat \nu_t(y) \, \rd \widehat \nu_t(z)-\al \int \frac{W'(x)}{x}\rd \widehat \nu_t(x)\Bigg]\\
&-2\beta \int \big( V'(x-y)V'(x-z)+2V'(x-y)W'(x)+W'(x)^2 \big) \, \rd\widehat \nu_1(x) \, \rd\widehat \nu_1(y) \, \rd\widehat\nu_1(z).
\end{split}\end{align}
\end{proposition}

\begin{cor} \label{cor:HOP-asymp}
Fix a  symmetric probability measure  $\widehat\mu_0$  and a sequence of initial conditions with symmetrized empirical measures $\widehat\nu^N_0$ having uniformly bounded second moment and converging weakly to $\widehat\mu_{0}$.  Suppose that the limits $\gamma, \delta$ in (\ref{eqn:gamma-delta-lims}) exist.  Then, if $\alpha_N$ converges towards $\alpha\in [0,\infty)$ when $N$ goes to infinity {{so that either $\alpha_N\ge 1/\beta N$ or $\alpha_N\equiv 0$}},   the distribution of   the empirical particle density $\{\widehat\nu_t^N\}_{0\leq t\leq 1}$ of the modified Dyson Bessel process \eqref{e:DBPcopy} satisfies a   large deviations principle in the scale $N^2$ and with good rate function $S_{\widehat\nu_0}^\al$. In particular, 
for any continuous symmetric measure-valued  process $\{\widehat\nu_t\}_{0\leq t\leq 1}$, we have: 
\begin{align}\begin{split}\label{e:ulbb-hyp}
\phantom{{}={}}\lim_{\delta\rightarrow 0} & \liminf_{N\rightarrow\infty}\frac{1}{N^2}\log \bP(\{\widehat\nu^N_t\}_{0\leq t\leq 1}\in \bB(\{\widehat\nu_t\}_{0\leq t\leq 1}, \delta))\\
&=\lim_{\delta\rightarrow 0}\limsup_{N\rightarrow\infty}\frac{1}{N^2}\log \bP(\{\widehat\nu^N_t\}_{0\leq t\leq 1}\in \bB(\{\widehat\nu_t\}_{0\leq t\leq 1}, \delta))\\
&= -{{S^\al_{\widehat\mu_0}(\{\widehat\nu_t\}_{0\leq t\leq 1})}}-F_{V,W}(\{\widehat\nu_t\}_{0\leq t\leq 1}).
\end{split}\end{align}
\end{cor}

\begin{proof}[Proof of Proposition \ref{p:changem}]
The first and second derivatives of $\theta_N$ are given by
\begin{align}\begin{split}\label{e:dtheta}
\del_{ s_i}\theta_N( s_1,  s_2,\cdots,  s_N)&=\frac{\beta}{2}\left(\sum_{j:j\neq i}(V'(s_i-s_j)+V'(s_i+s_j))+ 2N W_N'(s_i)\right),\\
\del^2_{ s_i}\theta_N( s_1,  s_2,\cdots,  s_N)&=\frac{\beta}{2}\left(\sum_{j:j\neq i}(V''(s_i-s_j)+V''(s_i+s_j))+ 2N W_N''(s_i)\right),
\end{split}\end{align}
for $1\leq i\leq N$, where we have used the fact that $V'$ is an odd function.
Since $\theta_N$ is $C^\infty$, It{\^o}'s lemma gives
\begin{align}\label{e:dL}\begin{split}
\phantom{{}={}}\rd \theta_N(s_1(t), s_2(t), &\cdots, s_N(t))
=\rd L_t
+\Bigg[ \sum_i\frac{\del^2_{ s_i}\theta_N}{2\beta N}\\
&+
\sum_{i} \del_{s_i}\theta_N \Bigg(\frac{1}{2N}\sum_{j: j \neq i}\frac{1}{s_i(t)-s_j(t)}+\frac{1}{2N}\sum_{j: j\neq i}\frac{1}{s_i(t)+s_j(t)}+\frac{\al_N}{2 s_i(t)}\Bigg) \Bigg] \rd t
\end{split}\end{align}
where the martingale term $L_t$ is 
\begin{align*}
\rd L_t=\sum_{i}\del_{ s_i}\theta_N\frac{\rd B_i(t)}{\sqrt{\beta N}}
=\sum_i \frac{1}{2}\sqrt{\frac{\beta}{N}}\left(\sum_{j:j\neq i}(V'(s_i-s_j)+V'(s_i+s_j))+ 2N W_N'(s_i)\right)\rd B_i(t).
\end{align*}
Its quadratic variation is given by
\begin{align}\label{e:LL}\begin{split}
\langle L, L\rangle_t &=\frac{\beta}{4N}\int^t_0 \sum_i \left(\sum_{j:j\neq i}\big( V'(s_i-s_j)+V'(s_i+s_j) \big)+ 2N W_N'(s_i)\right)^2\rd u\\
&=4\beta N^2\int_0^t \int \big( V'(x-y)V'(x-z)+2V'(x-y)W'(x)+W'(x)^2 \big) \, \rd\widehat \nu_u(x) \, \rd \widehat \nu_u(y) \, \rd \widehat\nu_u(z) \, \rd u+\oo(N^2).
\end{split}\end{align}

For the last term on the righthand side of \eqref{e:dL}, using \eqref{e:dtheta}, it is given explicitly by
\begin{align}\begin{split}\label{e:term1}
&\frac{\beta}{2}\frac{1}{4N}\left(\sum_{i\neq j}\sum_{k\neq i,j}\frac{V'(s_i-s_k)+V'(s_i+s_k)-V'(s_j-s_k)-V'(s_i+s_k)}{s_i-s_j}\right.\\
&\qquad \left. +
\sum_{i\neq j}\sum_{k\neq i,j}\frac{V'(s_i-s_k)+V'(s_i+s_k)+V'(s_j-s_k)+V'(s_i+s_k)}{s_i+s_j}
\right)\\
&\qquad +\frac{\beta}{4}\sum_{i\neq j}\left(\frac{W_N'(s_i)-W_N'(s_j)}{s_i-s_j}+\frac{W_N'(s_i)+W_N'(s_j)}{s_i+s_j}\right)\\
&\qquad +\al_N N\frac{\beta}{2} \sum_i\frac{W_N'(s_i)}{s_i}+\frac{\beta}{2}\frac{1}{2N}\sum_{i\neq j}\frac{V'(s_i-s_j)}{s_i-s_j}+\frac{V'(s_i+s_j)}{s_i+s_j}\\
&=\frac{\beta N^2}{2}\left(\int \frac{V'(x-z)-V'(x-z) + W'(x)-W'(y)}{x-y} \, \rd \widehat \nu_t(x) \, \rd \widehat \nu_t(y) \, \rd \widehat \nu_t(z)+\al \int \frac{W'(x)}{x} \, \rd \widehat \nu_t(x)+\oo(1)\right).
\end{split}\end{align}
For the sum over second derivatives of $\theta_N$ in \eqref{e:dL}, using \eqref{e:dtheta} we have
\begin{align}\label{e:term2}
\sum_i\frac{\del^2_{ s_i}\theta_N}{2\beta N} = \frac{1}{4N}\sum_{i\neq j} \big( V''(s_i-s_j)+V''(s_i+s_j) \big)+ 2N W_N''(s_i)=\OO(N).
\end{align}
By plugging \eqref{e:term1} and \eqref{e:term2} back into \eqref{e:dL}, we get
\begin{align}\begin{split}\label{e:df}
\rd \theta_N =\rd L_t
&+\frac{\beta N^2}{2}\left(\int \frac{V'(x-z)-V'(y-z) + W'(x)-W'(y)}{x-y} \, \rd \widehat \nu_t(x) \, \rd \widehat \nu_t(y) \, \rd \widehat \nu_t(z)\right.\\
& \left.+\al\int \frac{W'(x)}{x} \, \rd \widehat \nu_t(x)+\oo(1)\right) \rd t.
\end{split}\end{align}
From the definition of $\theta_N$ in  \eqref{e:theta}, we have
\begin{align}\label{e:theta-int}
\theta_N( s_1(t),  s_2(t),\cdots,  s_N(t))
=\frac{\beta N^2}{2}\left(2\int V(x-y) \, \rd \widehat\nu_t(x) \, \rd\widehat \nu_t(y)+4 \int W(x) \, \rd \widehat \nu_t(x)+\oo(1)\right).
\end{align} 
Integrating both sides of \eqref{e:df} from $0$ to $1$, we get
\begin{align}\begin{split}\label{e:L1}
L_1=\frac{\beta N^2}{2} &\Bigg(\left.\left(2\int V(x-y) \, \rd \widehat\nu_t(x)\, \rd\widehat \nu_t(y)+4 \int W(x) \, \rd \widehat \nu_t(x)\right)\right|_{t=0}^{t=1}
\\
&\qquad -\int \frac{V'(x-z)-V'(x-z) + W'(x)-W'(y)}{x-y} \, \rd \widehat \nu_t(x) \, \rd \widehat \nu_t(y) \, \rd \widehat \nu_t(z)\\
&\qquad -\al \int \frac{W'(x)}{x} \, \rd \widehat \nu_t(x)+\oo(1)\Bigg).
\end{split}\end{align}

Novikov's theorem \cite[H.10]{AGZ} implies that the process
\begin{align*}
e^{{L_{t  }-\frac{1}{2}\langle L,L\rangle_{t  }}}
\end{align*}
is an exponential martingale.  More explicitly, using \eqref{e:df} and \eqref{e:L1},
we can rewrite 
\begin{align*}
L_{1  }-\frac{1}{2}\langle L, L\rangle_{1  }
&=N^2 (F_{V,W}(\widehat\nu_t)+\oo(1)),
\end{align*}
where the functional $F_{V,W}$ is given in \eqref{e:FVW}.

We recall that $\bQ$ is the law of the Dyson Bessel process \eqref{e:DBP}, and denote the rescaled Brownian motions $M_1, \hdots, M_N$: 
\begin{align*}
M_i(t)&= s_i(t)- s_i(0)-\int_0^t\left(\frac{1}{2N}\sum_{j: j \neq i}\frac{1}{s_i(u)-s_j(u)}+\frac{1}{2N}\sum_{j: j\neq i}\frac{1}{s_i(u)+s_j(u)}+\frac{\al_N}{2 s_i(u)}\right)\rd u\\&=\int_0^t \frac{\rd B_i(u)}{\sqrt{\beta N}}=\frac{B_i(t)}{\sqrt{\beta N}}.
\end{align*}
Then Girsanov's theorem \cite[Theorem H.11]{AGZ} implies that
\begin{align}\begin{split}\label{e:newM}
M_i(t)-\langle M_i, L\rangle_{t  }= s_i(t) &- s_i(0)-\int_0^t\left(\frac{1}{2N}\sum_{j: j \neq i}\frac{1}{s_i(u)-s_j(u)}+\frac{1}{2N}\sum_{j: j\neq i}\frac{1}{s_i(u)+s_j(u)}+\frac{\al_N}{2 s_i(u)}\right) \rd u \\
&-\int_0^{t  }\left(\frac{1}{2N}\sum_{j:j\neq i}V'(s_i(u)-s_j(u))+V'(s_i(u)+s_j(u))
+W_N'(s_i(u))\right)\rd u,
\end{split}\end{align}
are independent Brownian motions under the measure
\begin{align*}
\bP =e^{{L_{1  }-\frac{1}{2}\langle L,L\rangle_{1  }}} \bQ,
\end{align*}
with quadratic variations given by
\begin{align*}
\langle \rd(M_i(t)-\langle M_i, L\rangle_{t  }), \rd(M_j(t)-\langle M_j, L\rangle_{t  })\rangle =\frac{\delta_{i=j}}{\beta N}.
\end{align*}
Therefore, under the new measure $\bP $, we find that $s_i(t)$ satisfies \eqref{e:DBPcopy}.
\end{proof}

\begin{remark}
    While Corollary \ref{cor:HOP-asymp} describes the large deviations of the radial Heckman--Opdam process, it does not immediately allow us to obtain the large-$N$ asymptotics for the Heckman--Opdam hypergeometric function in the same way that Theorem \ref{main1} gives the large-$N$ asymptotics of the generalized Bessel function.  The obstacle is the fact that, in contrast to the rational case where the transition function \eqref{eqn:rational-transition} for the radial Dunkl process is expressed very directly in terms of the generalized Bessel function, for the transition kernel \eqref{eqn:hyperbolic-transition} of the radial Heckman--Opdam process we have only a less explicit  expression as an integral transform of the hypergeometric function.
    
    The transition kernel $q_t(x,y)$ in \eqref{eqn:hyperbolic-transition} is known as the Heckman--Opdam heat kernel, and its asymptotic behavior is of independent interest.  Given a continuity estimate for the heat kernel analogous to Proposition \ref{p:cont}, Corollary \ref{cor:HOP-asymp} would yield a formula for the leading-order contribution to the large-$N$ limit of $q_t(x,y)$.  However, more work would still be required to obtain the large-$N$ behavior of $F_{k,\lambda}$ from such a result.
\end{remark}

\section{Preliminaries on weight multiplicities}
\label{sec:mult-prelims}

Let $\mathbb{Y}_N$ denote the set of Young diagrams with $N$ rows. For $\lambda \in \mathbb{Y}_N$, we write $\lambda_i$ for the number of boxes in the $i$th row of $\lambda$, so that we obtain an identification of $\mathbb{Y}_N$ with the set of vectors $(\lambda_1, \hdots, \lambda_N) \in \Z^N$ satisfying $\lambda_1 \ge \hdots \ge \lambda_N \ge 0$.

The Schur polynomial $s_\lambda$, $\lambda \in \mathbb{Y}_N$ can be decomposed on the basis of monomial symmetric polynomials $m_\mu,$ $\mu \in \mathbb{Y}_N$. The {\it Kostka number} $K_{\lambda \mu}$ is defined as the coefficient of $m_\mu$ in $s_\lambda$:
\begin{equation} \label{eqn:schur-monom}
s_\lambda = \sum_{\mu \in \mathbb{Y}_N} K_{\lambda \mu} \, m_\mu.
\end{equation}

Let $V_\lambda$ be the irreducible polynomial representation of $U(N)$ with highest weight $\lambda$. Then the Schur polynomial $s_\lambda$ expresses the character of $V_\lambda$, while monomial symmetric polynomials express the characters of a maximal torus $T \subset U(N)$. Accordingly, the Kostka numbers also give the multiplicities that arise in the weight space decomposition,
\begin{equation} \label{eqn:weight-decomp-A}
V_\lambda \cong \bigoplus_{\mu \in \mathbb{Y}_N} W_\mu^{\oplus K_{\lambda \mu}},
\end{equation}
where $W_\mu$ is the one-dimensional irreducible representation of $T$ with weight $\mu$.

In \cite{belinschi2022large}, Belinschi, Guionnet and the first author proved a large deviations principle for Kostka numbers $K_{\bmla_N, \bmmu_N}$, for suitable sequences of pairs of young diagrams $\bmla_N, \, \bmmu_N \in \mathbb{Y}_N$ as $N \to \infty$.  In Section \ref{sec:weight-asymp} below, we generalize the results of \cite{belinschi2022large} to show a large deviations principle for the weight multiplicities of irreducible representations of arbitrary compact Lie groups.

The key identity used to prove the LDP in \cite{belinschi2022large} is the following formula that expresses the Schur polynomial $s_\lambda$ as an integral over the group $\U(N)$.  Write
\begin{align*}
D_\lambda &= \mathrm{diag} \left( \frac{\lambda_1 + N - 1}{N}, \frac{\lambda_2 + N - 2}{N}, \cdots, \frac{\lambda_N}{N}\right), \qquad \lambda \in \mathbb{Y}_N, \\
Y &= \diag(y_1, \hdots, y_N), \qquad y \in \R^N.
\end{align*}
Then
\begin{equation} \label{eqn:schur-integral}
s_\lambda(e^{y_1}, \hdots, e^{y_N}) = \left( \prod_{i=2}^{N-1} i! \right)^{-1} \frac{\prod_{i<j} (y_i - y_j)(\lambda_i - \lambda_j - i + j)}{\prod_{i<j} (e^{y_i} - e^{y_j})} \int_{\U(N)} e^{N \Tr(YUD_\lambda U^*)} \rd U,
\end{equation}
where $\rd U$ is the Haar probability measure on $\U(N)$.  In fact, (\ref{eqn:schur-integral}) is a special case of a more general identity in representation theory known as the {\it Kirillov character formula}, which we now recall for the case of compact Lie groups.  For a thorough treatment of the Kirillov character formula, we refer the reader to the book \cite{AK}.

Let $G$ be a compact Lie group with Lie algebra $\gog$. In what follows, we assume that $G$ is simple, and moreover that $G$ is simply connected, so that the irreducible representations of $G$ and $\gog$ are in correspondence. For the purpose of studying the representation theory of $G$, these assumptions do not involve any meaningful loss of generality. Let $\tot \subset \gog$ be a Cartan subalgebra and $T = \exp(\tot) \subset G$ the corresponding maximal torus.  We fix an $\mathrm{Ad}$-invariant inner product $\langle \cdot, \cdot \rangle$ on $\gog$, which we use to identify $\gog \cong \gog^*$ and $\tot \cong \tot^*$.  Let $\Phi$ be the roots of $\gog$ with respect to $\tot$, which we take to be {\it real}-valued linear functionals on $\tot$, $\Phi^+ \subset \Phi$ a fixed system of positive roots, and $\mathcal{C}^+ \subset \tot$ the (closed) positive Weyl chamber.  A positive root $\alpha \in \Phi^+$ is a {\it simple root} if it cannot be written as a sum of two elements of $\Phi^+$; the simple roots form a basis of $\tot$.  The {\it fundamental weights} $\omega_1, \hdots, \omega_r$, where $r = \dim \tot$ is the rank of $G$, are defined by
\begin{equation} \label{eqn:fundweights-def}
  2 \frac{\langle \omega_i, \alpha_j \rangle}{\langle \alpha_j, \alpha_j \rangle} = \delta_{i,j},
\end{equation}
where $\alpha_1, \hdots, \alpha_r$ are the simple roots and $\delta_{i,j}$ is the Kronecker delta. The set $P^+$ of {\it dominant integral weights} of $\gog$ consists of all $\lambda \in \tot$ that can be written as a non-negative integer combination of the fundamental weights.

Via the identification $\tot \cong \tot^*$, we may identify $\Phi$ with a collection of vectors in $\tot$.  We then define
\[
\Delta_\gog(x) = \prod_{\alpha \in \Phi^+} \langle \alpha, x \rangle, \qquad \widehat{\Delta}_\gog(x) = \prod_{\alpha \in \Phi^+} (e^{i \langle \alpha, x\rangle / 2} - e^{-i \langle \alpha, x \rangle / 2}), \qquad x \in \tot.
\]
Let $\rho = \frac{1}{2} \sum_{\alpha \in \Phi^+} \alpha \in \tot$. For $\lambda \in \tot$ a dominant integral weight, let $\chi_\lambda$ be the character of the irreducible representation with highest weight $\lambda$.  We regard $\chi_\lambda$ as a function on the group $G$. The Kirillov character formula then states:
\begin{equation} \label{eqn:KCF}
    \chi_\lambda(e^x) = \frac{\Delta_\gog(ix)}{\widehat{\Delta}_\gog(x)} \frac{\Delta_\gog(\lambda + \rho)}{\Delta_\gog(\rho)} \int_G e^{i \langle \mathrm{Ad}_g(\lambda + \rho), x \rangle} \rd g,
\end{equation}
for $x \in \tot$ with $\Delta_\gog(x) \ne 0$, where $\rd g$ is the Haar probability measure on $G$ and $e^x$ is the Lie exponential of $x$.

The formula (\ref{eqn:KCF}) relates the character theory of $G$ to the rational Dunkl theory described above in Section \ref{sec:rational-theory}.  As explained in Example \ref{ex:GBF-HC-int}, the integral in (\ref{eqn:KCF}) is precisely the generalized Bessel function associated to the root system $\Phi$ on $\tot$, with multiplicity parameter $k \equiv 1$:
\[
  \int_G e^{i \langle \mathrm{Ad}_g(\lambda + \rho), x \rangle} \rd g = J_{\vec 1s, \lambda+\rho}(ix).
\]

We record for later use the following explicit formulae for $\Delta_\gog$,
$\widehat{\Delta}_\gog$ and $
\rho$ when $\Phi$ is one of the classical root systems, following the conventions of Section \ref{sec:root-systems}.

\begin{equation} \label{eqn:delta-cases}
\Delta_\gog(x) = \begin{cases}
  \prod_{1 \le i < j \le N} (x_i - x_j), & \Phi = A_{N-1}, \\
  \prod_{1 \le i < j \le N} (x_i - x_j)(x_i + x_j), & \Phi = D_{N}, \\
  \prod_{1 \le i < j \le N} (x_i - x_j)(x_i + x_j) \prod_{k=1}^N x_k, & \Phi = B_{N}, \\
  2^N \prod_{1 \le i < j \le N} (x_i - x_j)(x_i + x_j) \prod_{k=1}^N x_k, & \Phi = C_{N}.
\end{cases}
\end{equation}

\begin{equation} \label{eqn:deltahat-cases}
\widehat{\Delta}_\gog(x) = \begin{cases}
  (2i)^{N(N-1)/2}\prod_{1 \le i < j \le N} \sin\big(\frac{x_i - x_j}{2}\big), & \Phi = A_{N-1}, \\
  (2i)^{N(N-1)}\prod_{1 \le i < j \le N} \sin\big(\frac{x_i - x_j}{2}\big)\sin\big(\frac{x_i + x_j}{2}\big), & \Phi = D_{N}, \\
  (2i)^{N^2}\prod_{1 \le i < j \le N} \sin\big(\frac{x_i - x_j}{2}\big)\sin\big(\frac{x_i + x_j}{2}\big) \prod_{k=1}^N \sin\big(\frac{x_k}{2}\big), & \Phi = B_{N}, \\
  (2i)^{N^2} \prod_{1 \le i < j \le N} \sin\big(\frac{x_i - x_j}{2}\big)\sin\big(\frac{x_i + x_j}{2}\big) \prod_{k=1}^N \sin x_k, & \Phi = C_{N}.
\end{cases}
\end{equation}

\begin{equation} \label{eqn:rho-coords}
\rho = \begin{cases}
  \sum_{i=1}^N \frac{1}{2}(N-2i+1) \, e_i, & \Phi = A_{N-1}, \\
  \sum_{i=1}^N (N-i) \, e_i, & \Phi = D_{N}, \\
  \sum_{i=1}^N (N-i+\frac{1}{2}) \, e_i, & \Phi = B_{N}, \\
  \sum_{i=1}^N (N-i+1) \, e_i, & \Phi = C_{N}.
\end{cases}
\end{equation}

In order to treat all of the classical root systems simultaneously, we define, for $I \subset \{ A_{N-1}, \, B_N, \, C_N, \, D_N \}$,
\begin{equation} \label{eqn:roots-indicator}
    \epsilon_I = \begin{cases}
    1, & \Phi \in I, \\
    0, & \Phi \not \in I.
    \end{cases}
\end{equation}
We can then write, for example, (\ref{eqn:delta-cases}) more compactly as
\[
 \Delta_\gog(x) = 2^{N \epsilon_C} \prod_{1 \le i < j \le N} (x_i - x_j)(x_i + x_j)^{\epsilon_{BCD}} \prod_{k=1}^N x_k^{\epsilon_{BC}}.
\]

For $\lambda \in P^+$, the monomial $W$-invariant (exponential) polynomial associated to $\lambda$ is
\begin{equation} \label{eqn:mpoly-def}
  M_\lambda(y) = \frac{|W \cdot \lambda|}{|W|} \sum_{w \in W} e^{\langle w(\lambda), y \rangle}, \qquad y \in \tot.
\end{equation}
When $\Phi$ is the $A_{N-1}$ root system so that $W = S_N$ is the symmetric group, writing $x_j = e^{y_j}$, we find that
\begin{equation} \label{eqn:m-exp-to-poly}
  M_\lambda(y) = \frac{|S_N \cdot \lambda|}{N!} \sum_{\sigma \in S_N} \prod_{j=1}^N x_j^{\sigma(\lambda)_j}
\end{equation}
recovers the usual monomial symmetric polynomials\footnote{Since we have chosen to work in coordinates in which $\sum_j \lambda_j = 0$, the expression (\ref{eqn:m-exp-to-poly}) superficially appears to be a rational function of the $x_j$ variables. Note however that we also have $\sum_j y_j = 0$, or equivalently $\prod_j x_j = 1$, which we can use to clear negative powers of the $x_j$'s on the right-hand side of (\ref{eqn:m-exp-to-poly}) without changing its value.  Thus we find that, for type $A$ root systems, (\ref{eqn:mpoly-def}) can always be written as a polynomial in the $x_j$'s for any choice of $\lambda$.  For other root systems this is not generally possible, and (\ref{eqn:mpoly-def}) will indeed yield a rational function after the change of variables $x_j = e^{y_j}$.} in the variables $x_1, \hdots, x_N$.

For $\lambda \in P^+$, the irreducible representation $V_\lambda$ of $G$ with highest weight $\lambda$ has a weight space decomposition analogous to (\ref{eqn:weight-decomp-A}):
\begin{equation} \label{eqn:weight-decomp}
V_\lambda \cong \bigoplus_{\mu \in P^+} W_\mu^{\oplus \, \mult_\lambda(\mu)},
\end{equation}
where $W_\mu$ is the one-dimensional irreducible representation of $T$ with weight $\mu$.  Equivalently, the character $\chi_\lambda$ can be decomposed as
\begin{equation} \label{eqn:char-decomp}
  \chi_\lambda(e^y) = \sum_{\mu \in P^+} \mult_\lambda(\mu) \, M_\lambda(i y), \qquad y \in \tot.
\end{equation}
Rather than working with the character $\chi_\lambda : G \to \C$, it will sometimes be convenient to work instead with the function $\ch_\lambda$ on the complexification $\gog_\C = \gog \otimes_\R \C$, which is the holomorphic function on $\gog_\C$ defined by
\begin{equation} \label{eqn:ch-def}
\ch_\lambda(y) = \chi_\lambda(e^y), \qquad y \in \gog.
\end{equation}
Since $\ch_\lambda$ is holomorphic and invariant under the adjoint representation of $G$, its values on $\tot$ determine its values on all of $\gog_\C$.  Moreover, for $y \in \tot$ we have
\begin{equation} \label{eqn:ch-decomp}
  \ch_\lambda(-iy) = \sum_{\mu \in P^+} \mult_\lambda(\mu) \, M_\lambda(y),
\end{equation}
which is a manifestly real-valued exponential polynomial on $\tot$ rather than an oscillatory trigonometric polynomial.  This fact makes it easier to extract the asymptotics of the multiplicities $\mult_\lambda(\mu)$ from $\ch_\lambda(-iy)$ rather than $\chi_\lambda(e^y)$.

A well-known combinatorial formula for $\mult_\lambda(\mu)$, due to Kostant \cite{Kost}, implies that $\mult_\lambda(\mu) = 0$ unless
\begin{equation} \label{eqn:height-ineqs}
\langle \omega_i, \lambda \rangle \ge \langle \omega_i, \mu \rangle \qquad \forall \ i = 1, \hdots, r,
\end{equation}
where $\omega_1, \hdots, \omega_r$ are the fundamental weights defined in (\ref{eqn:fundweights-def}). The inequalities (\ref{eqn:height-ineqs}) are analogous to the fact that for Young diagrams $\lambda, \mu$ with $N$ rows, the Kostka number $K_{\lambda \mu}$ vanishes unless
\[
\sum_{i=1}^N \lambda_i = \sum_{i=1}^N \mu_i
\]
and
\[
\sum_{i=1}^k \lambda_i \ge \sum_{i=1}^k \mu_i
\]
for all $1 \le k < N$, where $\lambda_i$, $\mu_i$ are the numbers of boxes in the $i$th rows of $\lambda$ and $\mu$ respectively.  For later use, we record the explicit forms of the inequalities (\ref{eqn:height-ineqs}) for the other classical root systems $B_N$, $C_N$, and $D_N$.

For $B_N$, with the positive roots chosen in Section \ref{sec:root-systems}, the simple roots are $e_i - e_{i+1}$ for $1 \le i \le N-1$, and $e_N$. The fundamental weights are then
\begin{align} \label{eqn:B-fundweights}
\omega_k &= \sum_{i=1}^k e_i, \qquad 1 \le k \le N-1, \\
\nonumber \omega_N &= \frac{1}{2} \sum_{i=1}^N e_i.
\end{align}
Then (\ref{eqn:height-ineqs}) states that for $\lambda, \mu \in P^+$, $\mult_\lambda(\mu) = 0$ unless
\begin{equation} \label{eqn:BC-maj-ineqs}
\sum_{i=1}^k \lambda_i \ge \sum_{i=1}^k \mu_i, \qquad 1 \le k \le N.
\end{equation}
For $C_N$, with simple roots $e_i - e_{i+1}$, $1 \le i \le N-1$, and $2e_N$, the fundamental weights are
\begin{equation} \label{eqn:C-fundweights}
\omega_k = \sum_{i=1}^k e_i, \qquad 1 \le k \le N,
\end{equation}
which gives the same set of inequalities (\ref{eqn:BC-maj-ineqs}).  For $D_N$, with simple roots $e_i - e_{i+1}$, $1 \le i \le N-1$, and $e_N + e_{N-1}$, the fundamental weights are
\begin{align} \label{eqn:D-fundweights}
\omega_k &= \sum_{i=1}^k e_i, \qquad 1 \le k \le N-2, \\
\nonumber \omega_{N-1} &= \frac{1}{2} \sum_{i=1}^{N-1} e_i - \frac{1}{2} e_N, \\ 
\nonumber \omega_N &= \frac{1}{2} \sum_{i=1}^{N} e_i,
\end{align}
so that for $\lambda, \mu \in P^+$, $\mult_\lambda(\mu) = 0$ unless the inequalities (\ref{eqn:BC-maj-ineqs}) hold and additionally
\begin{equation} \label{eqn:D-maj-extra-ineq}
    \sum_{i=1}^{N-1} \lambda_i - \lambda_N \ge \sum_{i=1}^{N-1} \mu_i - \mu_N.
\end{equation}

To close this section, we show some lemmas on the asymptotic behavior of the monomial $W$-invariant polynomials $M_\lambda$ and the characters $\mathrm{ch}_\lambda$ as the rank of $G$ goes to infinity.  We will need these results below when we study the corresponding asymptotics of the multiplicities $\mult_\lambda(\mu)$.

In what follows, we assume that $\Phi$ is one of the root systems $B_N$, $C_N$ or $D_N$, as the analogous results for $A_{N-1}$ take a slightly different form and have already been shown in \cite{belinschi2022large}. Since we now consider the limit $N \to \infty$, below we take greater care to keep track of the index $N$, writing $\Phi_N$ in place of $\Phi$, $\tot_N$ in place of $\tot$, $\bmla_N \in \tot_N$ in place of $\lambda$, $\bmrho_N = \frac{1}{2}\sum_{\alpha \in \Phi_N^+} \alpha$ in place of $\rho$, and so on.

Given $\bmla_N = (\lambda_1, \hdots, \lambda_N) \in \tot_N \cong \R^N$, we write $\bmla_N' = \bmla_N + \bmrho_N$, and we define the empirical measure
\begin{equation} \label{eqn:m-def}
\mu[\bmla_N] = \frac{1}{N}\sum_{i=1}^N \delta_{\lambda_i},
\end{equation}
as well as the scaled, shifted empirical measure
\begin{equation}\label{e:defm}
m[\bmla_N]=\mu[(2N)^{-1}\bmla_N'].
\end{equation}
Given a probability measure $\mu$ on $\R$, let $T_\mu : (0, 1) \to \R$ be the right-continuous non-decreasing function such that $\mu$ is the pushforward of the uniform distribution on $(0, 1)$ by $T_\mu$. For $\delta > 0$, let $\bB_\delta(\mu)$ be the ball of radius $\delta$ centered at $\mu$, in the space of probability measures on $\R$ equipped with the Wasserstein distance.

The following lemma describes the asymptotic behavior of the monomial $W$-invariant polynomials as $N \to \infty$.
\begin{lem} \label{lem:monom-asymp}
For all $N = 2, 3, \hdots$, suppose that $\Phi_N$ is one of the root systems $B_N$, $C_N$ or $D_N$. Choose a sequence of dominant integral weights $\bmeta_N \in P^+_N$ and a sequence of points $\bmy_N \in \mathcal{C}^+_N$ such that, for $N$ sufficiently large, $\mu[\bmy_N] \in \bB_\delta(\nu)$ and $m[\bmeta_N] \in \bB_\delta(\mu)$ for some probability measures $\mu, \nu$. Then
\begin{equation}\label{eqn:monom-asymp}
\frac{1}{N^2}\log M_{\bmeta_N}(\bmy_N)
=\int_0^1 \big( 2 T_\mu(x)- x \big) T_{\nu}(x) \, \rd x+\oo_{\delta}(1)+\oo_N(1)
\end{equation}
as $N \to \infty$ and $\delta \to 0$.
\end{lem}
\begin{proof}
From the definition (\ref{eqn:mpoly-def}), we have
\[
  e^{\langle \bmy_N, \bmeta_N \rangle} \le M_{\bmeta_N}(\bmy_N) \le |W| e^{\langle \bmy_N, \bmeta_N \rangle},
\]
and for the classical root systems we have $N! \le |W| \le 2^N N!$.  Therefore
\[
  \frac{1}{N^2}\log M_{\bmeta_N}(\bmy_N) = \frac{1}{N^2} \langle \bmy_N, \bmeta_N \rangle + \oo_N(1).
\]
The following facts are then easily checked:
\begin{enumerate}
    \item \begin{align*}
  \frac{1}{N^2} \langle \bmy_N, \bmeta_N \rangle &= \int_0^1 T_{\mu[N^{-1}\bmeta_N]}(x) \, T_{\mu[\bmy_N]}(x) \, \rd x \\
  &= 2 \int_0^1 T_{\mu[(2N)^{-1}\bmeta_N]}(x) \, T_{\mu[\bmy_N]}(x) \, \rd x.
\end{align*}
\item If $\mu[\bmy_N]$ converges to $\nu$ in Wasserstein distance, then $T_{\mu[\bmy_N]} \to T_\nu$ almost everywhere with respect to Lebesgue measure on $(0,1)$.
\item If $m[\bmeta_N] = \mu[(2N)^{-1}\bmeta_N']$ converges to $\mu$ in Wasserstein distance, then by the explicit expressions (\ref{eqn:rho-coords}) for the coordinates of $\bmrho_N$, we find $T_{\mu[(2N)^{-1}\bmeta_N]}(x) \to T_\mu(x)-x/2$ almost everywhere with respect to Lebesgue measure on $(0,1)$.
\end{enumerate}
Together the above statements imply (\ref{eqn:monom-asymp}).
\end{proof}

Now we study the asymptotics of the characters.  To get the character of the irreducible representation with highest weight $\bmla_N$ from the generalized Bessel function, we take
\begin{align*}
 Y_N=(y_1,y_2,\cdots, y_N) \in \tot_N.
\end{align*}
Then it follows from \eqref{eqn:KCF} that
\begin{align}\begin{split}\label{e:character}
&\ch_{\bmla_N}(-iY_N)
=\frac{\Delta_\gog(Y_N)}{\widehat{\Delta}_\gog(-i Y_N)} \frac{\Delta_\gog(\bmla_N')}{\Delta_\gog(\rho)}J_{\vec 1,\bmla'_N}(Y_N),
\end{split}\end{align}
where $k = \vec 1$ is the constant multiplicity parameter with $k_\alpha = 1$ for all $\alpha \in \Phi_N$, and \eqref{e:Jasymp} implies the following asymptotics of $\ch_{\bmla_N}$:

\begin{lem} \label{lem:logchi}
For all $N = 2, 3, \hdots$, suppose that $\Phi_N$ is one of the root systems $B_N$, $C_N$ or $D_N$. Let $\bmla_N\in P^+_N$ be a sequence of deterministic dominant integral weights, such that the measures $m[\bmla_N]$ as defined in \eqref{e:defm} converge weakly to $m_{\bmla}$, and let $Y_N \in \tot_N$ be a sequence of points such that the empirical measures $\mu[Y_N]$ converge weakly to $\mu_{Y}$. We denote the symmetrizations of these limiting measures with respect to reflection through $0$ by $\widehat m_{\bmla}$ and $\widehat \mu_{Y}$ respectively.  Assume further that the measures $m[\bmla_N]$ and $\mu[Y_N]$ are supported in $[-\mathfrak{K}, \mathfrak{K}]$ for some $\mathfrak{K} > 0$, that
\[
  \epsilon_{BC} \int \log |x| \, \rd \mu[Y_N](x)
\]
is $\oo(N)$ as $N \to \infty$, and that there exists $C > 0$ such that
\[
\iint_{\Delta^c} \big( \log|x-y| + \log|x + y| \big) \, \rd m[\bmla_N](x) \, \rd m[\bmla_N](y) < C
\]
and
\[
\iint_{\Delta^c} \big( \log|x-y| + \log|x + y| \big) \, \rd \mu[Y_N](x) \, \rd \mu[Y_N](y) < C
\]
for all $N$, where $\Delta^c = \{(x,y) \in \R^2 : x \ne y \}$.  Then:
\begin{align} \nonumber
\lim_{N\rightarrow\infty}\frac{1}{N^2} \log \ch_{\bmla_N}(-iY_N)
= & \, J(\mu_Y, m_\bmla),\\
\label{e:logchi-proto} J( \mu_Y,m_\bmla) = & \, I_{0,2}(\widehat \mu_Y, \widehat m_\bmla)+ \iint \log|x-y|  \, \rd \widehat m_\bmla(x) \, \rd \widehat m_{\bmla}(y) \\
\nonumber & \quad - \iint  \log \frac{e^{x}-e^y}{x-y}  \, \rd \widehat \mu_Y(x) \, \rd\widehat \mu_Y(y) 
+\left(\frac{3}{2}-\log 2\right)
\end{align}
where $I_{0,2}$ is the functional defined in (\ref{e:arate}), with $\alpha = 0$ and $\beta = 2$.
\end{lem}

\begin{proof}
From (\ref{e:character}) and the homogeneity of $\Delta_\gog$, we have
\begin{align*}
    \frac{1}{N^2} \log \ch_{\bmla_N}(-iY_N) &= \frac{1}{N^2}\big [
    \log J_{\vec 1,\bmla'_N}(Y_N) + \log \Delta_\gog(\bmla_N') - \log \Delta_\gog(\rho) \\
    & \qquad \qquad + \log \Delta_\gog(Y_N) - \log \widehat{\Delta}_\gog(-i Y_N) \big ] \\
    &= \frac{1}{N^2}\big [
    \log J_{\vec 1,\bmla'_N}(Y_N) + \log \Delta_\gog((2N)^{-1}\bmla_N') + |\Phi^+| \log (2N) - \log \Delta_\gog(\rho) \\
    & \qquad \qquad + \log \Delta_\gog(Y_N) - \log \widehat{\Delta}_\gog(-i Y_N) \big ].
\end{align*}
Using the explicit formulae (\ref{eqn:delta-cases}) and (\ref{eqn:deltahat-cases}) for $\Delta_\gog$ and $\widehat{\Delta}_\gog$, and rewriting the functions of $\bmla_N$ and $Y_N$ as integrals with respect to the measures $m[\bmla_N]$ and $\mu[Y_N]$, we obtain
\begin{align*}
    \frac{1}{N^2} \log \ch_{\bmla_N}(-iY_N) & = \frac{1}{N^2} \log J_{\vec 1,\bmla'_N}(Y_N) + \bigg( \frac{2 \epsilon_C}{N} - \epsilon_{BC} - \epsilon_D \frac{N-1}{N} \bigg) \log 2 \\
    & \qquad + \frac{1}{2}\iint_{\Delta^c} \big( \log|x-y| + \log|x + y| \big) \, \rd m[\bmla_N](x) \, \rd m[\bmla_N](y) \\
    & \qquad + \frac{\epsilon_{BC}}{N} \int \log |x| \, \rd m[\bmla_N](x) \\
    & \qquad + \frac{1}{2}\iint_{\Delta^c} \big( \log|x-y| + \log|x + y| \big) \, \rd \mu[Y_N](x) \, \rd \mu[Y_N](y) \\
    & \qquad + \frac{\epsilon_{BC}}{N} \int \log |x| \, \rd \mu[Y_N](x) \\
    & \qquad - \frac{1}{2} \iint_{\Delta^c} \bigg[ \log \sinh \bigg( \frac{|x-y|}{2} \bigg) + \log \sinh \bigg( \frac{|x+y|}{2} \bigg) \bigg] \, \rd \mu[Y_N](x) \, \rd \mu[Y_N](y) \\
    & \qquad - \frac{1}{N} \int \big[ \epsilon_B \sinh(|x/2|) + \epsilon_C \sinh(|x|) \big ] \, \rd \mu[Y_N](x) \\
    & \qquad + \frac{|\Phi^+|}{N^2} \log (2N) - \frac{1}{N^2} \log \Delta_\gog(\bmrho_N).
\end{align*}

The limit of the first term above is given by Theorem \ref{main1}. With the assumed limiting behavior of $m[\bmla_N]$ and $\mu[Y_N]$, from (\ref{e:Jasymp}) and the symmetries (\ref{eqn:GBF-symm}), (\ref{eqn:GBF-scale}) of the generalized Bessel function we find
\begin{align*}
\lim_{N \to \infty} \frac{1}{N^2} \log J_{\vec 1,\bmla'_N}(Y_N) &= \lim_{N \to \infty} \frac{1}{N^2} \log J_{\vec 1,\sqrt{2N} (\bmla'_N / 2N)}(\sqrt{2N} \, Y_N) \\
&= I_{0,2}(\widehat \mu_Y, \widehat m_\bmla).
\end{align*}

Of the remaining terms, the ones involving integrals are easy to analyze. By assumption, the term
\[
\frac{\epsilon_{BC}}{N} \int \log |x| \, \rd \mu[Y_N](x)
\]
vanishes as $N \to \infty$. The terms
\[
  \frac{\epsilon_{BC}}{N} \int \log |x| \, \rd m[\bmla_N](x)
\]
and
\[
\frac{1}{N} \int \big[ \epsilon_B \sinh(|x/2|) + \epsilon_C \sinh(|x|) \big ] \, \rd \mu[Y_N](x)
\]
vanish as well: the first because $m_\bmla$ must have a density bounded above by 2 since each $\bmla_N$ is an \emph{integral} weight, and the second because the measures $\mu[Y_N]$ have uniformly bounded support.  Applying the convergence assumptions on $m[\bmla_N]$ and $\mu[Y_N]$, and using the fact that
\[
|\Phi^+| = N(N-1) + \epsilon_{BC} N,
\]
we thus find
\begin{align}
    \nonumber \frac{1}{N^2} \log \ch_{\bmla_N}(-iY_N) 
    & = I_{0,2}(\widehat \mu_Y, \widehat m_\bmla) \\
    \nonumber & \qquad + \frac{1}{2}\iint \big( \log|x-y| + \log|x + y| \big) \, \rd m_\bmla(x) \, \rd m_\bmla(y) \\
    \nonumber & \qquad + \frac{1}{2}\iint \big( \log|x-y| + \log|x + y| \big) \, \rd \mu_Y(x) \, \rd \mu_Y(y) \\
    \nonumber & \qquad - \frac{1}{2} \iint \bigg[ \log \sinh \bigg( \frac{|x-y|}{2} \bigg) + \log \sinh \bigg( \frac{|x+y|}{2} \bigg) \bigg] \, \rd \mu_Y(x) \, \rd \mu_Y(y) \\
    \nonumber & \qquad + \log N - \frac{1}{N^2} \log \Delta_\gog(\bmrho_N) + \oo(1) \\
    \label{eqn:logch-3} &= I_{0,2}(\widehat \mu_Y, \widehat m_\bmla)+ \iint \log|x-y|  \, \rd \widehat m_\bmla(x) \, \rd \widehat m_{\bmla}(y) \\
\nonumber & \qquad - \iint  \log \frac{e^{x}-e^y}{x-y}  \, \rd \widehat \mu_Y(x) \, \rd\widehat \mu_Y(y) \\
\nonumber & \qquad + \log N - \frac{1}{N^2} \log \Delta_\gog(\bmrho_N) + \oo(1)
\end{align}
as $N \to \infty$.

Finally, we study the term $N^{-2} \log \Delta_\gog(\bmrho_N)$.  From (\ref{eqn:delta-cases}) and (\ref{eqn:rho-coords}), we have
\begin{align*}
\Delta_\gog(\bmrho_N) &= \prod_{1 \le i < j \le N}(j-i)(2N - i -j + \epsilon_B + 2\epsilon_C) \cdot \prod_{m=1}^N \bigg( N-m + \frac{1 + \epsilon_C}{2} \bigg)^{\epsilon_{BC}} \\
&= \bigg[ \prod_{m=0}^{N-1} m! \, \bigg(m + \frac{1 + \epsilon_C}{2} \bigg)^{\epsilon_{BC}} \bigg] \cdot \prod_{0 \le i < j \le N-1} (i + j + \epsilon_B + 2\epsilon_C) \\
&= \bigg(\frac{1 + \epsilon_C}{2} \bigg)^{\epsilon_{BC}} \prod_{j=1}^{N-1} j! \, \bigg(j + \frac{1 + \epsilon_C}{2} \bigg)^{\epsilon_{BC}} \bigg( \frac{(2j - 1 + \epsilon_B + 2\epsilon_C)!}{(j - 1 + \epsilon_B + 2\epsilon_C)!} \bigg)
\end{align*}
which gives
\begin{align*}
\frac{1}{N^2} &\log \Delta_\gog(\bmrho_N) \\
&= \frac{1}{N^2} \sum_{j=1}^{N-1} \Big[ \log j! + \log (2j-1)! - \log (j-1)! \Big] + \oo(1) \\
&= \frac{1}{N^2} \sum_{j=1}^{N-1} \Big[ j (\log j - 1) + (2j-1) (\log (2j-1) - 1) - (j-1) (\log (j-1) -1) \Big] + \oo(1)
\end{align*}
as $N \to \infty$, where in the second equality we have used Stirling's approximation.  Interpreting the sum in the last line above as a Riemann sum approximating an integral, we then find
\begin{align*}
    \frac{1}{N^2} &\log \Delta_\gog(\bmrho_N) \\
    &= \frac{1}{N^2} \int_1^N \Big[ x \log x - x + (2x-1) \log(2x-1) - (x-1) \log(x-1) - x \Big ] \rd x + \oo(1) \\
    &= \frac{1}{2} \log N +  \log(2N-1) - \frac{1}{2} \log(N-1) - \frac{3}{2} + \oo(1) \\
    &= \log N + \log 2 - \frac{3}{2} + \oo(1).
\end{align*}
Plugging the above into (\ref{eqn:logch-3}) yields the desired formula (\ref{e:logchi-proto}).
\end{proof}

\section{Derivatives of the spherical integral}\label{s:derSI}

In this section we recall the spherical integral and its derivatives from \cite{belinschi2022large}. This will be crucial to analyzing the critical points of our large deviations rate functions, which are expressed as suprema of functions depending on spherical integrals.

We recall the following results on the asymptotics of spherical integrals over the unitary group.  As explained in Example \ref{ex:iz-int}, these integrals are given by generalized Bessel functions of type $A$ with $k \equiv 1$, so Theorem \ref{t:intconv} below can in fact be deduced from Theorem \ref{main1-A}.

\begin{thm}\label{t:intconv}
Let $A_{N},B_{N}$ be two sequences of deterministic self-adjoint matrices, such that their spectral measures $\mu_{A_N}$ and $\mu_{B_N}$ converge in Wasserstein distance towards $\mu_{A}$ and $\mu_{B}$ respectively. We assume that  there exists a constant $\fK>0$, such that $ \mu_{A_N}(|x|)\leq \fK$ and $\supp(\mu_{B_N}) \subset [-\fK, \fK]$. 
 Then the following limit exists:
\begin{align}\label{e:L1limit}
\lim_{N\rightarrow \infty}\frac{1}{ N^{2}}\log\int e^{ N \Tr(A_NUB_NU^{*})}\rd U= \frac{1}{2} I_{0,2}(\mu_{A},\mu_{B}),
\end{align}
where $\rd U$ is the Haar probability measure on the unitary group, and the functional $I_{0,2}$ is defined as in (\ref{e:arate}): 
\begin{align}\begin{split} \label{e:Iexp}
\frac{1}{2}I_{0,2}(\mu_A,\mu_B)
=&-\frac{1}{2}\inf \int_0^1 \int_\bR \left(\frac{\pi^2}{3}\rho_t^3 +u_t^2 \rho_t \right)\rd x \, \rd t\\
&-\frac{1}{2}\left(\Sigma(\mu_A)+\Sigma(\mu_B)\right)+\frac{1}{2}\left(\int x^2 \, \rd \mu_A(x)+\int x^2 \, \rd \mu_B(x)\right)+\frac{3}{4}
\end{split} \end{align}
where the infimum is taken over all the pairs $(u_t,\rho_t)$ such that $\del_t \rho_t+\del_x(\rho_tu_t)=0$ in the sense of distributions, $\rho_t\geq 0$ almost surely with respect to Lebesgue measure, $\int \rho_t \rd x=1$, and with initial and terminal data for $\rho$ given by 
\begin{align*}
\lim_{t\rightarrow 0^+}\rho_t(x)\rd x=\mu_A,\quad \lim_{t\rightarrow 1^-}\rho_t(x)\rd x=\mu_B.
\end{align*}

\end{thm}

\begin{proposition}\label{p:spbound}
We assume that the probability measures $\nu, \mu$ satisfy $\nu(|x|)<\infty$ and $\supp (\mu) \subset [-\fK,\fK]$  for some constant $\fK>0$. Then for any small $\varepsilon>0$, there exists a constant $C(\varepsilon)>0$ such that
\begin{align}\label{e:spbound}
\int T_{\nu}T_{\mu}\rd x-\OO(\varepsilon)\nu(|x|)-C(\varepsilon)
\leq \frac{1}{2}I_{0,2}( \nu,\mu)\leq \int T_{\nu}T_{\mu}\rd x.
\end{align}
Here, one can take $\OO(\varepsilon)=\fK(3\varepsilon+\varepsilon^{2})$ and $C(\varepsilon)$ depending only on $\varepsilon$.
\end{proposition}

As a consequence, we deduce that if $L\#\nu$ is the pushforward of $\nu$ by the homothety of factor $L$, that is, $\int f(Lx)\rd\nu(x)=\int f(x)\rd L\#\nu(x)$, then
\begin{equation}
    \lim_{L\to\infty}\frac{1}{2L} I_{0,2}( L\#\nu,\mu)=\int T_{\nu}T_{\mu}\rd x\,.
\end{equation}

\begin{proposition}\label{p:spbound2}
We assume that the probability measures $\nu, \mu$ satisfy $\nu(|x|)\leq \fK$ and $\supp (\mu) \subset [-\fK,\fK]$  for some constant $\fK>0$. Then for any small $\varepsilon>0$, we have
\begin{align}\label{e:spbound2}
I_{0,2}( \nu,\mu)=I_{0,2}(\nu^\varepsilon, \mu)+\int_{|T_\nu|> 1/\varepsilon} T_{\nu}T_{\mu}\rd x
+C_\fK\oo_\varepsilon(1),
\end{align}
where $\nu^\varepsilon$ is the restriction of $\nu$ to the interval $|x|\leq 1/\varepsilon$, i.e.~$\nu^\varepsilon=\nu\bm1(|x|\leq 1/\varepsilon)+\delta_0\int_{|x|> 1/\varepsilon}\rd \nu$, and the implicit error term depends only on $\varepsilon$ and $\fK$.
\end{proposition}

\begin{thm}\label{t:convl1}
Let $A_{N},B_{N}$ be two sequences of deterministic self-adjoint matrices, such that their spectral measures $\mu_{A_N}$ and $\mu_{B_N}$ converge in Wasserstein distance towards $\mu_{A}$ and $\mu_{B}$ respectively. We assume that  there exists a constant $\fK>0$ such that $ \mu_{A_N}(|x|)\leq \fK$ and $\supp (\mu_{B_N})\subset [-\fK, \fK]$. We denote the integrand of the spherical integral by
\begin{align}\label{e:lawU}
    \mu_{A_N, B_N}(U)=\frac{1}{Z}e^{ N\Tr(A_NUB_NU^{*})}\rd U,
\end{align}
where $Z$ is a normalization constant that makes $\mu_{A_N, B_N}$ into a probability measure.

 The non-commutative law of 
$(A_N,UB_NU^*)$ under $\mu_{A_N,B_N}$ as in \eqref{e:lawU} weakly converges to a non-commutative probability distribution $\tau_{\mu_A,\mu_B}$ over two non-commutative variables $(\sfa, \sfb)$, that is,
 the following limit exists in probability:
\begin{align}\label{e:tauF}
\lim_{N\rightarrow \infty}\int \frac{1}{N}\Tr(F(A_N,UB_NU^{*})) \, \rd\mu_{A_N,B_N}(U)=:\tau_{\mu_A,\mu_B} (F(\mathsf a, \mathsf b))\,,
\end{align}
where $F(\mathsf a,\mathsf b)\in\mathcal F$  belongs to  the complex vector space of test functions $\mathcal F$ generated by non-commutative polynomials in the form 
\begin{align}\label{e:defcF}
\sfb^{n_0}\frac{1}{z_1-\mathsf a}\sfb^{n_1} \frac{1}{z_2-\mathsf a}\sfb^{n_2}\cdots \frac{1}{z_k-\mathsf a} \sfb^{n_k},
\end{align}
where $k$ is any positive integer and $\{z_j\}_{1\leq j\leq k}$ belong to $\mathbb C\backslash \mathbb R$, whereas $(n_i)_{0\leq i\leq k}$ are non-negative integers. (Here we use resolvents instead of polynomials because $\mathsf a$ has a priori only its first moment finite.)
\end{thm}
  
 As a consequence of Theorem \ref{t:convl1}, for any bounded Lipschitz function $f: \bR\mapsto \bR$, 
\begin{align}\begin{split}\label{e:der}
\lim_{N\rightarrow \infty} \int \frac{1}{N}\Tr(f(A_N)UB_NU^*) \, \rd \mu_{A_N, B_N}(U)
&= \tau_{\mu_A,\mu_B} (f(\mathsf a) \mathsf b)
=\tau_{\mu_A,\mu_B} (\tau_{\mu_A,\mu_B} (f(\mathsf a)\mathsf b|\mathsf a))\\
&=\tau_{\mu_A,\mu_B} (f(\mathsf a)\tau(\mathsf b|\mathsf a)).
\end{split}\end{align}

For any $N\times N$ diagonal matrix $A_N=\diag\{a_1, a_2, \cdots, a_N\}$, we identify it with a multiplicative operator $\widetilde T_{A_N}: [0,1)\mapsto \bR$: 
\begin{align*}
\widetilde T_{A_N}(x)=\sum_{i=1}^N a_i {\bm 1}_{[\frac{i-1}{N}, \frac{i}{N})}(x).
\end{align*}
From the definition, the empirical eigenvalue distribution $\mu_{A_N}=(1/N)\sum_i \delta_{a_i}$ of $A_N$ is the pushforward measure of the uniform measure on $(0,1)$ by $\widetilde T_{A_N}$. We rearrange $a_1, a_2, \cdots, a_N$ in increasing order: $a_{1^*}\leq a_{2^*}\leq \cdots \leq a_{N^*}$ and define the multiplicative operator
\begin{align*}
T_{A_N}(x)=\sum_{i=1}^N a_{i^*} {\bm 1}_{[\frac{i-1}{N}, \frac{i}{N})}(x).
\end{align*}
Then $T_{A_N}$ is a right continuous nondecreasing function. Moreover, if $F_{A_N}$ is the cumulative density of the empirical eigenvalue distribution $\mu_{A_N}$, then $T_{A_N}$ is the functional inverse of $F_{A_N}$.
More generally, for any measurable function $\widetilde T_A: (0,1) \to \bR$, we define the measure $\mu_A=(\widetilde T_A)_\# \unif(0,1)$ to be the pushforward of the uniform measure on $(0,1)$ by $\widetilde T_A$, $F_A$ the cumulative density of $\mu_A$ and $T_{A}$ the functional inverse of $F_{A}$, which is right continuous and non-decreasing.

A sequence of measures $\mu_{A_N}$ converges weakly to $\mu_A$ if and only if  $T_{A_N}$ converges to $T_A$ at all continuous points of $T_A$. Moreover $\mu_{A_N}$ converges in Wasserstein distance to $\mu_A$ if and only if $T_{A_N}$ converges to $T_A$ in $L^1$ norm.

In the rest of this section, we identify the measurable function $\widetilde T_A$ on $(0,1)$ with the measure $\mu_A=(\widetilde T_A)_\# \unif(0,1)$. In particular, we can interpret the functional $I_{0,2}$ as defined on functions: $I_{0,2}(\widetilde T_A, \widetilde T_B)=I_{0,2}(\mu_A,\mu_B)$.

Next we recall a result characterizing the derivatives of the spherical integral in terms of the non-commutative distribution $\tau_{\mu_A,\mu_B}$ from \cite[Proposition 2.14]{belinschi2022large}
\begin{proposition}\label{p:derI}
Given two probability measures $\mu_A, \mu_B$, such that $\mu_A(|x|)<\infty$ and $\mu_B$ is compactly supported, for any compactly supported and Lipschitz  real-valued function $f$, we have
\begin{align}\label{e:nondelta}
\left.\del_\varepsilon I_{0,2}(T_A+\varepsilon f(T_ A), T_B)\right|_{\varepsilon=0}=\int f(x) \, \tau_{\mu_A,\mu_B}(\mathsf b|\mathsf a)(x) \, \rd \mu_A (x),
\end{align}
where $\sfa,\sfb$ are constructed in Theorem \ref{t:convl1}.
If the measure $\mu_A$ has a delta mass at $a$, for any bounded measurable function $\widetilde T_C$ supported on $\{x :  T_A(x)=a\}$, we have
\begin{align}\label{e:delta}
\left.\del_\varepsilon I_{0,2}(T_A+\varepsilon \widetilde T_C, T_B)\right|_{\varepsilon=0}= \tau_{\mu_A,\mu_B}(\mathsf b|\mathsf a)(a)\int \widetilde T_C(x) \, \rd x.
\end{align}
\end{proposition}

\section{Large-$N$ asymptotics of weight multiplicities} \label{sec:weight-asymp}

In this section, we use the asymptotics of the generalized Bessel function to derive large deviations estimates for weight multiplicities of irreducible representations of compact (or, equivalently, complex) simple Lie algebras.\footnote{The results in this section imply analogous aymptotics for weight multiplicities of arbitrary compact Lie groups, since all irreducible representations of a compact Lie group are also irreducible representations of its Lie algebra, and every compact Lie algebra is a direct sum of simple algebras and possibly an abelian algebra.}  Recall that these multiplicities are coefficients $\mult_{\bmla_N}(\bmeta_N)$ that can be defined as in \eqref{eqn:ch-decomp}:
\begin{equation}\label{ch-decompcopy}
  \ch_{\bmla_N}(-iY_N) = \sum_{\bmeta_N \in P^+} \mult_{\bmla_N}(\bmeta_N) \, M_{\bmeta_N}(Y_N),
\end{equation}
where now we consider sequences $\bmla_N, \bmeta_N \in P^+_N$ and $Y_N \in \mathcal{C}^+_N$ as $N \to \infty$.  Since the multiplicities for type $A$ (i.e., Kostka numbers) have already been studied in \cite{belinschi2022large}, below we restrict our attention to the three cases $\Phi = B_N, C_N, D_N$, as this simplifies the notation.  We treat these three families of root systems simultaneously, using the notation (\ref{eqn:roots-indicator}) to distinguish between different quantities that arise in each case.  In all cases we identify $\tot \cong \R^N$. For an explicit description of how this identification can be made for each of the Lie algebras studied below, see \cite[\textsection4]{McS-expository}.

We write $\cP(\R)$ for the set of probability measures on $\mathbb R$, $\cP(\R_{\ge 0})$ for the set of probability measures on $[0, \infty)$, and $\cP^{\rm b}([a,b])\subset \cP(\R)$ for the set of probability measures supported on $[a,b]$ with density bounded by $2$. We recall that for $\mu_A \in \cP(\R_{\ge 0})$, we define the right-continuous non-decreasing function $T_A : (0,1) \to \R$  such that $\mu_A=(T_A)_\#(\rm{unif}(0,1))$, and the symmetrized version $\widehat T_A$, such that $\widehat \mu_A=(\widehat T_A)_\#(\rm{unif}(0,1))$. Explicitly, we have
\begin{align}\label{e:hatT}
    \widehat T_A(x)=\left\{
    \begin{array}{ll}
    T_A(2x-1), & x\in[1/2,1)\\
    -T_A(1-2x), & x\in(0,1/2).
    \end{array}
    \right.
\end{align}
We will also use the following integral relation:
\begin{align}\label{e:ThatT}
\int T_A(x) T_B(x) \rd x= \int \widehat T_A(x) \widehat T_B(x) \rd x,\quad \int xT_A(x)  \rd x= 2\int x\widehat T_A(x) \rd x.
\end{align}

We recall the limiting value from Lemma \ref{lem:logchi}:
\begin{align}\label{e:logchi}
\lim_{N\rightarrow\infty}\frac{1}{N^2} \log \ch_{\bmla_N}(-iY_N)=J(\mu_Y,m_{\bmla}),
\end{align}
where
\begin{align}\label{e:logchi2}
 J( \mu_Y,m_\bmla) = 
& \, I_{0,2}(\widehat \mu_Y, \widehat m_\bmla)+ \iint \log|x-y|  \, \rd \widehat m_\bmla(x) \, \rd \widehat m_{\bmla}(y) \\
\nonumber & \quad - \iint  \log \frac{e^{x}-e^y}{x-y}  \, \rd \widehat \mu_Y(x) \, \rd\widehat \mu_Y(y) 
+\left(\frac{3}{2}- \log 2\right),
\end{align}
and the functional $I_{0,2}$ is given explicitly by \eqref{e:arate}:
\begin{align}\begin{split}\label{e:Ivec}
I_{0,2}(\widehat\nu_{A},\widehat\nu_{B})= & -\inf_{\{\widehat\rho_t\}_{0<t<1} \atop \text{satisfies \eqref{e:bbterm}}}\left\{\int_0^1 \int u_s^2  \widehat\rho_s(x)\rd x \rd s+\frac{\pi^2}{3}\int_0^1\int \widehat \rho^3_s(x) \rd x  \rd s
\right\}\\
& + \widehat\nu_{A}(x^{2})+\widehat\nu_{B}(x^{2}) - \Sigma(\widehat\nu_A) - \Sigma(\widehat\nu_B) +\frac{3}{2}.
\end{split}\end{align}
Note that this is the same functional appearing in \eqref{e:Iexp}.

\begin{thm}\label{t:Kostka}
Let $\bmla_N\in P^+_N$ be a sequence of deterministic highest weights, such that the measures $m[\bmla_N]$, as defined in \eqref{e:defm}, converge weakly to $m_{\bmla}$. Assume there exists a constant $\fK>0$, such that ${\rm supp} (m[\bmla_N])\subset [0,\fK]$ for all $N\in\mathbb N$.
\begin{enumerate}
\item Let $\mu\in \cP^{\rm b}([0,\fK])$, $\nu\in \cP(\R_{\geq 0})$ and set
\begin{align}\label{e:defHK}
H_\mu(\nu) = \int (2T_\mu-x) T_\nu \, \rd x-J(\nu, m_\bmla)\,,
\end{align}
where the functional $J(\cdot, \cdot)$ has been defined in \eqref{e:logchi} and $T_\mu, T_\nu$ are as defined immediately after \eqref{eqn:m-def}.
The functional
 \begin{align} \label{e:defIK}
 \mathcal I(\mu) = \sup_{\nu\in\mathcal M}H_{\mu}(\nu)
 \end{align}
is lower semicontinuous on $\cP^{\rm b}([0, \fK])$ and achieves its minimum
\begin{align}
    -\iint\log|x-y| \, \rd \widehat m_\bmla(x) \, \rd \widehat m_\bmla(y) \, \rd x \, \rd y-\left(\frac{3}{2}- \log 2\right)
\end{align}
only at the uniform measure $\unif(0, 1)$. Moreover, $\mathcal I(\mu)=+\infty$ unless  $\mu$ satisfies the limiting Schur--Horn inequalities:
\begin{equation}\label{limSHK} \int_y^1 \big( T_{\mu}(x)-T_{m_\bmla}(x) \big) \, \rd x\leq 0\quad\text{for all }y\in (0,1)\,.\end{equation}

\item
The weight multiplicities $\mult_{\bmla_N}(\bmeta_N)$ in \eqref{eqn:char-decomp} satisfy, for any $\mu\in \cP^{\rm b}([0,\fK])$,
\begin{align}\label{e:Kostka}
\nonumber \lim_{\delta\rightarrow 0} & \limsup_{N\rightarrow \infty} \frac{1}{N^2}\log \sup_{m[\bmeta_N]\in \bB_\delta(\mu)} \mult_{\bmla_N}(\bmeta_N)\\
&=\lim_{\delta\rightarrow 0}\liminf_{N\rightarrow \infty}\frac{1}{N^2}\log \sup_{m[\bmeta_N]\in \bB_\delta(\mu)} \mult_{\bmla_N}(\bmeta_N) =-\cI(\mu),
\end{align}
where $ \bB_{\delta}(\mu)$ is the ball $\{\nu\in \cP^{\rm b}([0,\fK]): \rd(\nu,\mu)<\delta\}$, and $\rd(\cdot, \cdot)$ is the Wasserstein distance.
\end{enumerate}

\end{thm}

\begin{remark} \label{rem:BC-vs-D-rep}
    For the $B_N$ and $C_N$ root systems, we have
    \[
    \mathcal{C}^+_N = \{ x \in \R^N : x_1 \ge \hdots \ge x_N \ge 0 \},
    \]
    whereas for $D_N$, we have
    \[
    \mathcal{C}^+_N = \{ x \in \R^N : x_1 \ge \hdots \ge |x_N| \}.
    \]
    Accordingly, in the case of $D_N$, the vectors $Y_N$ and $\bmla_N$ in (\ref{e:logchi}) may not have strictly non-negative coordinates. However, since at most one coordinate can be negative, the limiting measures $\mu_Y$ and $m_\bmla$ must be supported on $[0, \infty)$. Thus the assumptions on the supports of the measures in Theorem \ref{t:Kostka} involve no loss of generality for type $D$ relative to types $B$ and $C$.
\end{remark}

\subsection{Large deviations upper bound}

From the defining relation \eqref{eqn:ch-decomp}, we have
\begin{align}\label{kostkaub}
\mult_{\bmla_N}(\bmeta_N) \leq \frac{\ch_{\bmla_N}(-iY_N)}{M_{\bmeta_N}(Y_N)},
\end{align}
where $Y_N \in \mathcal{C}^+_N$ with $\mu[Y_N]$ converging in Wasserstein distance towards $\mu_Y$. (For example, we can take $y_1\geq y_2\geq \cdots\geq  y_N\geq 0$ to be the $1/N$-quantiles of $\mu_Y$.)

For the monomial symmetric function, we encode the highest weight $\bmeta_N$ by its shifted and scaled empirical measure $m[\bmeta_N]$. Then if $m[\bmeta_N]\in \mathbb B_{\delta}(\mu)$ for some probability measure $\mu$ such that $T_{\mu}(x)\geq x/2$, so that $m[\bmeta_N]$ goes to $(T_\mu)_\#(\unif(0,1))$,
Lemma \ref{lem:monom-asymp} implies
\begin{align}\label{e:msf}
\frac{1}{N^2}\log M_{\bmeta_N}(Y_N)
=\int (2T_\mu-x)T_{\mu_Y}\rd x+\oo_{\delta}(1)+\oo_N(1),
\end{align}
since here we are considering only types $BCD$. We can rewrite the first term in terms of $\widehat T_{ \mu}$ and $\widehat T_{ \mu_Y}$ as 
\begin{align}
    \int (2T_\mu-x)T_{\mu_Y}\rd x=\int (2\widehat T_{\mu}-2x)\widehat T_{\mu_Y}\rd x.
\end{align}

The large deviations upper bound follows from combining the asymptotics of the characters \eqref{e:logchi}, the estimate on the weight multiplicities \eqref{kostkaub}, and the limiting expression for the monomial symmetric polynomials \eqref{e:msf}:
\begin{align}\begin{split}\label{e:LDPu3}
&\limsup_{\delta\rightarrow 0}\limsup_{N\rightarrow \infty}\frac{1}{N^2}\log \sup_{m[\bmeta_N]\in \bB_\delta(\mu)} \mult_{\bmla_N}(\bmeta_N) \leq  -H_{\mu}(\mu_Y),\\
&H_\mu(\mu_Y)=2\int (\widehat T_\mu-x)\widehat T_{\mu_Y}\rd x-J(\mu_Y,m_\bmla),
\end{split}\end{align}
where the functional $J(\cdot,\cdot)$ is defined in \eqref{e:logchi}.
Taking the  infimum over $\mu_Y\in \mathcal M$ on the right-hand side of \eqref{e:LDPu3}  finishes the proof of the large deviations upper bound.

Recall that the multiplicity $\mult_{\bmla_N}(\bmeta_N)$ vanishes unless the inequalities \eqref{eqn:height-ineqs} are satisfied:
\begin{align}\label{e:Kregion}
\langle \omega^N_i, \bmla_N \rangle \ge \langle \omega^N_i, \bmeta_N \rangle \qquad \forall \ i = 1, \hdots, N,
\end{align}
where the explicit expressions of the fundamental weights $\omega^N_i$ and the inequalities \eqref{eqn:height-ineqs} for root systems of type $BCD$ are given in \eqref{eqn:B-fundweights}--\eqref{eqn:D-maj-extra-ineq}.
We recall  from Theorem \ref{t:Kostka} that $\mathcal I(\mu)=\sup_{\nu\in \cM}H_\mu(\nu)$. 
It turns out that the rate function $\cal I(\mu)$ equals $+\infty$ outside the admissible region $\cA_{m_\bmla}$ described by the limit of \eqref{e:Kregion}:
 \begin{align}\label{e:ami1}
  \int_y^1 \big( T_\mu(x)-T_{m_\bmla}(x) \big) \, \rd x \leq 0 \qquad \forall \ y\in (0,1).
 \end{align}
The inequalities (\ref{e:ami1}) are an infinite-dimensional analogue of the Schur--Horn inequalities that define the weak majorization order on Young diagrams.

The following proposition records some properties of the functionals $H_\mu(\cdot)$ and $\mathcal I(\cdot)$ for use below.  Its proof is given in Appendix \ref{sec:proof-appendix}.

\begin{proposition}\label{p:rateIK}
Under the assumptions and notations of Theorem \ref{t:Kostka}, the functional $H_\mu(\cdot)$ and rate function $\mathcal I(\cdot)$  satisfy:
\begin{enumerate}
\item\label{i:1}
For $\mu$ satisfying \eqref{limSHK}, $H_\mu(\cdot)$ is upper semicontinuous in the weak topology on $\{\mu\in \cP(\R_{\geq 0}): \nu(|x|)\leq \fR\}$ for any $\fR>0$.
\item If there exists some $0<y<1$ such that 
\begin{align}\label{e:midineq}
\int_y^1 \big(T_\mu(x)-T_{m_\bmla}(x)\big) \, \rd x>0, 
\end{align}
then $\mathcal I(\mu)=+\infty$.

\item If there exists some small constant $\fc>0$ such that the following strong limiting inequality holds:
\begin{align}\begin{split}\label{e:domainK}
&\int_y^1 \big(T_{\mu}(x) -T_{m_{\bmla}}(x)\big) \, \rd x\leq
\left\{\begin{array}{ll}
 -\fc y,      &\text{ for } 1-\fc\leq y\leq 1,\\
 -\fc,         &\text{ for } 0\leq y\leq 1-\fc,
 \end{array}\right.
\end{split}\end{align}
then
$\mathcal I(\mu)=H_\mu(\nu^*)<\infty$ for some probability measure $\nu^{*}$ such that  $\nu^*(|x|)<\infty$. 

\item The rate function $\mathcal I(\cdot)$ is lower semicontinuous on $\cP^{\rm b}([0, \fK])$ and  achieves its minimal value 
\begin{align}
    -\iint\log|x-y| \, \rd \widehat m_\bmla(x) \, \rd \widehat m_\bmla(y) \, \rd x \, \rd y-\left(\frac{3}{2}-\log 2\right)
\end{align}
only at the uniform probability measure $\unif(0,1/2)$.

\item For any measure $\mu$ in the admissible set  $\mathcal A_{m_\bmla}$ as defined in \eqref{e:ami1}, there exists a sequence of measures $\mu^{\varepsilon}$ inside the region as given in \eqref{e:domainK}, converging to $\mu$ in the weak topology and satisfying 
$\lim_{\varepsilon\rightarrow 0}\mathcal I(\mu^{\varepsilon})=\mathcal I(\mu)$.
\end{enumerate}
\end{proposition}

\subsection{Large deviations lower bound}

In this section we prove the large deviations lower bound in Theorem \ref{t:Kostka}. It follows from combining the following Propositions \ref{p:unique3} and \ref{p:lowerbound3}, and noticing that the number of dominant integral weights in $P^+_N$ whose empirical measures belong to $\bB_\delta(\mu)$ is at most $\exp [\OO(N\log N)]$.

\begin{proposition}\label{p:unique3}
Under the assumptions of Theorem \ref{t:Kostka}, for any probability measure $\mu_Y\in \cP(\R_{\geq 0})$, there exists a unique $\mu\in \cP^{\rm b}([0,\fK])$ such that 
\begin{align}\label{e:choicemuY}
\mu_Y\in \underset{\nu\in \cP(\R_{\geq 0})}{\arg\sup} H_\mu(\nu), \quad H_\mu(\nu)=\int (2T_\mu-x) T_{\nu} \, \rd x-J(\nu, m_\bmla).
\end{align}
Moreover $T_\mu$ is uniquely determined by $T_Y$ via
\begin{align*}
\widehat T_\mu(x)=\frac{1}{2}\tau(\widehat {\mathsf m}_\bmla |\widehat{\mathsf y})\circ \widehat T_Y(x)+x-\int \left(\frac{1}{\widehat T_Y(x)-\widehat T_Y(y)}-\frac{e^{\widehat T_Y(x)}+e^{\widehat T_Y(y)}}{2(e^{\widehat T_Y(x)}-e^{\widehat T_Y(y)})}\right)\rd y-\frac{1}{2}.
\end{align*}
Here, $\tau(\widehat {\mathsf m}_\bmla |\widehat{\mathsf y})$ is the conditional expectation of $\widehat {\mathsf m}_\bmla$ knowing $\widehat{\mathsf y}$ under the non-commutative distribution $\tau$ uniquely associated to $(\widehat m_\bmla,\widehat \mu_Y)$ as in Proposition \ref{p:derI}.
\end{proposition}

\begin{proposition}\label{p:lowerbound3}
Under the assumptions of Theorem \ref{t:Kostka}, for any probability measure $\mu_Y\in \cP(\R_{\geq 0})$, let $\mu$ be the unique measure in $\cP^{\rm b}([0, \fK])$ such that 
\begin{align*}
\mu_Y\in \underset{\nu\in \cP(\R_{\geq 0})}{\arg\sup} H_\mu(\nu), \quad H_\mu(\nu)=\int  (2T_\mu-x)T_{\nu} \,\rd x-J(\nu, m_\bmla).
\end{align*}
Then we have
\begin{align}\label{e:LDP3}
\frac{1}{N^2}\log\sup_{\bmeta_N: m[\bmeta_N]\in \bB_\delta(\mu)} \mult_{\bmla_N}(\bmeta_N) \geq -\left(H_\mu(\mu_Y)+\oo_{\delta}(1)+\oo_N(1)\right).
\end{align}
\end{proposition}

\begin{proof}[Proof of Theorem \ref{t:Kostka}]
Item 1 of Theorem \ref{t:Kostka} follows from Proposition \ref{p:rateIK}.
For Item 2, the large deviations upper bound follows from \eqref{e:LDPu3}. If $\mu$ does not satisfy the limiting Schur--Horn inequalities \eqref{limSHK}, then both sides of \eqref{e:Kostka} are $-\infty$, and there is nothing to prove. In the following we first prove \eqref{e:Kostka} when $\mu$ satisfies the strong limiting Schur--Horn inequalities \eqref{e:domainK} with some $\fc>0$.
In this case, thanks to Item 3 in Proposition \ref{p:rateIK}, 
 there 
 exists a probability measure $\mu_Y$ such that
$\mathcal I(\mu)=H_\mu(\mu_Y)<\infty$ and $\mu_Y\in \cP(\R_{\geq 0})$.  
Then Propositions \ref{p:unique3} and \ref{p:lowerbound3} imply that $\mu$ is uniquely 
determined by $\mu_Y$ and the large deviations lower bound holds. This gives the full large deviations principle when  the  strong limiting Schur--Horn inequalities \eqref{e:domainK} hold. 
Next we extend it to the boundary case by a continuity argument. Thanks to Item 5 in Proposition \ref{p:rateIK},  for any measure $\mu$ inside the admissible set \eqref{e:ami1} but not satisfying \eqref{e:domainK}, there exists a sequence of measures $\mu^{\varepsilon}$ inside the region as given in \eqref{e:domainK}, converging to $\mu$ in the weak topology and with 
$\lim_{\varepsilon\rightarrow 0}\mathcal I(\mu^{\varepsilon})=\mathcal I(\mu)$.
Then for any $\delta>0$, there exists $\varepsilon>0$ sufficiently small that
\begin{align}\label{e:bbcaseK}
\nonumber \liminf_{N\rightarrow\infty}& \frac{1}{N^2}\log \sup_{m[\bmeta_N]\in \bB_\delta(\mu)} \mult_{\bmla_N}(\bmeta_N) \\
 & \geq \liminf_{N\rightarrow\infty}\frac{1}{N^2}\log \sup_{m[\bmeta_N]\in \bB_{\delta/2}(\mu^\varepsilon)} \mult_{\bmla_N}(\bmeta_N) 
=\mathcal I(\mu^{\varepsilon})+\oo_\delta(1).
\end{align}
The large deviations lower bound follows by first sending $\varepsilon$ and then $\delta$ to zero in \eqref{e:bbcaseK}.
This finishes the proof of Theorem \ref{t:Kostka}.
\end{proof}

The proofs of both Propositions \ref{p:unique3} and \ref{p:lowerbound3} rely on the following probability estimate.
\begin{proposition}\label{c:expbound3}
Under the assumptions of Theorem \ref{t:Kostka}, let $Y_N=(y_1, \hdots, y_N) \in \mathcal{C}^+_N$ be a sequence of points such that the empirical measures $\mu[Y_N]$ converge in Wasserstein distance towards $\mu_Y$. For any $\mu$ with support in $[0, \fK]$, if 
\begin{align}\label{e:asup3}
\sup_{\nu\in \cP(\R_{\geq 0})}\left\{\int (2T_\mu-x) T_{\nu} \, \rd x-J(\nu,m_\bmla)\right\}
>\int (2T_\mu-x) T_{\mu_Y}\rd x-J(\mu_Y,m_\bmla),
\end{align}
then there exist a small $\delta>0$ and a positive constant $c(\delta)>0$ such that
\begin{align}\label{e:expbound3}\begin{split}
\phantom{{}={}}\sum_{\bmeta_N: m[\bmeta_N]\in \bB_\delta(\mu)}\mult_{\bmla_N}(\bmeta_N) \, M_{\bmeta_N}(Y_N)
\leq e^{-c(\delta)N^2}\ch_{\bmla_N}(-iY_N).
\end{split}\end{align}
\end{proposition}

\begin{proof}[Proof of Proposition \ref{c:expbound3}]
Under Assumption \eqref{e:asup3}, for sufficiently small $\varepsilon>0$, there exists a measure $\nu\in \cP(\R_{\geq 0})$ such that
\begin{align}\label{e:epserrorK}
\int (2T_\mu-x) T_{\nu}\rd x-J(\nu,m_\bmla)
\geq\int (2T_\mu-x) T_{\mu_Y}\rd x-J(\mu_Y,m_\bmla)+\varepsilon.
\end{align}
We divide by $\ch_{\bmla_N}(-iY_N)$ on both sides of \eqref{e:expbound3} and use the estimates \eqref{e:msf} and \eqref{e:logchi} to obtain
\begin{align*}\begin{split}
&\phantom{{}={}}\frac{1}{\ch_{\bmla_N}(-iY_N) }\sum_{\bmeta_N: m[\bmeta_N]\in \bB_\delta(\mu)}\mult_{\bmla_N}(\bmeta_N) \, M_{\bmeta_N}(Y_N)\\
&=\exp\left\{-N^2\left(J(\mu_Y, m_\bmla)-\int  (2T_\mu-x)T_{\mu_Y}\rd x+\oo_{\delta}(1)+\oo_N(1)\right)\right\}\sum_{\bmeta_N: m[\bmeta_N]\in \bB_\delta(\mu)} \mult_{\bmla_N}(\bmeta_N)\\
&\leq  
\exp\left\{- N^2\left(J(\mu_Y, m_\bmla)-\int  (2T_\mu-x)T_{\mu_Y}\rd x-J(\nu, m_\bmla)+\int  (2T_\mu-x)T_{\nu} \, \rd x+\oo_{\delta}(1)+\oo_N(1)\right)\right\}\\
&\leq \exp\left\{- N^2(\varepsilon+\oo_{\delta}(1)+\oo_N(1))\right\},
\end{split}
\end{align*}
where in the first inequality we used the large deviations upper bound \eqref{e:LDPu3}, and in the last inequality we used \eqref{e:epserrorK}. The claim follows provided we take $\delta$ sufficiently small and $N$ large.
\end{proof}

\begin{proof}[Proof of Proposition \ref{p:unique3}]
We first prove the existence of such $\mu$ by contradiction. If there is no such $\mu$, that is, if for every measure $\mu$ supported on $[0, \fK]$ we have 
\begin{align*}
\mu_Y\not\in \underset{\nu\in \cP(\R_{\geq 0})}{\arg\sup}\left\{\int (2T_\mu-x) T_{\nu}\rd x-J(\nu, m_\bmla)\right\},
\end{align*}
then it follows from Proposition \ref{c:expbound3} that  there exists a small $\delta>0$ and a positive constant $c(\delta)>0$ such that
\begin{align*}\begin{split}
\phantom{{}={}}\sum_{\bmeta_N: m[\bmeta_N]\in \bB_\delta(\mu)}\mult_{\bmla_N}(\bmeta_N) \, M_{\bmeta_N}(Y_N)
\leq e^{-c(\delta)N^2}\ch_{\bmla_N}(-iY_N).
\end{split}\end{align*}
Since the space of probability measures supported on $[0, \fK]$ is compact, it has a finite open cover $ \bigcup \bB_{\delta_j}(\mu_j)$ with each $\delta_j < \delta$. We then obtain a contradiction, since for $N$ large enough,
\begin{align*}\begin{split}
\ch_{\bmla_N}(-iY_N)
&=\sum_j\sum_{\bmeta_N: m[\bmeta_N]\in \bB_{\delta_j}(\mu_j)}\mult_{\bmla_N}(\bmeta_N) \, M_{\bmeta_N}(Y_N) \\
&\leq\sum_j e^{-c(\delta_j)N^2}\ch_{\bmla_N}(-iY_N) <\ch_{\bmla_N}(-iY_N).
\end{split}\end{align*}

In the following we prove the uniqueness of the measure $\mu$ satisfying \eqref{e:choicemuY}. We note that if $\mu\in \cP^{\rm b}([0, \fK])$, then $2T_\mu(x)-x$ is monotonically increasing.
Since $\mu_Y$ is a maximizer in \eqref{e:choicemuY}, for any $\varepsilon>0$,
\begin{align*}\begin{split}
\phantom{{}={}}\int  T_{Y} (2T_\mu &-x)\rd x-J(T_Y, T_{m_\bmla})\\
&\geq\int  (T_Y+\varepsilon \widetilde T_C)_\#[\unif(0,1)] \, (2T_\mu-x) \, \rd x-J(T_Y+\varepsilon \widetilde T_C, T_{m_\bmla})\\
&\geq\int  (T_{Y}+\varepsilon \widetilde T_C) (2T_\mu-x) \, \rd x-J(T_Y+\varepsilon \widetilde T_C, T_{m_\bmla}).
\end{split}\end{align*}
Rearranging the above expression and sending $\varepsilon \to 0$, we have
\begin{align}\label{e:lowerboundK}
\left.\del_{\varepsilon}J(T_Y+\varepsilon \widetilde T_C, T_{m_\bmla})\right|_{\varepsilon=0}\geq \int \widetilde T_C (2T_\mu-x) \, \rd x.
\end{align}

We recall that $J$ as in \eqref{e:logchi2} is given in terms of $I_{0,2}$ and some explicit integrals. We compute the derivative of $J$,
\begin{align*}\begin{split}
\phantom{{}={}}\left.\del_{\varepsilon}J(T_Y+\varepsilon \widetilde T_C, T_{m_\bmla})\right|_{\varepsilon=0}
&=\left.\del_\varepsilon I_{0,2}(\widehat T_Y+\varepsilon \widehat T_C, \widehat T_{m_\bmla})\right|_{\varepsilon=0}\\
& \qquad +\int \left(\frac{1}{\widehat T_Y(x)-\widehat T_Y(y)}-\frac{e^{\widehat T_Y(x)}+e^{\widehat T_Y(y)}}{2(e^{\widehat T_Y(x)}-e^{\widehat T_Y(y)})}\right)(\widehat T_C(x)-\widehat T_C(y)) \, \rd x \, \rd y\\
&=\left.\del_\varepsilon I_{0,2}(\widehat T_Y+\varepsilon \widehat T_C, \widehat T_{m_\bmla})\right|_{\varepsilon=0}\\
& \qquad +2\int \left(\frac{1}{\widehat T_Y(x)-\widehat T_Y(y)}-\frac{e^{\widehat T_Y(x)}+e^{\widehat T_Y(y)}}{2(e^{\widehat T_Y(x)}-e^{\widehat T_Y(y)})}\right)\widehat T_C(x) \,\rd x \,\rd y\\
\end{split}\end{align*}
where we used that $\int \widehat T_C(x)\rd x=0$.
We will choose $\widehat T_C$ in either the case \eqref{e:nondelta} or \eqref{e:delta}. We notice that in both cases if we replace $\widehat T_C$ by $-\widehat T_C$, both sides of \eqref{e:lowerboundK} change sign. Therefore, we conclude that
\begin{align}\label{e:dJ}
\left.\del_{\varepsilon}J(T_Y+\varepsilon \widetilde T_C, T_{m_\bmla})\right|_{\varepsilon=0}=\int \widetilde T_C (2T_\mu-x) \, \rd x.
\end{align}
Now if we choose $\widetilde T_C$ in \eqref{e:delta}, i.e. $\widetilde T_C$ supported on $\{x : T_Y(x)=a\}$, then \eqref{e:dJ} gives
\begin{align}\begin{split}\label{e:Teq}
\phantom{{}={}}\del_{\varepsilon}I_{0,2}(\widehat T_Y&+\varepsilon \widehat T_C, \widehat T_{m_\bmla})\Big|_{\varepsilon=0}+2\int \left(\frac{1}{\widehat T_Y(x)-\widehat T_Y(y)}-\frac{e^{\widehat T_Y(x)}+e^{\widehat T_Y(y)}}{2(e^{\widehat T_Y(x)}-e^{\widehat T_Y(y)})}\right)\widehat T_C(x)\,\rd x \,\rd y\\
&=\int \widetilde T_C(x) \, \rd x \, \frac{\tau(\widehat {\mathsf m}_\bmla |\widehat{\mathsf y})(a)-\tau(\widehat {\mathsf m}_\bmla |\widehat{\mathsf y})(-a)}{2}\\
&\qquad +2\int \left(\frac{1}{\widehat T_Y(x)-\widehat T_Y(y)}-\frac{e^{\widehat T_Y(x)}+e^{\widehat T_Y(y)}}{2(e^{\widehat T_Y(x)}-e^{\widehat T_Y(y)})}\right)\widehat T_C(x) \, \rd x \, \rd y\\
&=\int \widetilde T_C (2T_\mu-x) \,\rd x=\int \widehat T_C (2\widehat T_\mu-2x) \,\rd x,
\end{split}\end{align}
where we used \eqref{e:ThatT} in the last equality.
By our construction $\widetilde T_C(x)$ is supported on $\{x : T_Y(x)=a\}$, and we can write
\begin{align}\label{e:Teq2}
    \int \widetilde T_C(x) \, \rd x \, \frac{\tau(\widehat {\mathsf m}_\bmla |\widehat{\mathsf y})(a)-\tau(\widehat {\mathsf m}_\bmla |\widehat{\mathsf y})(-a)}{2}=
    \int \widehat T_C(x) \, \tau(\widehat {\mathsf m}_\bmla |\widehat{\mathsf y})\circ \widehat T_Y(x) \, \rd x.
\end{align}
By comparing \eqref{e:Teq} and \eqref{e:Teq2}, we conclude that  on the intervals  where $\widehat T_Y$ is a constant, we have 
\begin{align}\label{e:consteq}
\widehat T_\mu(x)=\frac{1}{2}\tau(\widehat {\mathsf m}_\bmla |\widehat {\mathsf y})\circ \widehat T_Y(x)+x-\int \left(\frac{1}{\widehat T_Y(x)-\widehat T_Y(y)}-\frac{e^{\widehat T_Y(x)}+e^{\widehat T_Y(y)}}{2(e^{\widehat T_Y(x)}-e^{\widehat T_Y(y)})}\right)\rd y-\frac{1}{2},
\end{align}
where we subtract $1/2$ to make $\widehat T_\mu(1/2)=0$.
Next we take $\widetilde T_C=f(T_Y)$ as in \eqref{e:nondelta}, where $f$ is an odd function, and we find:
\begin{align}\begin{split}\label{e:inc}
&\phantom{{}={}}\left.\del_{\varepsilon}J(T_Y+\varepsilon f(T_Y), T_{m_\bmla})\right|_{\varepsilon=0}\\
&=\int f(x) \, \tau(\widehat {\mathsf m}_\bmla |\widehat {\mathsf y})(x) \, \rd \widehat \mu_Y+2\int \int\left(\frac{1}{\widehat T_Y(x)-\widehat T_Y(y)}-\frac{e^{\widehat T_Y(x)}+e^{\widehat T_Y(y)}}{2(e^{\widehat T_Y(x)}-e^{\widehat T_Y(y)})}\right)\rd y \, f(\widehat T_Y(x)) \, \rd x\\
&=\int f(\widehat T_Y) \, \tau(\widehat {\mathsf m}_\bmla |\widehat {\mathsf y})\circ \widehat T_Y(x) \, \rd x+2\int\int \left(\frac{1}{\widehat T_Y(x)-\widehat T_Y(y)}-\frac{e^{\widehat T_Y(x)}+e^{\widehat T_Y(y)}}{2(e^{\widehat T_Y(x)}-e^{\widehat T_Y(y)})}\right) \rd y \, f(\widehat T_Y(x)) \, \rd x\\
&=\int f( T_{Y}) ( 2T_\mu-x) \, \rd x=\int f( \widehat T_{Y}) (2\widehat T_\mu-2x) \, \rd x=\int f( \widehat T_{Y}) (2\widehat T_\mu-2x+1) \, \rd x.
\end{split}\end{align}
On the intervals where $T_Y$ is increasing,  \eqref{e:inc} implies that \begin{align}\label{e:inceq}
\widehat T_\mu(x)=\frac{1}{2}\tau(\widehat {\mathsf m}_\bmla |\widehat{\mathsf y})\circ \widehat T_Y(x)+x-\int \left(\frac{1}{\widehat T_Y(x)-\widehat T_Y(y)}-\frac{e^{\widehat T_Y(x)}+e^{\widehat T_Y(y)}}{2(e^{\widehat T_Y(x)}-e^{\widehat T_Y(y)})}\right)\rd y-\frac{1}{2}.
\end{align}
Therefore, we conclude from \eqref{e:consteq} and \eqref{e:inceq} that \eqref{e:inceq} holds almost surely on $(0,1)$, which uniquely determines $\mu$. This finishes the proof of Proposition \ref{p:unique3}.
\end{proof}

\begin{proof}[Proof of Proposition \ref{p:lowerbound3}]
Thanks to the uniqueness of $\mu$, we have that for any $\mu'\neq \mu$ in $\cP^{\rm b}([0, \fK])$,
\begin{align*}
\mu_Y\not\in \underset{\nu\in \cP(\R_{\geq 0})}{\arg \sup} \left\{\int (2T_{\mu'}-x) T_{\nu}\rd x-J(\nu, m_\bmla)\right\}.
\end{align*}
As a consequence the assumption in Proposition \ref{c:expbound3} holds, 
\begin{align*}
\sup_{\nu\in \cP(\R)}\left\{\int (2T_{\mu'}-x) T_{\nu}\rd x-J(\nu,m_\bmla)\right\}
>\int (2T_\mu-x) T_{\mu_Y}\rd x-J(\mu_Y,m_\bmla),
\end{align*}
and there exist a small $\delta>0$ and a constant $c(\delta)>0$ such that
\begin{align*}\begin{split}
\sum_{\bmeta_N: m[\bmeta_N]\in \bB_\delta(\mu')}\mult_{\bmla_N}(\bmeta_N) \, M_{\bmeta_N}(Y_N)
\leq e^{-c(\delta)N^2}\ch_{\bmla_N}(-iY_N).
\end{split}\end{align*}
The space of probability measures $\cP^{\rm b}([0,\fK])$ minus the open ball $\bB_\delta(\mu)$ is compact, so we get a finite open cover
$$\bigcup_j \bB_{\delta_j}(\mu_j) = \cP^{\rm b}([0,\fK]) \setminus \bB_\delta(\mu),$$
and
\begin{align}\begin{split}\label{e:openball4}
\phantom{{}={}}\sum_{\bmeta_N: m[\bmeta_N]\in \bB_\delta(\mu)}\mult_{\bmla_N}(\bmeta_N) \, M_{\bmeta_N}(Y_N) &\geq \ch_{\bmla_N}(-iY_N)-\sum_i\sum_{\bmeta_N: m[\bmeta_N]\in \bB_{\delta_i}(\mu_i)}\mult_{\bmla_N}(\bmeta_N) \, M_{\bmeta_N}(Y_N)
\\
&\geq \left(1-\sum_ie^{-c(\delta_i)N^2}\right)\ch_{\bmla_N}(-iY_N) \\
&= \left(1-\sum_ie^{-c(\delta_i)N^2}\right)\exp\{ N^2(J(\mu_Y, \mu_\bmla)+\oo_N(1))\}.
\end{split}\end{align}
The large deviations lower bound at $\mu$ follows from the estimate \eqref{e:openball4} and \eqref{e:msf}:
\begin{align*}
\begin{split}
\sum_{\bmeta_N: m[\bmeta_N]\in \bB_\delta(\mu)} &\mult_{\bmla_N}(\bmeta_N)\\
&=\exp\left\{-N^2\left(\int (2T_\mu-x)T_{\mu_Y}\rd x+\oo_{\delta}(1)+\oo_N(1)\right)\right\}\sum_{\bmeta_N: m[\bmeta_N]\in \bB_\delta(\mu)} \mult_{\bmla_N}(\bmeta_N) \, M_{\bmeta_N}(Y_N)\\
&\geq \exp\left\{-N^2\left(\int (2T_\mu-x)T_{\mu_Y}\rd x+\oo_{\delta}(1)+\oo_N(1)\right)\right\}\exp\{ N^2(J(\mu_Y, \mu_\bmla)+\oo_N(1))\}\\
&= \exp\left\{- N^2\left(\int (2T_\mu-x)T_{\mu_Y}\rd x-J(\mu_Y, m_\bmla)+\oo_{\delta}(1)+\oo_N(1)\right)\right\},
\end{split}
\end{align*}
which is the large deviations lower bound.
\end{proof}

\begin{remark}
    In \cite{belinschi2022large}, Belinschi, Guionnet and the first author used the asymptotics of spherical integrals to prove a large deviations upper bound for Littlewood--Richardson coefficients, which are the tensor product multiplicities for irreducible representations of $\U(N)$.  Using the same methods and the results in this paper on the aysymptotics of generalized Bessel functions, one could generalize the bound in \cite{belinschi2022large} to an analogous large deviations upper bound for the tensor product multiplicities of arbitrary compact Lie groups or complex semisimple Lie algebras.  However, the analysis is somewhat arduous and the corresponding lower bound has not been proven even for the case of Littlewood--Richardson coefficients, so we have not pursued this here.
\end{remark}

\appendix
\section{Proof of Proposition \ref{p:rateIK}}
\label{sec:proof-appendix}

\begin{claim}\label{c:HDbound}
If $\mu$ is a probability measure  supported on $[-\fK,\fK]$ which satisfies \eqref{e:ami1}, for any probability measure $\nu$ with $\nu(|x|)\leq \fK$,  it holds
\begin{align*}
H_\mu(\nu)
\leq H_\mu(\nu^\delta)+C_{\fK}\oo_\delta(1),
\end{align*}
where the implicit error $\oo_\delta(1)$ is independent of the measure $\nu$.
\end{claim}
\begin{proof}
We recall the definition of $H_\mu(\nu)$ from \eqref{e:defHK}:
\begin{align*}\begin{split}
H_\mu(\nu)
&=\int (2T_\mu(x)-x)T_\nu(x)\rd x-J(\nu,m_{\bmla})\\
&=\int T_\nu(x)(2T_\mu(x)-x)\rd x-\left(\int_{|T_\nu|>1/\delta} T_\nu(x) (2T_{m_{\bmla}}(x)-x)\rd x +J(\nu^\delta, m_\bmla)+C_\fK\oo_\delta(1)\right)\\
&=\int_{|T_\nu|\leq1/\delta} T_\nu(x)(2T_\mu(x)-x)\rd x-J(\nu^\delta,m_{\bmla})+\int_{|T_\nu|>1/\delta} T_\nu(x) (2T_\mu(x)-2T_{m_{\bmla}}(x))\rd x+C_\fK\oo_\delta(1)\\
&\leq \int_{|T_\nu|\leq1/\delta} T_\nu(x)(2T_\mu(x)-x)\rd x-J(\nu^\delta,m_{\bmla})+C_{\fK}\oo_\delta(1)\\
&=\int T_{\nu^\delta}(x)(2T_\mu(x)-x)\rd x-J(\nu^\delta,m_{\bmla})+C_\fK\oo_\delta(1)=H_\mu(\nu^\delta)+C_\fK\oo_\delta(1),
\end{split}\end{align*}
where we used Proposition \ref{p:spbound2} and the formula of $J(\nu,m_\bmla)$ from \eqref{e:logchi-proto} in the second line. In the fourth line, we used  \eqref{e:ami1} and the fact that $x\rightarrow T_\nu(x) 1_{|T_\nu(x)|>1/\delta } $ is increasing to show that the last term in the third line is non-positive.
Finally, we used that  $|T_\mu(x)|\leq \fK$ in the last line.
\end{proof}

\begin{proof}[Proof of Proposition \ref{p:rateIK}]
 For Item 1,
unfortunately, $H_\mu(\cdot)$ is not continuous in the weak topology; it is only continuous in the topology of the Wasserstein metric. In the following we show that $H_\mu(\cdot)$ is upper semicontinuous in the weak topology on $\{\nu\in \cP(\bR_{\geq 0}): \nu(|x|)\leq \fR\}$ for any $\fR>0$.
Given a probability measure $\nu$, we define  the truncated measure 
$\nu^\delta=\nu\bm1(|x|\leq \delta^{-1})+\delta_0\int_{|x|>\delta^{-1}}\rd \nu $.

Let $\{\nu_n\}_{n\geq 1}$ be 
 a sequence of probability measures supported on $[0,\infty)$, satisfying $\nu_n(|x|)\leq \fR$ and converging weakly to $\nu$.
 It is easy to see that $\nu^{\delta}$ converges to $\nu$ in Wasserstein metric as $\delta\rightarrow0$. As a consequence, we get
\begin{align}\label{e:truncate}
H_\mu(\nu)
=H_\mu(\nu^\delta)+\oo_{\delta,\nu}(1),
\end{align}
where for any fixed measure $\nu$ with $\nu(|x|)\leq \fR$, the error $\oo_{\delta,\nu}(1)$ goes to zero as $\delta$ goes to zero. 
Moreover, we have that $\nu_n^\delta$ converges to $\nu^\delta$ in Wasserstein distance. Thus we have
\begin{align}\label{e:limitdelta}
\limsup_{n\rightarrow\infty}H_\mu(\nu_n^\delta)= H_\mu(\nu^\delta).
\end{align}
It follows from combining \eqref{e:truncate}, Claim \ref{c:HDbound} and \eqref{e:limitdelta} that
\begin{align*}\begin{split}
\limsup_{n\rightarrow\infty}H_\mu(\nu_n)
&\leq \limsup_{n\rightarrow\infty}H_\mu(\nu^\delta_n)+C_{\fK\vee\fR}\oo_\delta(1)\\
&=H_\mu(\nu^\delta)+C_{\fK\vee\fR}\oo_\delta(1)
=H_\mu(\nu)+C_{\fK\vee\fR}\oo_\delta(1)+\oo_{\delta,\nu}(1).
\end{split}\end{align*}
By sending $\delta$ to $0$ in the above estimate, we get that
\begin{align*}
\limsup_{n\rightarrow\infty}H_\mu(\nu_n)
\leq H_\mu(\nu),
\end{align*}
and the upper semicontinuity of $H_\mu$ follows.

For Item 2, using \eqref{e:spbound} in the expression \eqref{e:logchi2} for $J( \mu_Y, m_\bmla)$, we obtain
\begin{align}\begin{split}\label{e:JLY}
J( \mu_{LY}, m_\bmla)
&=I_{0,2}(\widehat\mu_{LY}, \widehat m_\bmla)-2L\int_{x\geq y}x \,\rd \widehat \mu_{Y}(x) \,\rd\widehat \mu_{Y}(y)+\OO(\log L),\\
&=I_{0,2}(\widehat\mu_{LY}, \widehat m_\bmla)-2L\int_{\widehat T_{\mu_Y}(x)\geq \widehat T_{\mu_Y}(y)}\widehat T_{\mu_Y}(x) \,\rd x \,\rd y+\OO(\log L),\\
&=2L\int \widehat T_{\mu_Y}\widehat T_{m_\bmla}\rd x-2L\int x \, \widehat T_{\mu_{Y}}(x) \, \rd x+\OO(\varepsilon L)+C(\varepsilon)
\end{split}\end{align}
as $L\rightarrow \infty$. 
For the first term in \eqref{e:defHK}, we have
\begin{align}\label{e:ftt}
    L\int (2T_\mu-x)T_{\mu_Y}\rd x
    =2L\int (\widehat T_\mu-x)\widehat T_{\mu_Y}\rd x.
\end{align}
Taking the difference of \eqref{e:JLY} and \eqref{e:ftt}, we find that as $L\rightarrow \infty$,
\begin{align}\begin{split}\label{e:Hnuasymp}
    H_{\mu}(\mu_{LY})
    &=\int (2T_\mu-x)T_{\mu_{LY}}\rd x-J( \mu_{LY}, m_\bmla)\\
    &=2L\int (\widehat T_{\mu}-\widehat T_{m_\bmla})\widehat T_{\mu_{Y}} \rd x +\OO(\varepsilon L)+C(\varepsilon)\\
    &=2L\int ( T_{\mu}- T_{m_\bmla}) T_{\mu_{Y}} \rd x +\OO(\varepsilon L)+C(\varepsilon)
    ,
\end{split}\end{align}
where we used \eqref{e:ThatT} in the last line. If \eqref{e:midineq} holds for some $0<y<1$, we can take $$\mu_Y=y\delta_0+(1-y)\delta_{1/(1-y)}.$$ Then $T_{\mu_Y}={\bf 1}_{[y,1]}/(1-y)$, and 
\begin{align*}
\mathcal I(\mu)\geq \lim_{L\rightarrow \infty}H_\mu(\mu_{LY})
= \lim_{L\rightarrow \infty} \frac{2L}{(1-y)}\int_y^1 (T_\mu-T_{m_{\bmla}}) \, \rd x=+\infty.
\end{align*}

For Item 3, we will first show that  \eqref{e:domainK}
implies that there exists a small $\delta>0$ and a large $L^*>0$ such that for any $\mu_Y\in \cP(\bR_{\geq 0})$ with $\int |x| \, \rd \widehat\mu_Y= 1$,  and for any $L\geq L^*$, we have $ H_{\mu}(\mu_{LY})\leq -\delta L$. Moreover, note that the set of measures $\mu_Y$ such that $T_{\mu_Y}$ is differentiable is dense in $\cP(\bR_{\geq 0})$. Hence,  by continuity of $ H_\mu$, we may assume that $\mu_Y$ 
is such that $T_{\mu_Y}$ is differentiable. Given  such a  $\mu_Y$, integration by parts yields

$$\int_0^1 (T_\mu-T_{m_{\bmla}})T_{ \mu_Y}\rd x=\int_0^1 (\widehat T_\mu-\widehat  T_{m_{\bmla}})\widehat  T_{ \mu_Y}\rd x=\int_0^1 \widehat  T_{\mu_Y}'(y) \int_{y}^{1} (\widehat  T_\mu(x)-\widehat  T_{m_{\bmla}}(x)) \, \rd x \, \rd y\,.$$
We deduce from \eqref{e:domainK} that 
\begin{align}\begin{split}
&\int_y^1 \left(\widehat T_{\mu}(x) -\widehat T_{m_{\bmla}}(x)\right)\rd x\leq
\left\{\begin{array}{ll}
 -\fc y/2,      &\text{ for } 1-\fc/2\leq y\leq 1,\\
 -\fc/2,         &\text{ for } \fc/2\leq y\leq 1-\fc/2,\\
 -\fc(1-y)/2,         &\text{ for } 0\leq y\leq \fc/2.
 \end{array}\right.
\end{split}\end{align}
Since $\widehat T_{\mu_Y}$ is non-decreasing, $\widehat T_{\mu_Y}'$ is non-negative, and we have
\begin{equation}\label{jh}\int_0^1 (\widehat T_\mu-\widehat T_{m_{\bmla}})\widehat T_{ \mu_Y}\rd x
\leq-\frac{\fc}{2}\int_0^1 \big(  y\bm 1_{[0,\fc]}(y) + \bm 1_{[\fc,1-\fc]}(y)+(1-y)\bm 1_{[1-\fc,1]}(y) \big) \widehat T_{\mu_Y}'(y) \, \rd y.
\end{equation}
Because $\widehat T_{\mu_Y}$ is symmetric around $x=1/2$,  we know that 
$\widehat T_{\mu_Y}(1/2)=0$ and  $\widehat T_{\mu_Y}(y)\leq 0$ for $0\leq y\leq 1/2$, 
$\widehat T_{\mu_Y}(y)\geq 0$ for $1/2\leq y\leq 1$.
Then we have $\int_{0}^{1/2}\widehat T_{\mu_{Y}}(y)\rd y=-1/2$ and $\int_{1/2}^{1}T_{\mu_{Y}}(y)\rd y=1/2$.
By an integration by parts, we conclude
\begin{align}\label{e:cc0}
\int_{1/2}^{1}T_{\mu_{Y}}(y)\rd y=\int_{1/2}^1 (1-y) T_{\mu_Y}'(y)\rd   y =\frac{1}{2}.
\end{align}
For $\fc\leq 1/2$, we have for $y\in (0,1)$,
\begin{align}\label{e:cc1}
y\bm 1_{[0,\fc]}(y) + \bm 1_{[\fc,1-\fc]}(y)+(1-y)\bm 1_{[1-\fc,1]}(y)\geq   (1-y) 1_{[1/2,1]}(y).
\end{align}
Plugging \eqref{e:cc0} and \eqref{e:cc1}  into \eqref{jh}, we get the following upper bound:
\begin{equation}\label{e:yconstraint}
\int_0^1 (T_\mu-T_{m_{\bmla}})(x) \, T_{ \mu_Y}(x) \, \rd x
\leq -\frac{\fc}{2}.
\end{equation}
We use \eqref{e:Hnuasymp} and \eqref{e:yconstraint} to estimate $H_\mu(\mu_{LY})$. As a consequence, if we take $\varepsilon$ much smaller than $\fc/4$, \eqref{e:Hnuasymp} implies that there exists a small $\delta>0$ and a large $L^*>0$ (depending only on $\fc$) such that for any $L\geq L^*$, we have $ H_{\mu}(\mu_{LY})\leq -\delta L$.
We conclude that 
\begin{align*}
\sup_{\nu\in \cP(\bR_{\geq 0})} H_\mu(\nu)=\sup_{\mu_Y\in \cP(\bR_{\geq 0}): \int |x|\rd \mu_Y\leq L^*}H_\mu(\mu_Y)<\infty,
\end{align*}
and the supremum is achieved at some $\nu^*\in \cP(\bR_{\geq 0})$ with $\int |x|\rd \nu^*\leq L^*$, since $\big\{\mu_Y\in \cP(\bR_{\geq 0}) : \int |x| \,\rd \mu_Y\leq L^* \big\}$ is compact and $H_\mu$ is upper semicontinuous.

For Item 4, since $(\mu, \nu)\mapsto H_\mu(\nu)$ is continuous in $\mu$, $\cI(\mu)=\sup_{\nu\in \cM}H_\mu(\nu)$ is lower semicontinuous. Moreover \begin{align}
    \cI(\mu)\geq H_\mu(\delta_0)
=-\iint\log|x-y| \, \rd \widehat m_\bmla(x) \, \rd \widehat m_\bmla(y) \, \rd x \, \rd y-\left(\frac{3}{2}-\log 2\right)=:\fC(\bmla).
\end{align}

If $\mathcal I(\mu)=\fC(\bmla)$, then  for all probability measures $\nu\in \cP(\bR_{\geq 0})$,
\begin{equation}\label{po}
\int (2\widehat T_{\mu}-2x)\widehat T_{\nu}\rd x\leq J(\nu,m_{\bmla})+\fC(\bmla).
\end{equation}
We denote by $\widehat \nu_\varepsilon=\varepsilon_\#\widehat \nu$  the pushforward of $\widehat \nu$ by the homothety of factor $\varepsilon$, so that $\widehat T_{\nu_\varepsilon}=\varepsilon \widehat T_\nu$. We write $\nu_\varepsilon$ for the unsymmetrized version of $\widehat \nu_{\varepsilon}$. Then \cite[Theorem 0.1]{CGM} implies that for $\varepsilon>0$ small enough,
$$J(\nu_{\varepsilon},m_{\bmla})=
 I_{0,2}(\widehat \nu_\varepsilon, \widehat m_\bmla)-\fC(\bmla)+\OO(\varepsilon^2)
=
2\varepsilon\int x \, \rd \widehat m_{\bmla}\int x \, \rd\widehat \nu -\fC(\bmla)+O(\varepsilon^{2})=-\fC(\bmla)+O(\varepsilon^{2})\,.$$
Hence, we deduce from \eqref{po} by replacing $\nu$ by $\nu_{\varepsilon}$  and sending $\varepsilon$  to zero that 
\begin{align}
\label{e:bTnu}
\varepsilon \int 
(2\widehat T_{\mu}(x)-2x)\widehat T_{\nu}(x)\rd x\leq  O(\varepsilon^{2}),
\end{align}
or equivalently for any probability measure $\nu\in \cP(\bR_{\geq 0})$,
$$\int \left(2T_{\mu}(x)-x\right)T_{\nu}(x)\rd x\leq 0.$$
Since $\mu\in  \cP^{\rm b}([0,\fK])$, we have $T_\mu(x)\geq x/2$, thus \eqref{e:bTnu} implies that $T_\mu(x)=x/2$, so that $\mu$ must be the uniform measure on $(0,1/2)$.

{ Finally, for Item 5, we pick $\mu\in \mathcal A_{m_{\bmla}}$ and construct $\mu^\varepsilon$ satisfying \eqref{e:ami1}  
converging to $\mu$ when $\varepsilon$ goes to zero.
If $\mu$ 
 is not a delta mass, we have for small  enough $\varepsilon>0$, $\widehat T_\mu(1-\varepsilon)>\widehat T_\mu(\varepsilon)+2\varepsilon. $
We  take $$\widehat T_{ \mu^\varepsilon}(y)=\left\{
\begin{array}{ll}
\widehat T_\mu(y)+\varepsilon &\mbox{ for }y\in (0,\varepsilon], \cr
\widehat T_\mu(\varepsilon)+\varepsilon &\mbox{ for } y\in   [\varepsilon,\varepsilon_{1}],\cr
\widehat T_{\mu}(y)  &\mbox{ for }y\in [\varepsilon_{1},\varepsilon_{2}],\cr
T(1-\varepsilon)-\varepsilon &\mbox{ for } y\in [\varepsilon_{2},1-\varepsilon],\cr
\widehat T_{\mu}(y)-\varepsilon &\mbox{ for }y\in[1-\varepsilon,1),\cr
\end{array}
\right.$$
where $$\varepsilon_{1}=\sup\{ x>\varepsilon : \widehat T_\mu(\varepsilon)+\varepsilon\geq \widehat T_\mu(x)\},\quad \varepsilon_{2}=\sup\{ x: \widehat T_{\mu}(x)\leq  \widehat T_{\mu}(1-\varepsilon)-\varepsilon\}\,.$$

We claim that $\mu^{\varepsilon}$ satisfies \eqref{e:ami1}. For all $y\in (0,1)$, we define
\begin{equation}\label{sd}
\varphi(y)=\int_y^1 \big(\widehat T_\mu(x)-\widehat T_{\mu^\varepsilon}(x) \big) \, \rd x\leq \int_y^1 \big(\widehat T_{m_{\bmla}}(x)-\widehat T_{\mu^\varepsilon}(x) \big) \, \rd x\,.
\end{equation}
From the construction $\varphi(0)=\varphi(1)=0$, and $\varphi(y)$ first decreases then increases on $(0,1)$. Moreover, for $y\in[\varepsilon, 1-\varepsilon]$, 
\begin{align*}
\varphi(y)=\int_y^1 \big( T_\mu(x)-T_{\mu^\varepsilon}(x) \big) \, \rd x\geq \varepsilon^2.
\end{align*}
Therefore, $\mu^{\varepsilon}$ satisfies \eqref{e:domainK} with $\fc=\varepsilon^2>0$.
We finally prove that $\mathcal I(\mu^{\varepsilon})$ goes to $\mathcal I(\mu)$ as $\varepsilon$ goes to zero.
By lower semicontinuity of $\mathcal I$, we already know that \begin{align}
\label{e:lowb}\mathcal I(\mu)\leq  \liminf_{\varepsilon\rightarrow 0}\mathcal I(\mu^{\varepsilon})\,.
\end{align}
For the converse bound, note that for all $\nu\in \cP(\bR_{\geq 0})$, integration by parts and \eqref{sd} imply that
$$\int T_\nu(T_\mu-T_{\mu^\varepsilon})(x) \, \rd x =\int T_\nu'(y)\int_y^1 (T_\mu(x)-T_{\mu^\varepsilon}(x)) \, \rd x \, \rd y\geq 0,$$
which results in
$$H_{\mu}(\nu)=\int (2T_\mu(x)-x)T_\nu(x) \, \rd x-J(\nu,m_{\bmla})\geq \int (2T_\mu^{\varepsilon}(x)-x)T_{\nu}(x) \, \rd x-J(\nu,m_{\bmla})=H_{\mu^\varepsilon}(\nu)\,.$$
As a consequence, we have
$$\mathcal I(\mu)\geq \mathcal I(\mu^\varepsilon),$$
and therefore 
\begin{align}\label{e:upbodd}
\mathcal I(\mu)\geq \limsup_{\varepsilon\rightarrow 0}\mathcal I(\mu^{\varepsilon})\,.
\end{align}
The claim follows from combining \eqref{e:lowb} and \eqref{e:upbodd}.}
\end{proof}

\section*{Acknowledgements}
The work of J.H. is supported by the National Science Foundation under grant number DMS-2054835.
The work C.M. is supported by the National Science Foundation under grant number DMS-2103170 and by a grant from the Simons Foundation.  The authors would like to thank Nizar Demni for helpful comments.

\bibliography{refs}
\bibliographystyle{amsplain}

\end{document}